\documentclass[11pt]{article}

\usepackage{amsthm,amsmath,amssymb}
\usepackage{natbib}
\usepackage{multirow}
\usepackage[pdftex]{graphicx}
\usepackage{subfigure}
\usepackage{makecell}
\usepackage{booktabs}
\usepackage{array}
\usepackage{fullpage}
\usepackage{url}
\usepackage{algorithm}
\usepackage{algorithmic}
\usepackage{bm}
\usepackage{smile}
\usepackage{mathtools}
\usepackage{wrapfig}
\usepackage{lipsum}
\usepackage{mathrsfs}
\usepackage{dsfont}
\usepackage{titling}
\usepackage{epstopdf}

\usepackage{natbib}
\usepackage{multirow}

\usepackage[usenames,dvipsnames,svgnames,table]{xcolor}
\usepackage[colorlinks=true,
            linkcolor=blue,
            urlcolor=blue,
            citecolor=blue]{hyperref}

\newcommand*{\supp}{\mathrm{supp}}

\begin{document}

\title{\huge Sparse Confidence Sets for Normal Mean Models}


\author{
  Yang Ning%
  \thanks{\linespread{1}\selectfont Department of Statistics and Data Science, Cornell University, Ithaca, NY 14853; e-mail:~\href{mailto:yn265@cornell.edu}{yn265@cornell.edu}. Ning was supported in part by NSF CAREER Award DMS-1941945 and DMS-1854637.}
  \and Guang Cheng%
  \thanks{\linespread{1}\selectfont
    Corresponding Author. Department of Statistics, Purdue University, IN 47906; e-mail:~\href{mailto:chengg@purdue.edu}{chengg@purdue.edu}.
    Guang Cheng gratefully acknowledges DMS-1811812, DMS-1821183 and Adobe Data Science Award.}}

\maketitle

\vspace{-0.1in}

\begin{abstract}
In this paper, we propose a new framework to construct confidence sets for a $d$-dimensional unknown sparse parameter $\btheta$ under the normal mean model $\bX\sim N(\btheta,\sigma^2\Ib)$. A key feature of the proposed confidence set is its capability to account for the sparsity of $\btheta$, thus named as {\em sparse} confidence set. This is in sharp contrast with the classical methods, such as Bonferroni confidence intervals and other resampling based procedures, where the sparsity of $\btheta$ is often ignored. Specifically, we require the desired sparse confidence set to satisfy the following two conditions: (i) uniformly over the parameter space, the coverage probability for $\btheta$ is above a pre-specified level; (ii) there exists a random subset $S$ of $\{1,...,d\}$ such that $S$ guarantees the pre-specified true negative rate (TNR) for detecting nonzero $\theta_j$'s. To exploit the sparsity of $\btheta$, we define that the confidence interval for $\theta_j$ degenerates to a single point 0 for any $j\notin S$.
Under this new framework, we first consider whether there exist sparse confidence sets that satisfy the above two conditions. To address this question, we establish a non-asymptotic minimax lower bound for the non-coverage probability over a suitable class of sparse confidence sets. The lower bound deciphers the role of sparsity and minimum signal-to-noise ratio (SNR) in the construction of sparse confidence sets. Furthermore, under suitable conditions on the SNR, a two-stage procedure is proposed to construct a sparse confidence set. To evaluate the optimality, the proposed sparse confidence set is shown to attain a minimax lower bound of some properly defined risk function up to a constant factor. Finally, we develop an adaptive procedure to the unknown sparsity and SNR. Numerical studies are conducted to verify the theoretical results.
\end{abstract}


\section{Introduction}

Assume that we observe a $d$-dimensional random vector $\bX=(X_1,...,X_d)$ satisfying the following normal mean model, also known as Gaussian sequence model, 
$$
\bX\sim N(\btheta,\sigma^2\Ib),
$$ 
where $\btheta=(\theta_1,...,\theta_d)$ is a $d$-dimensional unknown parameter, $\Ib$ is an identity matrix and $\sigma^2$ is the common variance which is assumed to be known. The mathematical simplicity of normal mean models is often exploited to discover fundamental phenomena underlying more complicated statistical models. In particular, the normal mean model has attracted numerous interest in high-dimensional statistics. Among others, \cite{abramovich2006adapting} proposed an adaptive procedure for estimating sparse $\btheta$ which is asymptotically minimax for $\ell_r$ loss, while \cite{butucea2018variable} derived the minimax risk for the recovery of sparsity pattern under the Hamming loss. From a complementary perspective, the detection boundary for testing the null hypothesis $\btheta=0$ has been well studied by \cite{ingster2010detection,hall2010innovated,arias2011global}, among many others. However, the uncertainty quantification in terms of confidence sets for $\btheta$ is less explored, partly because one can easily construct the following $(1-\alpha)$ level confidence sets 
\begin{equation}\label{eqci_bon}
\big\{\btheta\in\RR^d: \max_{1\leq j\leq d}|X_j-\theta_j|\leq t_{\alpha}\sigma\big\},
\end{equation}
where the cutoff $t_{\alpha}$ can be determined by the Gaussianity of $\bX$ with Bonferroni (or Sidak) correction or resampling methods \citep{arlot2010some,chernozhukov2013gaussian}. With a slightly different goal, \cite{benjamini2005false} proposed to construct confidence intervals for some randomly selected components of $\btheta$, known as  selective confidence intervals; see also \cite{weinstein2013selection,fuentes2018confidence,benjamini2019confidence,zhao2020constructing} for some recent development.   

Recently, there is a growing interest in developing confidence intervals for sparse linear regression and other regression models, for instance,  \cite{van2013asymptotically,zhang2011confidence,belloni2014uniform,javanmard2013confidence,ning2017general,cai2017confidence}, a list that is far from exhaustive. The method is often termed as debiased or desparsifying approach in the literature. Their main idea is to remove the bias of the penalized estimator, e.g., Lasso, so that the resulting estimator of the unknown regression coefficients is regular and asymptotically linear. The confidence intervals for each component of the regression parameter are obtained by Gaussian approximation. 
Intuitively, the debiased estimator can be viewed as the random vector $\bX$ in the normal mean model after use of the central limit theorem and other asymptotic approximation. As a result, one can construct confidence sets for the whole vector of regression parameter in a similar way as (\ref{eqci_bon}) by using the resampling method; see \cite{zhang2017simultaneous}. While the aforementioned works provide confidence sets with the desired coverage probability in the asymptotic setting, the construction of confidence sets does not reflect the sparsity of the parameter. For example, it is unclear whether it is possible for confidence sets as (\ref{eqci_bon}) to incorporate the information on the sparsity of $\btheta$, and if possible how to deal with it in an optimal way. 


In this paper, we propose a new framework to construct sparse confidence sets for $\btheta$ under the normal mean model. We first consider the setting that the parameter $\btheta=(\theta_1,...,\theta_d)$ belongs to a one-sided sparse set in $\RR^d$, i.e., $\btheta\in\Theta^+(s,a)$, where 
\begin{equation}\label{eqspace1}
\Theta^+(s,a)=\{\btheta\in\RR^d: \|\btheta\|_0\leq s, \min_{j: \theta_j\neq 0} \theta_j\geq a\},
\end{equation}
for some $s, a>0$.  Given $\bX\sim N(\btheta,\sigma^2\Ib)$, a sparse confidence set $M(S, \bU, \bL)$ for $\btheta$ is defined in the following form:
\begin{equation}\label{eqci}
M(S, \bU, \bL)=\{\btheta\in\RR^d: \btheta_{S^c}=0\;\mbox{and}\;\theta_j\in[L_j, U_j]\;\textrm{for any}~ j\in S\},
\end{equation}
where $S:=S(\bX)$ is a random subset of $[d]=\{1,2,...,d\}$, $S^c$ denotes the complement of $S$, and $\bL=(L_1,...,L_d)$ and $\bU=(U_1,...,U_d)$ with $L_j:=L_j(\bX)$ and $U_j:=U_j(\bX)$ being the lower and upper confidence bounds for $\theta_j$. If $j$ belongs to $S$, $[L_j, U_j]$ is the confidence interval for $\theta_j$, otherwise the confidence interval degenerates to a single point $0$. The cardinality of the random set $S$ determines the ``sparsity" level of $M(S, \bU, \bL)$. Note that by setting $S=[d]$, $M(S, \bU, \bL)$ reduces to the classical confidence intervals, such as (\ref{eqci_bon}).  On the other hand, if the support set of $\btheta$ is known, one can take $S=\supp(\btheta)$ and $M(S, \bU, \bL)$ reduces to the so called oracle confidence intervals. By exploiting the sparsity of $\btheta$, the oracle confidence interval degenerates to $0$ for those $\theta_j$ not in the support, and therefore is an example of sparse confidence sets in (\ref{eqci}). Since the support set of $\btheta$ is unknown, in regression models, \cite{fan2001variable,fan2014strong,wang2013calibrating} proposed to construct asymptotically valid oracle confidence intervals for the nonzero parameters under the assumption that the support set can be recovered with probability tending to 1. 

Formally, we require that the desired sparse confidence set (\ref{eqci}) should satisfy the following two conditions.
\begin{itemize}
\item $M(S, \bU, \bL)$ has the desired coverage probability for $\btheta$ uniformly over $\Theta^+(s,a)$, that is for a given level $0< \alpha< 1$, 
\begin{equation}\label{eqcoverage}
\sup_{\btheta\in\Theta^+(s,a)}\PP_{\btheta}(\btheta\notin M(S, \bU, \bL))\leq \alpha.
\end{equation}
This is the typical requirement for the validity of the confidence set. 
\item  $M(S, \bU, \bL)$ is ``sparse." Formally, for a given level $0< \delta< 1$ and any $j\in[d]$,
\begin{equation}\label{eqsparsity}
\sup_{\btheta\in\Theta^+(s,a), \theta_j=0}\PP_{\btheta}(j\in S(\bX))\leq 1-\delta.
\end{equation}
This condition implies that if $\theta_j=0$, then $j$ does not belong to the random set $S$ with probability at least $\delta$. Recall that for any $j\notin S$, we have $L_j=U_j=0$ by the definition of (\ref{eqci}) and naturally our estimate of $\theta_j$ is 0. For this reason, $\delta$ corresponds to the true negative rate (TNR) of the random set $S$ for detecting nonzero $\theta_j$'s. Equivalently, $1-\delta$ is the false positive rate (FPR). Thus, a larger value of $\delta$ requires the confidence set to have less false positives. 
Finally, we note that $\delta$ controls the expected cardinality of $S=S(\bX)$, where we use $|S(\bX)|$ to denote the cardinality of the set $S(\bX)$. Specifically, by (\ref{eqsparsity}) we obtain 
\begin{align*}
\sup_{\btheta\in\Theta^+(s,a)}\EE_{\btheta}|S(\bX)|&= \sup_{\btheta\in\Theta^+(s,a)} \Big[\sum_{j: \theta_j\neq 0}\PP_{\btheta}(j\in S(\bX))+ \sum_{j: \theta_j= 0}\PP_{\btheta}(j\in S(\bX))\Big]\\
& \leq s+(d-s)(1-\delta),
\end{align*}
which implies at least $s(1-\delta)$ components of intervals $[\bL, \bU]$ degenerate to $0$. 
\end{itemize}

Conceptually, it may be more intuitive to directly pre-specify the size of $S$ when constructing $M(S, \bU, \bL)$ as opposed to requiring (\ref{eqsparsity}). However, an appropriate choice of $|S|$ depends on the unknown sparsity of $\btheta$ and is often difficult to specify in practice. Therefore, we take the current approach which requires (\ref{eqsparsity}) together with (\ref{eqcoverage}).

Under this novel framework, our goal is to construct $M(S, \bU, \bL)$ such that (\ref{eqcoverage}) and (\ref{eqsparsity}) hold. 
In view of the definition of the sparse confidence set (\ref{eqci}), it is easily seen that
if there exists some $j\in [d]$ such that $j\in \supp(\btheta)$ and $j\notin S$, then  $\theta_j$ would never be covered by $M(S, \bU, \bL)$. 
Similarly, if $|S|$ is too large (e.g., $S=[d]$), there may exist too many false positives such that (\ref{eqsparsity}) is violated.
Thus, the bottleneck is how to construct a set $S$ for which  $\supp(\btheta)\subseteq S$ holds with some desired probability and (\ref{eqsparsity}) is valid. We first study the existence of such set $S$. To this end, a non-asymptotic minimax lower bound for $\PP_{\btheta}(\supp(\btheta)\not\subseteq S)$ is established in Theorem \ref{themlower} over a suitable class of random sets $S$ satisfying (\ref{eqsparsity}). More precisely, the class of the random sets is defined in (\ref{eqF}). The lower bound details the conditions on the  sparsity and minimum signal-to-noise ratio (SNR) in the construction. 
To match the lower bound, we further show in Theorem \ref{themupper} that, under appropriate conditions on the SNR, a random set $\hat S_{\alpha'}$ obtained by a simple thresholding procedure contains $\supp(\btheta)$ with probability greater than $1-\alpha'$ and satisfies (\ref{eqsparsity}), where $\alpha'$ is a pre-specified tolerance level.  

Given the set $\hat S_{\alpha'}$, we proceed to construct the lower and upper confidence bounds $\bL$ and $\bU$. Since the parameter space $\Theta^+(s,a)$ in (\ref{eqspace1}) is one-sided, we focus on the one-sided sparse confidence set with $U_j=+\infty$ for $j\in S$. 
In Section \ref{sec_construc}, we derive the lower confidence bound $\hat L_j$ for those $j\in \hat S_{\alpha'}$ using Bonferroni correction to account for the multiple comparisons and the randomness of the estimated set $\hat S_{\alpha'}$. In Theorem \ref{thm_spraseCI}, we show that the sparse confidence set constructed above satisfies the desired conditions (\ref{eqcoverage}) and (\ref{eqsparsity}). We note that our two-stage procedure for constructing sparse confidence set is similar to that for selective confidence intervals \citep{benjamini2005false}. We refer to Remark \ref{sec_compare} for further discussion.

Theorems \ref{themlower}, \ref{themupper} and \ref{thm_spraseCI} together characterize the role of the minimum SNR, defined as $a/\sigma$, in the construction of sparse confidence sets. In particular, in the asymptotic regime $d,s\rightarrow\infty$, a phase transition phenomenon occurs when the SNR reaches the level $ \Phi^{-1}(\delta)+\sqrt{2\log s}$, where $\Phi^{-1}(\cdot)$ is the inverse function of the Gaussian c.d.f. $\Phi(\cdot)$. To be specific, if $a/\sigma\leq \Phi^{-1}(\delta)+(1-\epsilon)\sqrt{2\log s}$ for an arbitrarily small positive constant $\epsilon$, it is impossible to construct sparse confidence sets. On the other hand, if $a/\sigma\geq \Phi^{-1}(\delta)+\sqrt{2\log s}$, the proposed sparse confidence set satisfies the conditions (\ref{eqcoverage}) and (\ref{eqsparsity}).


When the conditions on the SNR are fulfilled, there often exist infinite number of sparse confidence sets of form (\ref{eqci}) that meet (\ref{eqcoverage}) and (\ref{eqsparsity}). In Section \ref{sec_opt}, we further evaluate the optimality of the proposed sparse confidence set. For the one-sided interval $M(S, \bU, \bL)$, we formally define the following optimality criterion function 
\begin{equation}\label{eqR}
R(M(S, \bU, \bL), \Theta^+(s,a)):=\sup_{1\leq j\leq d}\sup_{\btheta\in \Theta^+(s,a)}\EE_{\btheta}(\theta_j-L_j),
\end{equation} 
which represents the maximum distance between $\theta_j$ and $\EE_{\btheta}(L_j)$; see Section \ref{sec_opt} for further details. Intuitively, $L_j\leq \theta_j$ is expected in order for the one-sided confidence interval to cover the unknown parameter $\theta_j$. As a result, the smaller $\EE_{\btheta}(\theta_j-L_j)$ is, the more preferred the confidence interval is. However, the non-coverage probability of the confidence set $M(S, \bU, \bL)$ can be inflated, if we force $R(M(S, \bU, \bL), \Theta^+(s,a))$ to be too small. This trade-off is formalized in Theorem \ref{thm_minimax_onesided}. In particular,  we establish the non-asymptotic minimax lower bound for the non-coverage probability of $M(S, \bU, \bL)$ over the class of confidence sets that satisfy (\ref{eqsparsity}) and $R(M(S, \bU, \bL), \Theta^+(s,a))\leq m$ for some given $m$. Under the asymptotic regime $d,s\rightarrow\infty$, a direct implication of Theorem \ref{thm_minimax_onesided} is the minimax lower bound for $R(M(S, \bU, \bL), \Theta^+(s,a))$. This result is shown in Corollary \ref{cor_minimax_onesided2}. We further show that the asymptotic version of the proposed sparse confidence set, denoted by $\bar M_{\alpha'}$, attains the above minimax lower bound upto a constant factor $2$. Thus, the proposed sparse confidence set is optimal (upto a constant) with respect to $R(M(S, \bU, \bL), \Theta^+(s,a))$. 

While the proposed sparse confidence set $\bar M_{\alpha'}$ is optimal, the construction of $\bar M_{\alpha'}$ requires the knowledge of the unknown sparsity $s$ and the minimum signal strength $a$. In Section \ref{sec_adaptive}, we propose a sparse confidence set that is adaptive to the unknown $s$ and $a$. In Theorem \ref{themadap}, we show that, under the  asymptotic regime, the adaptive sparse confidence set attains the same minimax lower bound for $R(M(S, \bU, \bL), \Theta^+(s,a))$ upto a constant. 

Finally, in Section \ref{sec_twosided} we extend our methodology and theoretical results to two-sided sparse confidence intervals for $\btheta\in \Theta(s,a)$, where  
$$
\Theta(s,a)=\{\btheta\in\RR^d: \|\btheta\|_0\leq s, \min_{j: \theta_j\neq 0} |\theta_j|\geq a>0\}.
$$
Numerical studies are conducted in Section \ref{sec_numerical} to backup our methodology and theoretical results. The proofs are deferred to Section \ref{sec_proof}.

The following notations are used throughout the paper.  For any $a, b\in\RR$, denote $a\vee b=\max(a, b)$ and $a\wedge b=\min (a,b)$. Denote $(a)_+=a$ if $a>0$ and $0$ otherwise. For any sequences $a_n, b_n$, we write $a_n\sim b_n$ if $a_n/b_n\rightarrow 1$ as $n\rightarrow\infty$. 

\section{Sparse Confidence Sets for One-sided Parameter Space}\label{sec_onesided}
In this section, we consider how to construct sparse confidence sets $M(S, \bU, \bL)$ under the normal mean model $\bX\sim N(\btheta,\sigma^2\Ib)$, where $\btheta$ belongs to the space $\Theta^+(s,a)$ defined in (\ref{eqspace1}). In order to guarantee (\ref{eqcoverage}) and (\ref{eqsparsity}), the bottleneck is to identify the set $S$, if it is possible. In Section \ref{sec_index}, we consider how to construct the set $S$ as our first step. Once the set $S$ is available, we construct appropriate lower and upper confidence bounds $\bL$ and $\bU$ in Section \ref{sec_construc}. 



\subsection{Construction of the set $S$}\label{sec_index}

The first question concerns whether it is possible to construct an index set $S$ with the desired properties. Define
\begin{align}\label{eqF}
\cF(\delta)=\{S(\bX):  &~\PP_0(j\in S(\bX))\leq 1-\delta,\nonumber\\
&~~\textrm{and the event $\{j\in S(\bX)\}$ only depends on $X_j$ for any $j\in[d]$}\},
\end{align}
where $\PP_0$ denotes the probability of $X_j\sim N(0,\sigma^2)$ and $\delta$ is specified in (\ref{eqsparsity}). On top of (\ref{eqsparsity}), we require the following technical condition: whether $j$ is selected by $S(\bX)$ or not is independent of the data $X_i$ for $i\neq j$.  This additional restriction of $S(\bX)$ seems to be reasonable, as $X_i$ for $i\neq j$ is ancillary for $\theta_j$.

 The following theorem provides the non-asymptotic  minimax lower bound for $\PP_{\btheta}(\supp(\btheta)\not\subseteq \hat S )$ over $\hat S\in\cF(\delta)$ for any given $\delta$. 

\begin{theorem}[Minimax lower bound]\label{themlower}
For any $s\geq 1$ and $0< \delta<1$,  we have
\begin{equation}
\inf_{\hat S \in\cF(\delta)}\sup_{\btheta\in\Theta^+(s,a)}\PP_{\btheta}(\supp(\btheta)\not\subseteq \hat S )\geq 1-\frac{1}{(\Delta+1)^s},    \label{low}
\end{equation}
where $\Delta=\Phi(\Phi^{-1}(\delta)-a/\sigma)$. Furthermore, consider the asymptotic setting that $s,d\rightarrow\infty$.  Let $c_s$ denote an arbitrary sequence $c_s\rightarrow\infty$ and $c_s/s\rightarrow 0$. When the SNR satisfies
\begin{equation}\label{alower}
a/\sigma\le \kappa_*:=\Phi^{-1}(\delta)-\Phi^{-1}(c_s/s),
\end{equation} 
we have
\begin{equation}\label{low2}
\liminf_{d,s\rightarrow\infty}\inf_{\hat S \in\cF(\delta)}\sup_{\btheta\in\Theta^+(s,a)}\PP_{\btheta}(\supp(\btheta)\not\subseteq \hat S )=1. 
\end{equation} 
\end{theorem}

A few remarks are in order. First, we note that the non-asymptotic lower bound in (\ref{low}) depends on the true negative rate (TNR) $\delta$, the signal-to-noise ratio (SNR) $a/\sigma$ and the sparsity level $s$. Since we are only interested in whether the nonzero parameters in $\btheta$ are selected by $\hat S$ or not, the lower bound is free of the dimensionality $d$, which differs from the lower bounds for support recovery \citep{butucea2018variable}. Second, the role of SNR and TNR becomes more transparent in the asymptotic regime as both $d,s\rightarrow\infty$. In particular, the asymptotic lower bound in (\ref{low2}) implies that when the SNR is finite or diverges slowly enough $(a/\sigma\leq \kappa_*)$, it is impossible to construct sparse confidence sets that cover $\supp(\btheta)$ uniformly over the parameter space $\Theta^+(s,a)$. Third, we comment that $\kappa_*> 0$ if and only if $\delta> c_s/s$. Thus, the negative result (\ref{low2}) is meaningful only if the pre-specified TNR is greater than $c_s/s$. 



Recall that in view of the definition of the sparse confidence set, $\btheta\in M(S, \bU, \bL)$ implies $\supp(\btheta)\subseteq  S$. Thus, Theorem \ref{themlower} leads to the following simple corollary on the feasibility of sparse confidence sets. To avoid repetition, we only present the asymptotic result. 
\begin{corollary}\label{corlower1}
Under the asymptotic setting $s,d\rightarrow\infty$, if (\ref{alower}) holds, then for any sparse confidence set $M(\hat S, \bU, \bL)$ with $\hat S \in\cF(\delta)$ we always have
$$
\liminf_{d,s\rightarrow\infty}\sup_{\btheta\in\Theta^+(s,a)}\PP_{\btheta}(\btheta\notin M(\hat S, \bU, \bL))=1. 
$$
\end{corollary}
As a result, the two requirements (\ref{eqcoverage}) and (\ref{eqsparsity}) cannot hold simultaneously unless the SNR is above the threshold $\kappa_*$ defined in (\ref{alower}).

In the following, we construct an index set that satisfies the desired coverage probability under certain signal strength condition. The estimator $\hat S_{\alpha'}$ is defined as
\begin{equation}\label{eq1}
\hat S_{\alpha'}=\Big\{j\in [d]: \frac{X_j}{\sigma}\geq (\Phi^{-1}(\frac{\alpha'}{s})+\frac{a}{\sigma})\vee \Phi^{-1}(\delta)\Big\},
\end{equation}
where $\alpha'$ denotes the tolerance level for the non-coverage probability of the index set. The following theorem shows that $\hat S_{\alpha'}$ belongs to the set $\cF(\delta)$ and the non-coverage probability of $\hat S_{\alpha'}$ is no greater than $\alpha'$.

\begin{theorem}[Upper bound]\label{themupper}
For any $0<\alpha'<1$, it holds that $\hat S_{\alpha'} \in\cF(\delta)$. In addition, if 
\begin{equation}\label{astar}
a/\sigma\geq \kappa^*:= \Phi^{-1}(\delta)-\Phi^{-1}(\alpha'/s)
\end{equation}
holds, then
\begin{equation}\label{up2}
\sup_{\btheta\in\Theta^+(s,a)}\PP_{\btheta}( \supp(\btheta) \not\subseteq \hat S_{\alpha'} )\leq \alpha'.
\end{equation} 
\end{theorem}



\begin{remark}\label{rem1}
It is of interest to compare the two thresholds $\kappa_*$ in (\ref{alower}) and $\kappa^*$ in (\ref{astar}). Assume that $s\rightarrow\infty$ and $\alpha'$ is fixed. By the tail bound inequality for Gaussian random variables (e.g., Lemma \ref{lemtail}), we can show that $\kappa^*\sim \Phi^{-1}(\delta)+\sqrt{2\log s}$.  Similarly, we have $\kappa_*\sim \Phi^{-1}(\delta)+\sqrt{2\log (s/c_s)}$. Thus, Theorems \ref{themlower} and \ref{themupper} together imply a phase transition at the level $\Phi^{-1}(\delta)+\sqrt{2\log s}$, i.e.,
\begin{itemize}
\item if $a/\sigma\leq \Phi^{-1}(\delta)+(1-\epsilon)\sqrt{2\log s}$ for some small positive constant $\epsilon$, it is impossible to construct sparse confidence sets, i.e., (\ref{low2}) holds.
\item if $a/\sigma\geq \Phi^{-1}(\delta)+\sqrt{2\log s}$, $\hat S_{\alpha'}$ has the desired coverage probability,  i.e., (\ref{up2}) holds, which leads to a valid sparse confidence set as shown in the next subsection.
\end{itemize}
\end{remark}

\subsection{Construction of one-sided confidence sets}\label{sec_construc}

In this section, we are ready to construct the confidence set based on $\hat S_{\alpha'}$ in (\ref{eq1}). Recall that $\theta_j\geq 0$ in $\Theta^+(s,a)$. We are mainly interested in the one-sided confidence interval concerning the distance of the lower confidence bound to $0$. Specifically, for $j\in \hat S_{\alpha'}$, we would like to construct a one-sided confidence interval $[c_j,+\infty)$ with some $c_j\geq 0$. If $c_j$ is strictly greater than $0$ (i.e., $0$ is not contained in the confidence interval), we can conclude that $\theta_j$ is nonzero with the desired confidence level. 
Thus, we define the one-sided sparse confidence set as 
\begin{equation}\label{eqhatm}
\hat M_{\alpha'}=M(\hat S_{\alpha'},  \hat\bU,  \hat\bL), ~\textrm{where}~  \hat L_j=(X_j- \hat u_{\alpha'}\sigma)_+, ~\hat U_j=+\infty
\end{equation} 
for any $j\in \hat S_{\alpha'} $ and $\hat u_{\alpha'}$ is to be specified later to attain the desired coverage probability. To simplify the presentation, we treat $\alpha'$ as a given tuning parameter.


We partition the signal-to-noise ratio (SNR) region into low and high levels for constructing the sparse confidence set (\ref{eqhatm}): 
\begin{itemize}
\item Low SNR region: $R_L=\big\{\kappa: \kappa^*\leq \kappa< \kappa^*\vee \hat \kappa\big\}$,
\item High SNR region: $R_H=\big\{\kappa: \kappa\geq \kappa^*\vee \hat \kappa\big\}$,
\end{itemize}
where
\begin{equation}\label{ahat}
\hat \kappa=-\Phi^{-1}(\frac{\alpha-\alpha'}{d})-\Phi^{-1}(\frac{\alpha'}{s})
\end{equation}
and $\alpha$ is the desired level specified in (\ref{eqcoverage}). In both regions, we require the SNR to be no smaller than $\kappa^*$ in order to guarantee (\ref{up2}); see Remark~\ref{rem1}. Under the asymptotic regime $d,s\rightarrow\infty$, provided that $\alpha, \alpha'$ and $\delta$ are all taken to be constants, we have $\kappa^*< \hat \kappa$ and $R_L$ and $R_H$ reduce to $\{\kappa: \kappa^*\leq \kappa<\hat \kappa\}$ and $\{\kappa: \kappa\geq \hat \kappa\}$, respectively. However, if the pre-specified TNR is sufficiently close to $1$, i.e., $\delta> 1-(\alpha-\alpha')/d$, we have $\kappa^*>\hat\kappa$. In this case, $R_L$ becomes an empty set and $R_H=\{\kappa: \kappa\geq \kappa^*\}$.

The following theorem shows that with a suitable choice of $\hat u_{\alpha'}$ the sparse confidence set (\ref{eqhatm}) satisfies (\ref{eqcoverage}) and (\ref{eqsparsity}).

\begin{theorem}\label{thm_spraseCI}
For any $0<\alpha'<\alpha$, provided (\ref{astar}) holds, we have 
$$
\sup_{\btheta\in\Theta^+(s,a), \theta_j=0}\PP_{\btheta}(j\in \hat S_{\alpha'})\leq 1-\delta,~~\textrm{and}~~\sup_{\btheta\in\Theta^+(s,a)}\PP_{\btheta}(\btheta\notin \hat M_{\alpha'})\leq \alpha,
$$ 
where $\hat u_{\alpha'}$ in (\ref{eqhatm}) is given by
$$
\hat u_{\alpha'}=\left\{
\begin{array}{ll}
\Phi^{-1}\Big(1-\frac{\alpha-\alpha'}{d}\Big) &\textrm{if}~ a/\sigma\in R_L,\\
\Phi^{-1}\Big(1-\frac{\alpha-\alpha'-(d-s)(1-\eta^+)}{s}\Big)&\textrm{if}~ a/\sigma\in R_H,
\end{array}
\right.
$$
where $\eta^+=\Phi(\frac{a}{\sigma}+\Phi^{-1}(\frac{\alpha'}{s}))$.
\end{theorem}

Theorem \ref{thm_spraseCI} implies that, when the SNR belongs to the low SNR region assuming it exists,  the confidence interval for $\theta_j$ is either $0$ if $j\notin \hat S_{\alpha'} $ or $[(X_j-\sigma\Phi^{-1}(1-\frac{\alpha-\alpha'}{d}))_+, +\infty)$ if $j\in \hat S_{\alpha'}$. Note that the one-sided confidence interval for $\theta_j$ with Bonferroni correction (without accounting for sparsity) is given by 
\begin{equation}\label{eqBon}
\Big[(X_j-\sigma\Phi^{-1}(1-\frac{\alpha}{d}))_+, +\infty\Big) 
\end{equation}
for $1\leq j\leq d$. Thus, in the case of low SNR, our sparse confidence set for $\theta_j$ with $j\in \hat S_{\alpha'}$ can be viewed as the Bonferroni correction at a reduced level $\alpha-\alpha'$ in order to account for the randomness of the estimated set $\hat S_{\alpha'}$.

To better understand the choice of $\hat u_{\alpha'}$ in the high SNR region $R_H$, we focus on the following subset of $R_H$,  
\begin{equation}\label{eqspraseCI0}
\frac{a}{\sigma}\geq \Big(-\Phi^{-1}(\frac{\alpha-\alpha'}{(1+\epsilon)d})-\Phi^{-1}(\frac{\alpha'}{s})\Big)\vee \kappa^*,
\end{equation} 
where $\epsilon$ is an arbitrarily small positive constant.  In this case,  we can show that
$$
(d-s)(1-\eta^+)\leq \frac{\alpha-\alpha'}{1+\epsilon}\frac{d-s}{d}\leq \frac{\alpha-\alpha'}{1+\epsilon}.
$$
As a result, we have
\begin{equation}\label{eqspraseCI1}
\hat u_{\alpha'}\leq \Phi^{-1}\Big(1-\frac{(\alpha-\alpha')}{s(1+\epsilon)/\epsilon}\Big).
\end{equation}
Recall that the oracle confidence interval is defined as 
\begin{equation}\label{eqoracle}
\Big[(X_j-\sigma\Phi^{-1}(1-\frac{\alpha}{s}))_+, +\infty\Big)
\end{equation}
for $j\in  \supp(\btheta) $ and $0$ otherwise. Thus, when (\ref{eqspraseCI0}) holds, our sparse confidence set with (\ref{eqspraseCI1}) is in the similar spirit to the oracle interval with a multiplicity correction factor $s(1+\epsilon)/\epsilon$ at level $\alpha-\alpha'$. 



Theorem \ref{thm_spraseCI} and the above remarks demonstrate the non-asymptotic behavior of the sparse confidence set $\hat M_{\alpha'}$ in (\ref{eqhatm}). 
To investigate the optimality of the sparse confidence set in the next section, it is more convenient to study the asymptotic version of $\hat M_{\alpha'}$ as  $d,s\rightarrow\infty$. In the asymptotic regime, the pre-specified levels $\alpha$ and $\delta$ are treated as fixed. 

Define two cut-off points for the SNR as
$$
\kappa^{**}=\Phi^{-1}(\delta)+\sqrt {2\log \left(\frac{s}{C_{s,\alpha'}\alpha'}\right)},
$$
and 
$$
\bar \kappa=\sqrt{2\log \left(\frac{2(d-s)}{(\alpha-\alpha')C_{d-s,\alpha-\alpha'}}\right)}+\sqrt {2\log \left(\frac{s}{C_{s,\alpha'}\alpha'}\right)},
$$
where $C_{s,\alpha'}=2(\pi\log(s/\alpha'))^{1/2}$.  The cut-off points $\kappa^{**}$ and $\bar \kappa$ are the asymptotic versions of $\kappa^*$ in (\ref{astar}) and $\hat \kappa$ in (\ref{ahat}) respectively, by applying the tail bound inequality for Gaussian random variables (e.g., Lemma \ref{lemtail}). Note that $\bar \kappa$ diverges to infinity faster than $\kappa^*$ as $d,s\rightarrow\infty$. Thus, unlike the high SNR region $R_H$ in the non-asymptotic setting, there is no need to take the maximum of $\bar \kappa$ and $\kappa^*$. 

The asymptotic version of our sparse confidence set $\hat M_{\alpha'}$ is
\begin{equation}\label{eqhatmasy}
\bar M_{\alpha'}=M(\bar S_{\alpha'},  \bar\bU,  \bar\bL), ~\textrm{where}~  \bar L_j=(X_j- \bar u_{\alpha'}\sigma)_+, ~\bar U_j=+\infty
\end{equation} 
for $j\in \bar S_{\alpha'}$. Here, $\bar S_{\alpha'}$ and $\bar u_{\alpha'}$ are defined as follows. 
\begin{itemize}
\item When $\kappa^{**}\leq a/\sigma<\bar \kappa$, define $j\in \bar S_{\alpha'}$ if and only if $X_j/\sigma\geq  \Phi^{-1}(\delta)$, and 
\begin{equation}\label{eqhatu1}
\bar u_{\alpha'}=\sqrt{2\log \Big(\frac{d}{(\alpha-\alpha')C_{d,\alpha-\alpha'}}\Big)}.
\end{equation} 
\item When $a/\sigma\geq \bar \kappa$, define $j\in \bar S_{\alpha'}$ if and only if $X_j/\sigma\geq \sqrt{2\log (\frac{2(d-s)}{(\alpha-\alpha')C_{d-s,\alpha-\alpha'}})}$, and 
\begin{equation}\label{eqhatu2}
\bar u_{\alpha'}=\sqrt{2\log \Big(\frac{2s}{(\alpha-\alpha')C_{s,\alpha-\alpha'}}\Big)}.
\end{equation} 
\end{itemize}
The asymptotic properties of $\bar M_{\alpha'}$ are shown in the following corollary.

\begin{corollary}\label{cor_asym}
Assume that $d,s\rightarrow\infty$ and $\delta, \alpha$ are pre-specified fixed constants. For any $0<\alpha'<\alpha$, provided $\kappa^{**}\leq a/\sigma$, we have
$$
\limsup_{d,s\rightarrow\infty}\sup_{\btheta\in\Theta^+(s,a), \theta_j=0}\PP_{\btheta}(j\in \bar S_{\alpha'})\leq 1-\delta, ~~\limsup_{d,s\rightarrow\infty}\sup_{\btheta\in\Theta^+(s,a)}\PP_{\btheta}(\btheta\notin \bar M_{\alpha'})\leq \alpha.
$$
\end{corollary}

\begin{remark}[On the choice of $\alpha'$]\label{rem_alpha}
We note that in general $\kappa^{**}$, $\bar \kappa$ and $\bar u_{\alpha'}$ in (\ref{eqhatu1}) and (\ref{eqhatu2}) all depend on the choice of $\alpha'$. However, in the asymptotic regime, if we set $\alpha'=\gamma\alpha$ for any fixed constant $0<\gamma<1$, then $\bar u_{\alpha'}\sim\sqrt{2\log d}$  in (\ref{eqhatu1}) and $\bar u_{\alpha'}\sim\sqrt{2\log s}$ in (\ref{eqhatu2}), and similarly, $\kappa^{**}\sim\Phi^{-1}(\delta)+\sqrt{2\log s}$ and $\bar \kappa\sim \sqrt{2\log (d-s)}+\sqrt{2\log s}$, which are all asymptotically independent of $\alpha'$. From a theoretical perspective, when $d, s$ are large enough, the choice of $\alpha'$ has little effect on the proposed confidence interval. Therefore, in the asymptotic analysis, we treat $\alpha'$ as a fixed small constant. We refer to the numerical studies in Section \ref{sec_numerical} for sensitivity analysis and further practical guidelines on choosing $\alpha'$. 

\end{remark}

\section{Optimality of Sparse Confidence Sets}\label{sec_opt}


In this section, we will establish the optimality of the proposed sparse confidence sets with respect to the criterion function $R(M(S, \bU, \bL), \Theta^+(s,a))$ in (\ref{eqR}). We define a generic class of one-sided sparse confidence sets as
\begin{align*}
CI_+=\{M(S, \bU, \bL):  ~&\textrm{for any $j\in[d]$, $L_j, U_j$ only depend on $X_j$, }\\ &~~~~\textrm{$0\leq L_j\leq X_j\vee 0$ and $U_j=+\infty$ if $j\in S$, otherwise $L_j=U_j=0$}\}.
\end{align*}
For any confidence set $M(S, \bU, \bL)$ in $CI_+$, we first require that the construction of $(L_j, U_j)$ is separable for $1\leq j\leq d$, which is compatible with the condition in the definition of $\cF(\delta)$ in (\ref{eqF}). In addition, we require $L_j\leq X_j\vee 0$, a technical condition to control the tail of $L_j$. In order for the interval $[L_j,\infty)$ to cover $\theta_j$, one would expect that the lower confidence bound $L_j$ is smaller than $X_j$. Together with $L_j\geq 0$, this implies $0\leq L_j\leq X_j\vee 0$. 
It is easily seen that the one-sided Bonferrroni confidence set with $S=[d]$ and $L_j=(X_j-\sigma\Phi^{-1}(1-\frac{\alpha}{d}))_+$  belongs to $CI_+$.


Within the class of confidence sets $CI_+$, we further define $\cM_+(m,\delta)$ as a subset such that $R(M(S, \bU, \bL), \Theta^+(s,a))\leq m$ for some given $m>0$ and $S \in\cF(\delta)$ holds as defined in (\ref{eqF}). Formally, for any $m>0$ and $\delta$ in (\ref{eqF}), we have
\begin{align}\label{eq_Mplus}
\cM_+(m, \delta)=\Big\{M(S, \bU, \bL)\in CI_+:  R(M(S, \bU, \bL), \Theta^+(s,a))\leq m,~\textrm{and}~S \in\cF(\delta)\Big\},
\end{align}
where the quantity $m$ characterizes the maximum distance between $\theta_j$ and the expected value of the lower confidence bound $\EE(L_j)$ as shown in the definition of $R(M(S, \bU, \bL), \Theta^+(s,a))$. Intuitively, given any two confidence sets in $CI_+$ both with the desired coverage probability, we would favor the one with a smaller value of $R(M(S, \bU, \bL), \Theta^+(s,a))$, as it corresponds to a ``shorter" one-sided confidence interval and is more informative on the possible range of $\theta_j$. However, if we set $m$ to be too small, the non-coverage probability of any confidence sets in $\cM_+(m, \delta)$ may go beyond the desired level $\alpha$. 

In the following theorem, we demonstrate this trade-off by showing the non-asymptotic lower bound for the non-coverage probability of any confidence set in $\cM_+(m,\delta)$. 

\begin{theorem}[Minimax lower bound]\label{thm_minimax_onesided}
For any $s\geq 1$ and $M \in\cM_+(m,\delta)$, it holds that
\begin{equation}\label{eqthm_minimax_onesided}
\sup_{\btheta\in \Theta^+(s,a)} \PP_{\btheta}(\btheta\notin M)\geq \max\Big(\sup_{\rho\geq a,A\leq s} G(d,A,\rho,m), \sup_{\rho\geq 0, B\leq s} G(s,B,\rho,m), 1-\frac{1}{(\Delta+1)^s}\Big),
\end{equation}
where $\Delta$ is defined in Theorem \ref{themlower}, 
$$
G(d,A,\rho,m)=\frac{A[g(d,A,\rho)-(m+R)/\rho]_+}{1+A[g(d,A,\rho)-(m+R)/\rho]_+},
$$
 with 
$$
g(d,A,\rho)=\frac{d-A}{A}\Phi\Big(-\frac{\rho}{2\sigma}-\frac{\sigma}{\rho}\log (\frac{d}{A}-1)\Big)+\Phi\Big(-\frac{\rho}{2\sigma}+\frac{\sigma}{\rho}\log (\frac{d}{A}-1)\Big),
$$
and 
$$
R=\sqrt{\frac{1}{2\pi}}\sigma \exp\Big(-\frac{1}{2}(\frac{\rho}{\sigma})^2\Big)\frac{\sqrt{1+4(\sigma/\rho)^2}-1}{\sqrt{1+4(\sigma/\rho)^2}+1},
$$
and $G(s,B,\rho,m)$ is defined similarly.
\end{theorem}

We show that the non-asymptotic lower bound  (\ref{eqthm_minimax_onesided}) is the maximum of three terms. The first term $G(d,A,\rho,m)$ is derived by varying the parameters in the following way: the support set with cardinality $A$ is randomly selected among $d$ coordinates and the parameter on the support set is fixed at $\rho$.  Since the nonzero entries of the parameters in $\Theta(s,a)$ is no smaller than $a$ and the sparsity level is no greater than $s$, we require $\rho\geq a$ and $A\leq s$ for $G(d,A,\rho,m)$. When $a$ diverges slowly enough as $s,d\rightarrow\infty$, this term dominates and converges to $1$ for some suitable $m$, see case (2) of  the following Corollary \ref{cor_minimax_onesided}.  Similarly, the second term $ G(s,B,\rho,m)$ is derived by fixing the support set of the parameters and varying the values of the parameters on the support set. As seen in case (3) of Corollary \ref{cor_minimax_onesided}, this term converges to $1$ when $a$ is sufficiently large. The last term is inherited from Theorem \ref{themlower}, see also the discussion of Corollary \ref{corlower1}. 


To simplify the results in Theorem \ref{thm_minimax_onesided}, we consider the asymptotic regime in the following corollary.


\begin{corollary}\label{cor_minimax_onesided}
Assume that $s,d\rightarrow\infty$. 
\begin{itemize}
\item[(1).] If $a/\sigma\leq \kappa_*$ defined in (\ref{alower}), then
$$
\liminf_{d,s\rightarrow\infty}\inf_{M\in\cM_+(m,\delta)} \sup_{\btheta\in \Theta^+(s,a)} \PP_{\btheta}(\btheta\notin M)=1.
$$
\item[(2).] If $\kappa_*<a/\sigma\leq \sqrt{2\log (d/A_d-1)}$, then
$$
\liminf_{d,s\rightarrow\infty}\inf_{M\in\cM_+(m,\delta)} \sup_{\btheta\in \Theta^+(s,a)} \PP_{\btheta}(\btheta\notin M)=1,
$$
for $m\leq \sigma(\frac{1}{2}-\frac{W_d}{A_d})\sqrt{2\log (d/A_d-1)}+\frac{\sqrt{2}\sigma}{4\sqrt{\pi}}(1-\frac{A_d}{d-A_d})$, where $A_d, W_d$ are two arbitrary sequences satisfying 
\begin{equation}\label{eqcor_minimax_onesided1}
2W_d\leq A_d\leq s,~~\frac{d}{A_d}\rightarrow\infty,~~\textrm{and}~~W_d\rightarrow\infty.
\end{equation} 
\item[(3).] If $a/\sigma\geq \sqrt{2\log (d/A-1)}$ for some constant $0<A\leq s$, then
$$
\liminf_{d,s\rightarrow\infty}\inf_{M\in\cM_+(m,\delta)} \sup_{\btheta\in \Theta^+(s,a)} \PP_{\btheta}(\btheta\notin M)=1,
$$
for $m\leq \sigma(\frac{1}{2}-\frac{V_s}{B_s})\sqrt{2\log (s/B_s-1)}+\frac{\sqrt{2}\sigma}{4\sqrt{\pi}}(1-\frac{B_s}{s-B_s})$, where $B_s, V_s$ are two arbitrary sequences satisfying $2V_s\leq B_s< s$ and
${s/B_s}\rightarrow\infty, V_s\rightarrow\infty$.  
\end{itemize}
\end{corollary}

This corollary details the trade-off between the coverage probability of the confidence set and the magnitude of $R(M(S, \bU, \bL)$ in three regions depending on the value of $a/\sigma$. 
In particular, the case (1) is inherited from Corollary \ref{corlower1}. To understand the case (2), we can first pick a sequence $A_d$ that diverges to infinity sufficiently slow (e.g., slower than $s$), and then set $W_d=A_d^{1/2}$.  The condition (\ref{eqcor_minimax_onesided1}) holds. Thus, in case (2), $m$ cannot be smaller than $\sigma(\frac{1}{2}-o(1))\sqrt{2\log d}$ in order to guarantee the desired coverage probability. Similarly, when the minimum SNR $a/\sigma$ grows fast enough as in case (3), $m$  cannot be smaller than $\sigma(\frac{1}{2}-o(1))\sqrt{2\log s}$. 

Finally, we prove the optimality of the sparse confidence set $\bar M_{\alpha'}$ in Corollary \ref{cor_asym}. Consider the class of one-sided confidence sets for which the coverage probability is no smaller than $1-\alpha$ uniformly over $\Theta^+(s,a)$, defined as
\begin{align*}
\cM_+=\{M(S, \bU, \bL)\in CI_+:   \liminf_{d,s\rightarrow\infty}\inf_{\btheta\in \Theta^+(s,a)} \PP_{\btheta}(\btheta\in M(S, \bU, \bL))\geq 1-\alpha,~\textrm{and}~S \in\cF(\delta)\}.
\end{align*}
Recall that Corollary \ref{cor_asym} implies $\bar M_{\alpha'}\in\cM_+$. In the  following corollary, we establish the optimality of $\bar M_{\alpha'}$ within the class $\cM_+$ with respect to the criterion function $R(M, \Theta^+(s,a))$ defined in (\ref{eqR}). 


\begin{corollary}\label{cor_minimax_onesided2}
Assume that $d,s\rightarrow\infty$ and $\delta, \alpha$ are pre-specified fixed constants. 
\begin{itemize}
\item[(1).] If $\kappa^{**}\leq a/\sigma\leq \sqrt{2\log (d/A_d-1)}$ for some sequence $A_d\leq s$ satisfying $A_d\rightarrow\infty$ and $d/A_d \rightarrow\infty$, then
\begin{equation}\label{cor_minimax_onesided2_eq1}
\liminf_{d,s\rightarrow\infty}\inf_{M\in\cM_+} \frac{R(M, \Theta^+(s,a))}{\sigma\sqrt{2\log d}/2}\geq 1.
\end{equation}
Consider the sparse confidence set $\bar M_{\alpha'}$ in Corollary \ref{cor_asym} with $\alpha'=\gamma\alpha$ for any constant $0<\gamma<1$. Then $\bar M_{\alpha'}\in \cM_+$ and 
\begin{equation}\label{cor_minimax_onesided2_eq2}
\limsup_{d,s\rightarrow\infty}\frac{R(\bar M_{\alpha'}, \Theta^+(s,a))}{\sigma\sqrt{2\log d}}\leq 1.
\end{equation} 
\item[(2).] If $a/\sigma\geq\tilde\kappa$ where $\tilde\kappa=\sqrt{2\log (d-s)-\log\log (d-s)+C'}+\sqrt{2\log s-\log\log s+C'}\vee \xi_d$ for some sufficiently large positive constant $C'$ and $\xi_d=\sqrt{(\log\log (d-s)-\log\log s)_+}$, then
\begin{equation}\label{cor_minimax_onesided2_eq3}
\liminf_{d,s\rightarrow\infty}\inf_{M\in\cM_+} \frac{R(M, \Theta^+(s,a))}{\sigma\sqrt{2\log s}/2}\geq 1.
\end{equation}
The sparse confidence set $\bar M_{\alpha'}$ satisfies $\bar M_{\alpha'}\in \cM_+$ and 
\begin{equation}\label{cor_minimax_onesided2_eq4}
\limsup_{d,s\rightarrow\infty}\frac{R(\bar M_{\alpha'}, \Theta^+(s,a))}{\sigma\sqrt{2\log s}}\leq 1.
\end{equation} 
\end{itemize}
\end{corollary}

\begin{remark}\label{rem_gap}
The inequalities (\ref{cor_minimax_onesided2_eq1}) and (\ref{cor_minimax_onesided2_eq3}) together lead to the asymptotic lower bound for $R(M, \Theta^+(s,a))$ over the class of one-sided confidence sets $\cM_+$ in two different regimes. Furthermore, (\ref{cor_minimax_onesided2_eq2}) and (\ref{cor_minimax_onesided2_eq4}) imply that the sparse confidence set $\bar M_{\alpha'}$ developed in Corollary \ref{cor_asym} matches the lower bounds up to a constant factor 2 in both regimes. 

However, we note that there exists a gap on the minimum SNR between these two regimes. Let us consider the following setting. By taking $A_d$ to be a sequence that diverges to infinity sufficiently slow, we have  $\sqrt{2\log (d/A_d-1)} \sim\sqrt{2\log d}$ in case (1). For case (2), assume that $s=d^\beta$ for some $0<\beta\leq c<1$ where $c$ is a constant. Then $\log (d-s)=\beta\log d+\log (d^{1-\beta}-1)=(1+o(1))\log d$, and  $\log s=\beta\log d$.  After some  algebra, it can be shown that $\tilde\kappa$ in case (2) satisfies
\begin{equation}\label{eqrem_gap}
1\leq \lim_{d\rightarrow\infty}\frac{\tilde\kappa}{\sqrt{2\log d}}\leq 1+\beta^{1/2}.
\end{equation}
Thus, the ratio between the two cut-points $\tilde\kappa$ in case (2) and $\sqrt{2\log (d/A_d-1)}$ in case (1) converges to $1$ as $\beta\rightarrow 0$, which occurs if $\btheta$ is very sparse with $s=\log d$ (i.e., $\beta=\log\log d/\log d$). In this case, the gap between the two regimes diminishes to $0$ asymptotically.
\end{remark}

\begin{remark}[Support recovery and oracle confidence set]\label{remcomparison}
Recall that if we know the support of $\btheta$, we can construct the following one-sided oracle confidence interval  $L_j^{oracle}=(X_j-\sigma\Phi^{-1}(1-\frac{\alpha}{s}))_+$ and $U_j^{oracle}=+\infty$ for $j\in  \supp(\btheta) $ and $L_j^{oracle}=U_j^{oracle}=0$ otherwise. This implies $\EE_{\btheta}(\theta_j-L_j^{oracle})\sim \sigma \sqrt{2\log s}$ for $j\in  \supp(\btheta) $. Intuitively, if the support set can be recovered exactly with high probability, i.e., $\hat S=\supp(\btheta)$ for some estimator $\hat S$, one would expect that (under some conditions) the same result holds for the plug-in oracle interval 
\begin{equation}\label{eqoracle_int}
\Big[\Big(X_j-\sigma\Phi^{-1}(1-\frac{\alpha}{|\hat S|})\Big)_+, +\infty\Big) ~~\textrm{for $j\in  \hat S $ and $0$ otherwise.} 
\end{equation}

However, in the following, we will show that the construction of oracle intervals (i.e., support recovery) is impossible even if the SNR satisfies the condition in case (2). 
In a recent work, \cite{butucea2018variable} established sufficient and necessary conditions for exact (and almost full) support recovery under the Gaussian mean model. Using their notation, define the expected Hamming loss for variable selection as $\EE_{\btheta}\|\hat\eta-\eta\|_1$, where $\eta=(\eta_1,...,\eta_d)$ with $\eta_j=I(\theta_j\neq 0)$ denotes the sparsity pattern of $\btheta$ and $\hat\eta$ is an estimator of $\eta$.  Consider the setting $a/\sigma=\tilde\kappa$ as in case (2). Theorem 4.2 (ii) of \cite{butucea2018variable} implies that, for $d$ large enough,
$$
\inf_{\hat\eta}\sup_{\btheta\in\Theta^+(s,a)} \EE_{\btheta}\|\hat\eta-\eta\|_1\geq s\Phi(-\Delta), ~~\textrm{where}~~\Delta=\frac{W}{2\sqrt{2\log (d-s)-2\log s+W}}
$$
and $W=4\log s+2\sqrt{(2\log (d-s)-\log\log (d-s)+C')(2\log s-\log\log s+C')}$ with $C'$ given in case (2). To  simplify the expression of $s\Phi(-\Delta)$, we consider the very sparse case with $s=\log d$. After some calculation, we can show that for $d$ sufficiently large,
$$
s\Phi(-\Delta)\geq s\Phi(-\sqrt{2\log s-\log\log s+C})\geq \sqrt{\frac{2}{\pi}}\frac{\sqrt{\log s}\exp(-C/2)}{3\sqrt{2\log s-\log\log s+C}}\rightarrow \sqrt{\frac{1}{\pi}}\frac{\exp(-C/2)}{3}>0,
$$
where $C$ is a constant and the second step follows by the tail inequality for the Gaussian random variables in Lemma \ref{lemtail}. The above derivation shows that, when $a/\sigma=\tilde\kappa$ satisfies the SNR condition in case (2), it is impossible to recover the support of $\btheta$ no matter what estimators to use. Since the support recovery is impossible, the plug-in oracle interval (\ref{eqoracle_int}) may not guarantee the desired coverage probability. Therefore, the minimax optimality results in Corollary \ref{cor_minimax_onesided2} do not hold for the plug-in oracle interval.   

\end{remark}

\section{Adaptive Sparse Confidence Sets}\label{sec_adaptive}

In this section, we consider how to construct optimal sparse confidence sets which are adaptive to the unknown sparsity $s$ and minimum signal strength $a$. In particular, we will show that adaptation is feasible in the following two scenarios, respectively. 
\begin{itemize}
\item[(A)] $\kappa^{**}\leq a/\sigma\leq \sqrt{2\log (d/A_d-1)}$ for some sequence $A_d\leq s$ satisfying $A_d\rightarrow\infty$ and $d/A_d \rightarrow\infty$.
\item[(B)] $a/\sigma\geq \sqrt{2\log (d-s)-\log\log (d-s)+C'}+\sqrt{2\log s-\log\log s+C'}\vee \bar\xi_d$ for some sufficiently large positive constant $C'$ and $\bar\xi_d=\sqrt{(2\log\log (d-s)-\log\log s)_+}$
\end{itemize}

Note that the scenario (A) is identical to the case (1) in Corollary \ref{cor_minimax_onesided2}. However, the scenario (B) is slightly different from the case (2) in Corollary \ref{cor_minimax_onesided2}, where the quantity $\xi_d$ is replaced with $\bar\xi_d$. If we consider the very sparse regime with $s=\log d$ as in Remark \ref{rem_gap}, the SNR condition in (B) reduces to $a/\sigma\geq (1+o(1))\sqrt{2\log d}$, which is asymptotically equivalent to the SNR condition in the case (2) in Corollary \ref{cor_minimax_onesided2}. 


The construction of the adaptive sparse confidence set in (A) is already available from the previous results as follows. Consider the sparse confidence set $\bar M_{\alpha'}$, which is defined via (\ref{eqhatmasy}) with $j\in \bar S_{\alpha'}$ if and only if $X_j/\sigma\geq  \Phi^{-1}(\delta)$ and $\bar u_{\alpha'}$ given by (\ref{eqhatu1}). Corollary \ref{cor_minimax_onesided2} part (1) implies that $\bar M_{\alpha'}$ satisfies $\bar M_{\alpha'}\in \cM_+$ and the inequality (\ref{cor_minimax_onesided2_eq2}), 
when the SNR satisfies the condition in scenario (A). Since the construction of $\bar M_{\alpha'}$ is free of any unknown quantities, the confidence set $\bar M_{\alpha'}$ is automatically adaptive in scenario (A). 

Now, we focus on the scenario (B). While Corollary \ref{cor_minimax_onesided2} part (2) implies that $\bar M_{\alpha'}$ has the desired coverage probability and is asymptotically optimal, the construction of  $\bar M_{\alpha'}$ with $\bar u_{\alpha'}$ given by (\ref{eqhatu2}) requires the knowledge of unknown sparsity $s$ and therefore is not adaptive. In the following, we propose to construct an adaptive sparse confidence set in scenario (B). Define 
$$
\bar S^{ad}_{\alpha'}=\Big\{j\in [d]: X_j/\sigma\geq \sqrt{2\log \left(\frac{2d}{(\alpha-\alpha')C_{d,\alpha-\alpha'}}\right)}\Big\},
$$ 
where $C_{d,\alpha'}=2\sqrt{\pi\log (\frac{d}{\alpha'})}$. Consider a grid of points $\{1, 2, 2^2,...,2^T\}$, where $T$ is the largest integer such that $2^T\leq d$. Define $\hat s=2^{\hat m}$, where $\hat m=\{m\in[T]: 2^{m-1}\leq |\bar S^{ad}_{\alpha'}|<2^m\}$. Finally, define the adaptive sparse confidence set as 
\begin{equation}\label{eqhatmasy2}
\hat M^{ad}_{\alpha'}=M(\bar S^{ad}_{\alpha'},  \hat\bU,  \hat\bL_{\hat s}), ~\textrm{where}~  \hat L_{j,\hat s}=(X_j- u_{\alpha',\hat s}\sigma)_+, ~\hat U_j=+\infty
\end{equation} 
for $j\in \bar S^{ad}_{\alpha'}$ and $\hat L_{j,\hat s}=\hat U_j=0$ otherwise, and 
$$
u_{\alpha',\hat s}=\sqrt{2\log \Big(\frac{4\hat s}{(\alpha-\alpha')C_{2\hat s,\alpha-\alpha'}}\Big)}.
$$

It is seen that the construction of the adaptive interval $\hat M^{ad}_{\alpha'}$ is similar to $\bar M_{\alpha'}$, but there are several key differences. First, we use a slightly different cutoff for $X_j/\sigma$ in $\bar S^{ad}_{\alpha'}$. When $2s\leq d$ and $s,d\rightarrow\infty$, both the cutoffs in $\bar S^{ad}_{\alpha'}$ and $\bar S_{\alpha'}$ 
are asymptotically equivalent to $\sqrt{2\log d}$. Second, we replace the unknown sparsity $s$ in $\bar u_{\alpha'}$ in (\ref{eqhatu2}) with $2\hat s$, where $\hat s$ can be viewed as the rounding of the cardinality of the set  $\bar S^{ad}_{\alpha'}$ to the grid $\{1, 2, 2^2,...,2^T\}$. The intuition is as follows. While the asymptotic exact recovery of the support set of $\btheta$ is infeasible under Scenario (B) (see Remark \ref{remcomparison}), $\bar S^{ad}_{\alpha'}$ is still a reasonable approximation of the unknown support set. In particular, we prove that the cardinality of $\bar S^{ad}_{\alpha'}$ is of an order $s$ with high probability. We further round $|\bar S^{ad}_{\alpha'}|$ to the grid in order to rigorously control $\EE_{\btheta}(\theta_j-\hat L_{j,\hat s})$ when $|\bar S^{ad}_{\alpha'}|$ is too large. The rounding step is similar to the peeling method in the empirical process \citep{vaart1996weak,kosorok2007introduction} and has been used in the Lepski's method for adaptive estimation \citep{lepskii1991problem,lepskii1992asymptotically,birge2001alternative}. 

The following theorem presents the main result in this section. 
\begin{theorem}\label{themadap}
Assume that $2s\leq d$, $s,d\rightarrow\infty$ and $\delta$ and $\alpha$ are fixed.  Let $\alpha'=\gamma\alpha$ for any constant $0<\gamma<1$. The adaptive sparse confidence set $\hat M^{ad}_{\alpha'}$ belongs to $\cM_+$ and 
$$
\limsup_{d,s\rightarrow\infty}\frac{R(\hat M^{ad}_{\alpha'}, \Theta^+(s,a))}{\sigma\sqrt{2\log s}}\leq 1,
$$
where $s,a$ satisfy the condition in scenario (B). 
\end{theorem}


Thus, the upper bound of $R(\hat M^{ad}_{\alpha'}, \Theta^+(s,a))$ is asymptotically identical to the ``non-adaptive" confidence set $\bar M_{\alpha'}$ as shown in Corollary \ref{cor_minimax_onesided2} part (2) and minimax optimal up to a constant.



\section{Extension to Two-sided Sparse Confidence Sets}\label{sec_twosided}
In this section, we assume that $\btheta\in\Theta(s,a)$, where 
$$
\Theta(s,a)=\{\btheta\in\RR^d: \|\btheta\|_0\leq s, \min_{j: \theta_j\neq 0} |\theta_j|\geq a>0\}.
$$
The goal is to generalize the results in Sections \ref{sec_onesided} and \ref{sec_opt} to two-sided sparse confidence intervals for $\btheta$ in $\Theta(s,a)$. To this end, consider the following estimator of the support set, 
\begin{equation}\label{eq2}
\hat S^{TS}_{\alpha'}=\Big\{j\in[d]:|X_j|/\sigma\geq \left(\Phi^{-1}(\frac{\alpha'}{2s})+a/\sigma\right)_+\vee \Phi^{-1}\left(\frac{1+\delta}{2}\right)\Big\},
\end{equation}
where $\alpha'$ is the tolerance level. Similarly, we require $|X_j|/\sigma\geq \Phi^{-1}((1+\delta)/2)$ to 
guarantee the resulting confidence interval is sparse, i.e., $\hat S^{TS}_{\alpha'}\in\cF(\delta)$, where $\cF(\delta)$ is defined in (\ref{eqF}). 

The following theorem, which is parallel to  Theorems \ref{themlower} and \ref{themupper}, establishes the upper and lower bounds of the non-coverage probability $\PP_{\btheta}(\supp(\btheta)\not\subseteq \hat S )$ under $\Theta(s,a)$. 
\begin{theorem}\label{them2s_uplow}
\begin{itemize}
\item[(1)] For any $s\geq 1$ and $0< \delta<1$,  we have
\begin{equation}
\inf_{\hat S \in\cF(\delta)}\sup_{\btheta\in\Theta(s,a)}\PP_{\btheta}(\supp(\btheta)\not\subseteq \hat S )\geq 1-\frac{1}{(\Delta_{TS}+1)^s},    \label{low_TS}
\end{equation}
where $\Delta_{TS}=\Phi(\Phi^{-1}(\frac{1+\delta}{2})+\frac{a}{\sigma})-\Phi(-\Phi^{-1}(\frac{1+\delta}{2})+\frac{a}{\sigma})$. 
\item[(2)] Assume that $s,d\rightarrow\infty$.  Let $c_s$ be a sequence satisfying $c_s\rightarrow\infty$ and $c_s/s\rightarrow 0$. Assume that $\delta\geq c$ for some constant $c>0$. If
\begin{equation}\label{alower_TS}
a/\sigma\le \phi_*:=\Phi^{-1}(\frac{1+\delta}{2})-\Phi^{-1}(\frac{c_s}{s}),
\end{equation} 
we have
\begin{equation}\label{low2_TS}
\liminf_{d,s\rightarrow\infty}\inf_{\hat S \in\cF(\delta)}\sup_{\btheta\in\Theta(s,a)}\PP_{\btheta}(\supp(\btheta)\not\subseteq \hat S )=1. 
\end{equation} 
\item[(3)] For any $0<\alpha'<1$, it holds that $\hat S^{TS}_{\alpha'} \in\cF(\delta)$. In addition, if 
\begin{equation}\label{astar_TS}
\frac{a}{\sigma}\geq \phi^*:=\Phi^{-1}(\frac{\delta+1}{2})-\Phi^{-1}(\frac{\alpha'}{2s})
\end{equation}
holds, then
\begin{equation}\label{up2_TS}
\sup_{\btheta\in\Theta(s,a)}\PP_{\btheta}( \supp(\btheta) \not\subseteq \hat S^{TS}_{\alpha'} )\leq \alpha'.
\end{equation} 
\end{itemize}
\end{theorem}
Note that in part (2), we require $\delta$ to be bounded away from $0$ by a constant. To see the reason, consider the extreme case $\delta=0$, which further implies $\Delta_{TS}=0$. In this case, the lower bound in (\ref{low_TS}) becomes $0$, which is no longer informative. 

In view of (\ref{alower_TS}) and (\ref{astar_TS}), we observe a similar phase transition phenomenon under the parameter space $\Theta(s,a)$; see Remark \ref{rem1} for details. 

Given the index set $\hat S^{TS}_{\alpha'}$, we define the two-sided sparse confidence set for $\btheta\in\Theta(s,a)$ as
$$
\hat M^{TS}_{\alpha'}=M(\hat S^{TS}_{\alpha'}, \hat\bU^{TS}, \hat\bL^{TS}), ~~~\textrm{where}~~\hat L_j^{TS}=X_j-\hat u^{TS}_{\alpha'}\sigma, ~~\hat U_j^{TS}=X_j+\hat u^{TS}_{\alpha'}\sigma
$$
for any $j\in \hat S^{TS}_{\alpha'} $ and
$$
\hat u^{TS}_{\alpha'}=\left\{
\begin{array}{ll}
\Phi^{-1}\Big(1-\frac{\alpha-\alpha'}{2d}\Big) &\textrm{if}~ \phi^*\leq \frac{a}{\sigma}< \phi^*\vee \Big[-\Phi^{-1}(\frac{\alpha-\alpha'}{2d})-\Phi^{-1}(\frac{\alpha'}{2s})\Big],\\
\Phi^{-1}\Big(1-\frac{\alpha-\alpha'-2(d-s)(1-\eta)}{2s}\Big)&\textrm{if}~\frac{a}{\sigma}\geq \phi^*\vee \Big[-\Phi^{-1}(\frac{\alpha-\alpha'}{2d})-\Phi^{-1}(\frac{\alpha'}{2s})\Big],
\end{array}
\right.
$$
where $\eta=\Phi(a/\sigma+\Phi^{-1}(\alpha'/(2s)))$. 

The following theorem shows that $\hat M^{TS}_{\alpha'}$ satisfies the conditions (\ref{eqcoverage}) and (\ref{eqsparsity}). 
\begin{theorem}\label{thm_spraseCI2}
For any given level $0< \alpha'<\alpha$, provided (\ref{astar_TS}) holds, we have 
$$
\sup_{\btheta\in\Theta(s,a), \theta_j=0}\PP_{\btheta}(j\in \hat S^{TS}_{\alpha'})\leq 1-\delta,~~~~\sup_{\btheta\in\Theta(s,a)}\PP_{\btheta}(\btheta\notin\hat M^{TS}_{\alpha'})\leq \alpha.
$$ 
\end{theorem}

We can develop a similar framework as in Section \ref{sec_opt} to study the optimality of the two-sided sparse confidence intervals. To this end, define the class of two-sided confidence sets as
\begin{align*}
CI=\{M(S, \bU, \bL):  ~&\textrm{$L_j, U_j$ only depend on $X_j$, $L_j\leq U_j$, and for $j\notin S$,  $L_j=U_j=0$}\}.
\end{align*}
To evaluate the optimality, it boils down to investigate the trade-off between the length of the interval $M(S, \bU, \bL)\in CI$, i.e., $\sup_{1\leq j\leq d}\EE_{\btheta}(U_j-L_j)$, and its coverage probability. Define
\begin{align*}
\cM(m, \delta)=\Big\{M(S, \bU, \bL)\in CI:  \sup_{1\leq j\leq d}\sup_{\btheta\in \Theta^+(s,a)}\EE_{\btheta}(U_j-L_j)\leq m,~\textrm{and}~S \in\cF(\delta)\Big\},
\end{align*}
to be the class of confidence sets such that the length is no greater than $m$ uniformly over $1\leq j\leq d$ and $\btheta\in \Theta(s,a)$ and $S \in\cF(\delta)$ holds as defined in (\ref{eqF}). 

The following theorem, parallel to Theorem \ref{thm_minimax_onesided}, provides the lower bound for the non-coverage probability of $M \in\cM(m,\delta)$. 
\begin{theorem}[Minimax lower bound]\label{thm_minimax_twosided}
For any $s\geq 1$ and $M \in\cM(m,\delta)$, it holds that
\begin{equation}\label{eqthm_minimax_twosided}
\sup_{\btheta\in \Theta(s,a)} \PP_{\btheta}(\btheta\notin M)\geq \max\Big(\sup_{\rho\geq a,A\leq s} G_{TS}(d,A,\rho,m), \sup_{\rho\geq 0, B\leq s} G_{TS}(s,B,\rho,m), 1-\frac{1}{(\Delta_{TS}+1)^s}\Big),
\end{equation}
where $\Delta_{TS}$ is defined in Theorem \ref{them2s_uplow}, 
$$
G_{TS}(d,A,\rho,m)=\frac{A[g_{TS}(d,A,\rho)-m/\rho]_+}{1+A[g_{TS}(d,A,\rho)-m/\rho]_+},
$$
with 
$$
g_{TS}(d,A,\rho)=\frac{2(d-A)}{A}\Phi(-D)+\Phi\Big(\frac{\rho}{\sigma}+D\Big)-\Phi\Big(\frac{\rho}{\sigma}-D\Big),
$$
and 
$$
D=\frac{\sigma}{\rho}\cosh^{-1}\left(\frac{d-A}{A}\exp(\frac{\rho^2}{2\sigma^2})\right),
$$
and $G_{TS}(s,B,\rho,m)$ is defined similarly. Note that $\cosh(x)=\exp(x)/2+\exp(-x)/2$ and $\cosh^{-1}$ is the inverse function of $\cosh(x)$ on $\RR^+$.
\end{theorem}

In practice, we usually pre-specify the coverage probability of the confidence set. Define 
\begin{align*}
\cM=\{M(S, \bU, \bL)\in CI:   \liminf_{d,s\rightarrow\infty}\inf_{\btheta\in \Theta(s,a)} \PP_{\btheta}(\btheta\in M(S, \bU, \bL))\geq 1-\alpha,~\textrm{and}~S \in\cF(\delta)\}.
\end{align*}
to be the two-sided sparse confidence sets with coverage probability no smaller than $1-\alpha$. We can similarly invert Theorem \ref{thm_minimax_twosided} to derive the lower bound for the length of confidence intervals $\sup_{1\leq j\leq d}\EE_{\btheta}(U_j-L_j)$ of $M\in \cM$. To match the lower bound, we consider the asymptotic version of $\hat M^{TS}_{\alpha'}$. Define
$$
\phi^{**}=\Phi^{-1}(\frac{1+\delta}{2})+\sqrt {2\log (\frac{2s}{C_{2s,\alpha'}\alpha'})},
$$
and 
$$
\bar \phi=\sqrt{2\log (\frac{4(d-s)}{(\alpha-\alpha')C_{2(d-s),\alpha-\alpha'}})}+\sqrt {2\log (\frac{2s}{C_{2s,\alpha'}\alpha'})},
$$
where $C_{s,\alpha'}=2(\pi\log(s/\alpha'))^{1/2}$. Define
\begin{equation}\label{eqhatmasy_TS}
\bar M^{TS}_{\alpha'}=M(\bar S^{TS}_{\alpha'},  \bar\bU^{TS},  \bar\bL^{TS}), ~\textrm{where}~  \bar L_j^{TS}=X_j- \bar u^{TS}_{\alpha'}\sigma, ~ \bar L_j^{TS}=X_j+ \bar u^{TS}_{\alpha'}\sigma
\end{equation} 
for $j\in \bar S_{\alpha'}$, where $\bar S^{TS}_{\alpha'}$ and $\bar u^{TS}_{\alpha'}$ are given as follows: 
\begin{itemize}
\item When $\phi^{**}\leq a/\sigma<\bar \phi$, define $j\in \bar S^{TS}_{\alpha'}$ if and only if $|X_j/\sigma|\geq  \Phi^{-1}((\delta+1)/2)$, and 
\begin{equation}\label{eqhatu1_TS}
\bar u^{TS}_{\alpha'}=\sqrt{2\log \Big(\frac{2d}{(\alpha-\alpha')C_{2d,\alpha-\alpha'}}\Big)}.
\end{equation} 
\item When $a/\sigma\geq \bar \phi$, define $j\in \bar S^{TS}_{\alpha'}$ if and only if $|X_j/\sigma|\geq \sqrt{2\log (\frac{4(d-s)}{(\alpha-\alpha')C_{2(d-s),\alpha-\alpha'}})}$, and 
\begin{equation}\label{eqhatu2_TS}
\bar u^{TS}_{\alpha'}=\sqrt{2\log \Big(\frac{4s}{(\alpha-\alpha')C_{2s,\alpha-\alpha'}}\Big)}.
\end{equation} 
\end{itemize}
Similar to Corollary \ref{cor_asym}, we can show that 
$$
\limsup_{d,s\rightarrow\infty}\sup_{\btheta\in\Theta(s,a), \theta_j=0}\PP_{\btheta}(j\in \bar S^{TS}_{\alpha'})\leq 1-\delta, ~~\limsup_{d,s\rightarrow\infty}\sup_{\btheta\in\Theta(s,a)}\PP_{\btheta}(\btheta\notin \bar M^{TS}_{\alpha'})\leq \alpha.
$$

Finally, in the  following corollary, we establish the optimality of $\bar M^{TS}_{\alpha'}$ within the class $\cM$. 

\begin{corollary}\label{cor_minimax_twosided2}
Assume that $d,s\rightarrow\infty$ and $0<\delta, \alpha<1$ are fixed. 
\begin{itemize}
\item[(1).] If $\phi^{**}\leq a/\sigma\leq \sqrt{2\log (d/A_d-1)}$ for some sequence $A_d\leq s$ satisfying $A_d\rightarrow\infty$ and $d/A_d \rightarrow\infty$, then
\begin{equation}\label{cor_minimax_twosided2_eq1}
\liminf_{d,s\rightarrow\infty}\inf_{M\in\cM} \frac{\sup_{1\leq j\leq d}\sup_{\btheta\in \Theta(s,a)} \EE_{\btheta}(U_j-L_j)}{\sigma\sqrt{2\log d}/2}\geq 1.
\end{equation}
For $\bar M^{TS}_{\alpha'}$ with $\alpha'=\gamma\alpha$ for any constant $0<\gamma<1$, we have $\bar M^{TS}_{\alpha'}\in \cM$ and 
\begin{equation}\label{cor_minimax_twosided2_eq2}
\limsup_{d,s\rightarrow\infty}\frac{\sup_{1\leq j\leq d}\sup_{\btheta\in \Theta(s,a)} \EE_{\btheta}(\bar U_j^{TS}-\bar L_j^{TS})}{2\sigma\sqrt{2\log d}}\leq 1.
\end{equation} 
\item[(2).] If $a/\sigma\geq \sqrt{2\log (d-s)-\log\log (d-s)+C'}+\sqrt{2\log s-\log\log s+C'}$ for some sufficiently large positive constant $C'$, then
\begin{equation}\label{cor_minimax_twosided2_eq3}
\liminf_{d,s\rightarrow\infty}\inf_{M\in\cM} \frac{\sup_{1\leq j\leq d}\sup_{\btheta\in \Theta(s,a)} \EE_{\btheta}(U_j-L_j)}{\sigma\sqrt{2\log s}/2}\geq 1.
\end{equation}
The sparse confidence set $\bar M_{\alpha'}$ satisfies $\bar M_{\alpha'}\in \cM$ and 
\begin{equation}\label{cor_minimax_twosided2_eq4}
\limsup_{d,s\rightarrow\infty}\frac{\sup_{1\leq j\leq d}\sup_{\btheta\in \Theta(s,a)} \EE_{\btheta}(\bar U^{TS}_j-\bar L_j^{TS})}{2\sigma\sqrt{2\log s}}\leq 1.
\end{equation} 
\end{itemize}
\end{corollary}

\begin{remark}[Comparison with Selective Confidence Intervals]\label{sec_compare}

Recently, there is a growing interest in developing confidence sets for the selected parameters $\btheta_S$. To be specific, let $S:=S(\bX)\subseteq [d]$ denote the set of indices of the selected parameters. For example, $S(\bX)$ can be $\{(1)\}$ where $(1)$ denotes the index of the larger of $X_1$ and $X_2$ or $S$ may contain the indices of  significant variables via some model selection procedures. Their goal is to construct two-sided confidence intervals for the randomly selected parameters  $\{\theta_i\}_{i\in S}$. Within the framework of selective confidence intervals, there are different types of error rates one may want to control, such as simultaneous over all possible selection (SoP) error rate \citep{berk2013valid}, conditional over selected error rate \citep{lee2016exact} and simultaneous over selected (SoS) error rate \citep{fuentes2018confidence,benjamini2019confidence}. Refer to \cite{benjamini2019confidence} for the detailed literature review. Note that one requirement of our sparse confidence set is (\ref{eqcoverage}) which also holds for the Bonferroni confidence intervals. As a result, the sparse confidence set controls the SoP and SoS errors at level $\alpha$; see Section 5 in \cite{benjamini2019confidence}. 

Indeed, our two-stage procedure is in a similar spirit to  selective confidence intervals. However, we have a different goal from the selective confidence intervals. In their framework, only the selected parameters $\{\theta_i\}_{i\in S}$ are of interest, without any confidence statement about the parameters not selected in $S$ (or equivalently their confidence interval for $\theta_i$ is $(-\infty,+\infty)$ for $i\notin S$). In contrast, the sparse confidence set is constructed  to cover the entire vector of $\btheta$ with any desired coverage probability. If $i\notin \hat S^{TS}_{\alpha'}$ in (\ref{eq2}), our confidence interval for $\theta_i$ is $0$. The uncertainty of assigning $0$ confidence intervals to $\theta_i$ is taken into account in the construction. 

\end{remark}


\section{Numerical Results}\label{sec_numerical}

In this section, we conduct simulation studies to evaluate the performance of the proposed sparse confidence sets and compare with several existing methods in terms of coverage probability, interval length, and support recovery (sparsity). The sensitivity to the choice of $\alpha'$ is also examined empirically. 

We generate $X$ from the normal mean model with $d=1000$, $\sigma=1$ and $\btheta=(a,...,a,0,...,0)$ where the first $s=100$ entries equal $a$, which is also the SNR, and the rest are $0$. We set $\alpha=0.05$, $\delta=0.7$ and vary the value of SNR in the simulations. Recall that the proposed one-sided sparse confidence set $\hat M_{\alpha'}$ in (\ref{eqhatm}) and its asymptotic version $\bar M_{\alpha'}$ in (\ref{eqhatmasy}) depend on the choice of $\alpha'$. For simplicity, we set $\alpha'=\alpha/2$ in view of Remark \ref{rem_alpha}. The sensitivity analysis of $\alpha'$ and further discussions will be shown subsequently.  

We compare the proposed sparse confidence set $\hat M_{\alpha'}$ in (\ref{eqhatm}), its asymptotic version $\bar M_{\alpha'}$ in (\ref{eqhatmasy}), and the adaptive version $\hat M^{ad}_{\alpha'}$ in (\ref{eqhatmasy2}) with the following three methods: Bonferroni confidence interval (\ref{eqBon}),  oracle interval (\ref{eqoracle}) assuming the support of $\btheta$ is known and the plug-in oracle interval (\ref{eqoracle_int}) where $j\in\hat S$ if and only if $X_j/\sigma>(2\log d)^{1/2}$. Provided the SNR is sufficiently large, the threshold $(2\log d)^{1/2}$ guarantees the exact support recovery as shown by \cite{butucea2018variable}. The simulation was repeated 500 times. We report the empirical coverage probability of the above confidence sets for $\btheta$ and the average distance $(\theta_j-L_j)$ over $j\in \supp(\btheta)$ (which can be viewed as a version of interval length for one-sided intervals). For $j\notin \supp(\btheta)$, we often observe that the lower confidence bound is $0$ and $\theta_j-L_j=\theta_j$. Hence, it is not very informative to look at the average distance $(\theta_j-L_j)$ over $j\notin \supp(\btheta)$, and thus we do not report these results. 

Figure \ref{fig1} shows the coverage probability and the average distance $(\theta_j-L_j)$ of the proposed sparse confidence set $\hat M_{\alpha'}$ (hat M), $\bar M_{\alpha'}$ (bar M), oracle interval (oracle), plug-in oracle interval (plug-in), Bonferroni confidence interval (Bonferroni) and our adaptive interval (adaptive) over 500 simulations. It is seen from the left panel that when SNR is small the sparse confidence sets ($\hat M_{\alpha'}$, $\bar M_{\alpha'}$ and $\hat M^{ad}_{\alpha'}$) all have considerably low coverage probability. This agrees with the minimax lower bound in Theorem \ref{themlower}, i.e., construction of sparse confidence sets is impossible if the SNR is too small. Provided the SNR exceeds $4$, all three versions of sparse confidence sets have very similar performance and their coverage probability becomes very close to the desired level.  It is of interest to mention that the coverage probability of the plug-in oracle intervals is only around $0.9$ even if the SNR is sufficiently large. This is because in finite sample the set $\hat S$ may still miss one or two nonzero signals so that the resulting confidence intervals  fail to cover the target parameter $\btheta$. 

From the right panel, we can see that when the SNR is moderate (say between $4$ and $7$) the average distance of our sparse confidence sets is comparable to Bonferroni confidence interval, which is consistent with part (1) of Corollary \ref{cor_minimax_onesided2}. Once SNR exceeds $7$, our sparse confidence sets have a smaller distance and outperform Bonferroni confidence interval; see part (2) of Corollary \ref{cor_minimax_onesided2}. Among these three versions of sparse confidence sets, $\hat M^{ad}_{\alpha'}$ is the most conservative one (with largest average distance $(\theta_j-L_j)$). This can be viewed as the price to pay for not knowing the sparsity $s$ when constructing the sparse confidence sets.

\begin{figure}
\centering
\subfigure{\includegraphics[scale=0.45]{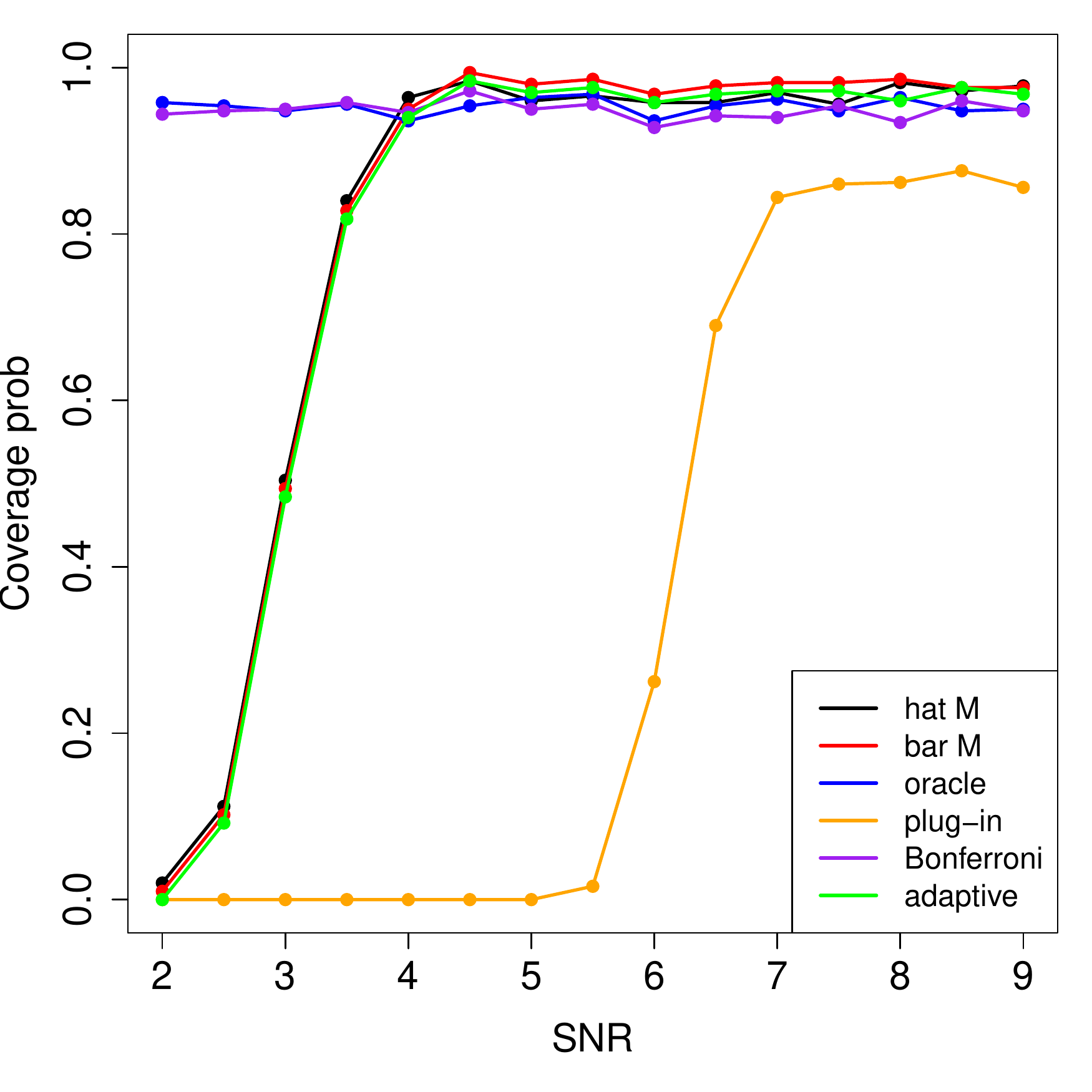}}
   \subfigure{\includegraphics[scale=0.45]{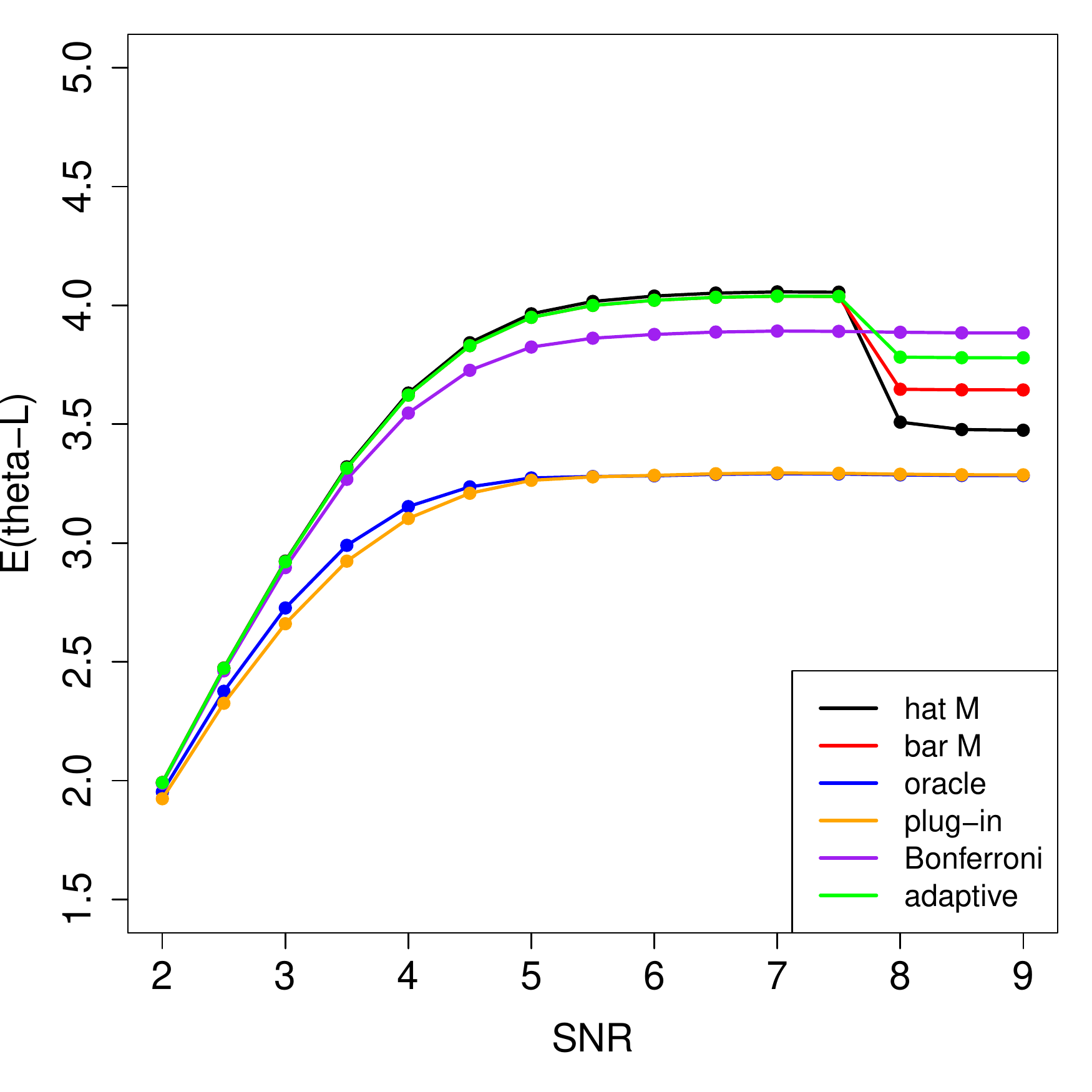}}
   \caption{\label{fig1} Coverage probability and the average distance $(\theta_j-L_j)$ over 500 simulations. }\vspace{-3mm}
   \end{figure}

To better understand the sparsity of the proposed sparse confidence set, we can take a closer look at the estimators of the support set, that is $\hat S_{\alpha'}$ in (\ref{eq1}), $\bar S_{\alpha'}$ in (\ref{eqhatmasy}) and $\hat S$ for the plug-in oracle interval. In particular, we plot $|\hat S_{\alpha'}|$, $|\bar S_{\alpha'}|$ and $|\hat S|$ in log scale in Figure \ref{fig2}. When the SNR is relatively small, $\hat S_{\alpha'}$ reduces to $\{j\in [d]: X_j/\sigma\geq \Phi^{-1}(\delta)\}$. This explains why the curve  for $\hat S_{\alpha'}$ (and similarly  $\bar S_{\alpha'}$) is horizontal for small SNR. As SNR further grows, it becomes easier to separate the nonzero signals from the rest, and therefore, the size of $\hat S_{\alpha'}$ and $\bar S_{\alpha'}$ decreases and eventually reduces to the true sparsity level. In contrast, the set $\hat S$ for support recovery has completely different behaviors. When the SNR is small, very few nonzero $\theta_j$ can be identified via $\hat S$ as $X_j\sim N(\theta_j,1)$ tends to be below the threshold $(2\log d)^{1/2}$. This explains why the coverage probability of the plug-in oracle interval is much lower than the desired level as seen in the left panel of Figure \ref{fig1}.



\begin{figure}
\centering
\includegraphics[scale=0.45]{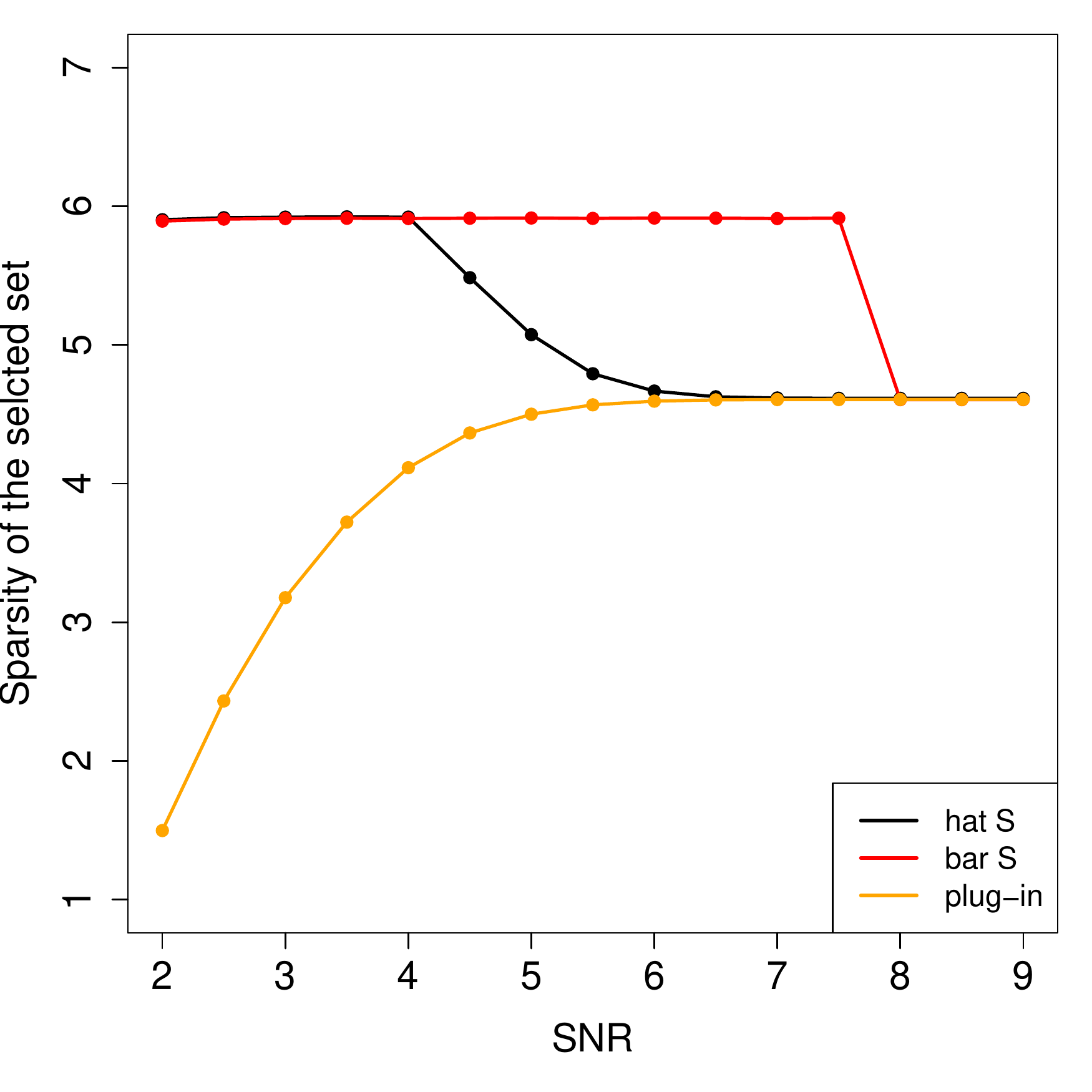}
   \caption{\label{fig2} Cardinality of $\hat S_{\alpha'}$ in (\ref{eq1}), $\bar S_{\alpha'}$ in (\ref{eqhatmasy}) and $\hat S=\{j: X_j/\sigma>(2\log d)^{1/2}\}$ for the plug-in oracle interval over 500 simulations. }\vspace{-3mm}
   \end{figure}

Finally, we analyze how sensitive the coverage probability and the average distance $(\theta_j-L_j)$ of proposed sparse confidence set $\hat M_{\alpha'}$ (hat M), $\bar M_{\alpha'}$ (bar M) is to the choice of $\alpha'$. Figure \ref{fig3} illustrates the results in two cases $SNR=3.8$ (moderate SNR) and $SNR=9$ (high SNR) respectively. In panel (a) and (b), when we increase $\alpha'$, the coverage probability becomes closer to the desired level, with the price that the average distance $(\theta_j-L_j)$ is slightly inflated. For the case where the SNR is sufficiently large (panel (c) and (d)), the coverage probability is less dependent on $\alpha'$, whereas the average distance tends to be much larger when $\alpha'$ is close to $\alpha=0.05$. While the effect of $\alpha'$ is asymptotically ignorable as seen in Remark \ref{rem_alpha}, in finite sample $\alpha'$ influences both the coverage probability and the distance $(\theta_j-L_j)$ of the proposed sparse confidence sets. As seen in Figure \ref{fig3}, it seems they are not very sensitive to the choice of $\alpha'$. For this reason, we simply take $\alpha'=\alpha/2$ in the previous simulations, leading to satisfactory numerical results.

\begin{figure}
\centering
\subfigure[$SNR=3.8$]{\includegraphics[scale=0.45]{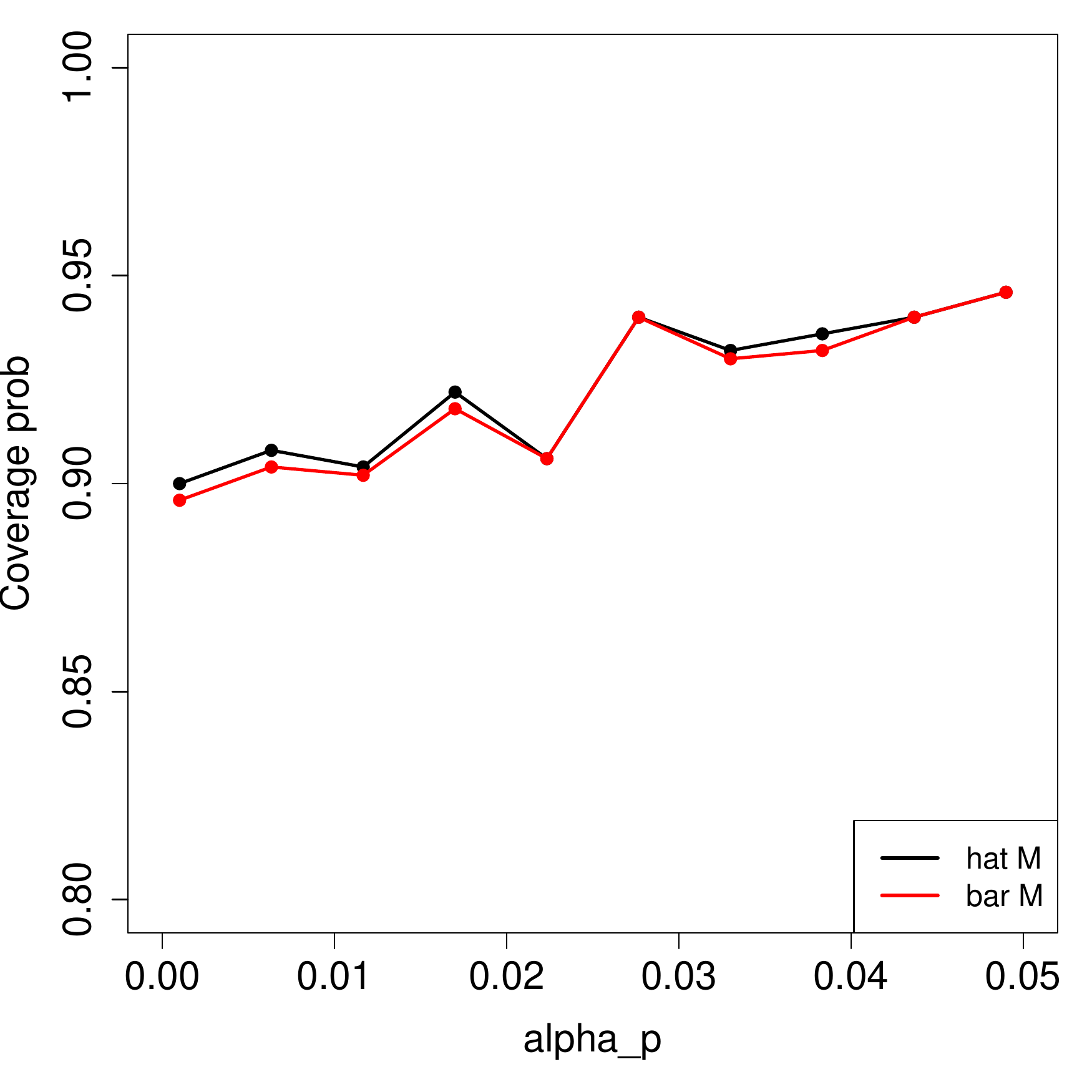}}
   \subfigure[$SNR=3.8$]{\includegraphics[scale=0.45]{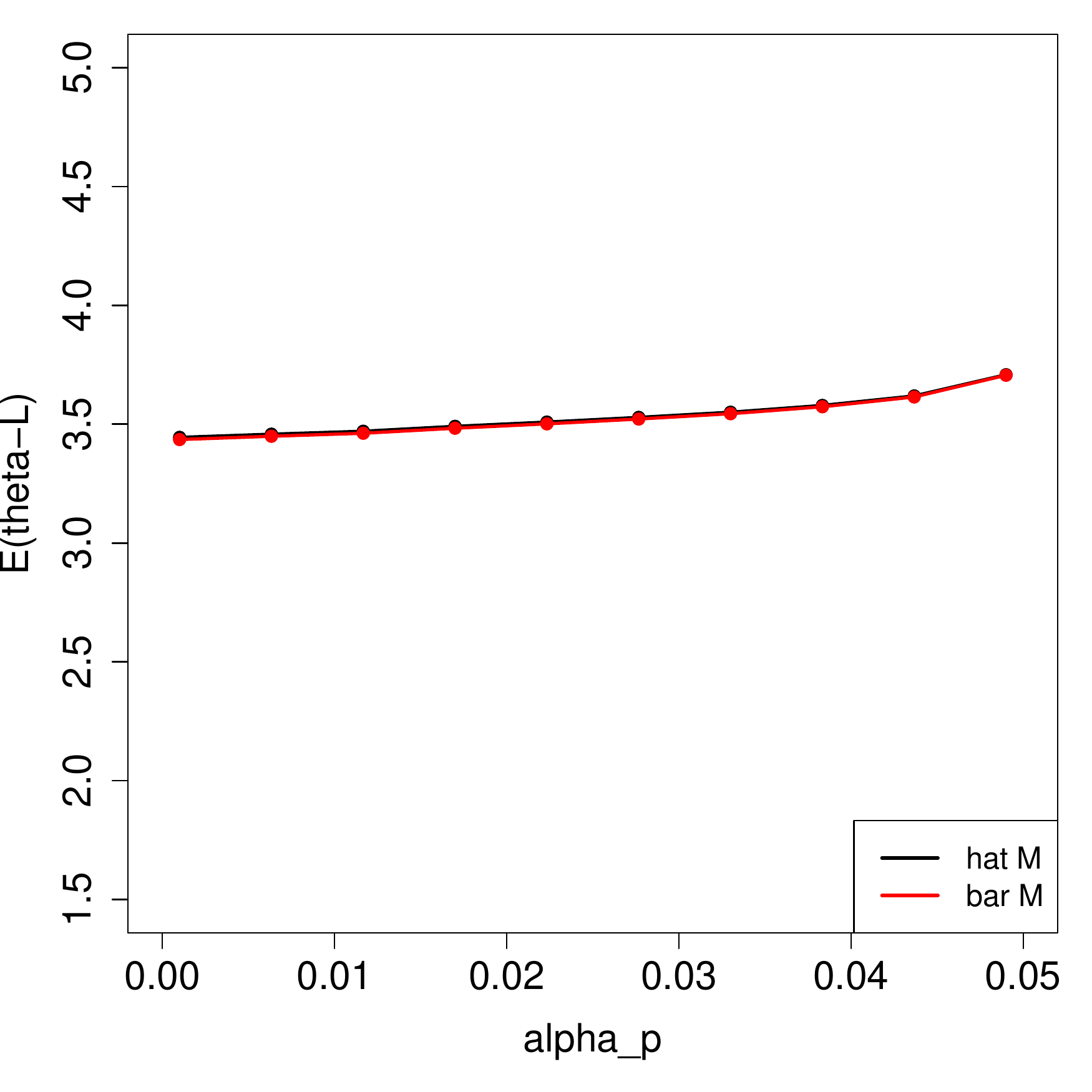}}\\
 \subfigure[$SNR=9$]{\includegraphics[scale=0.45]{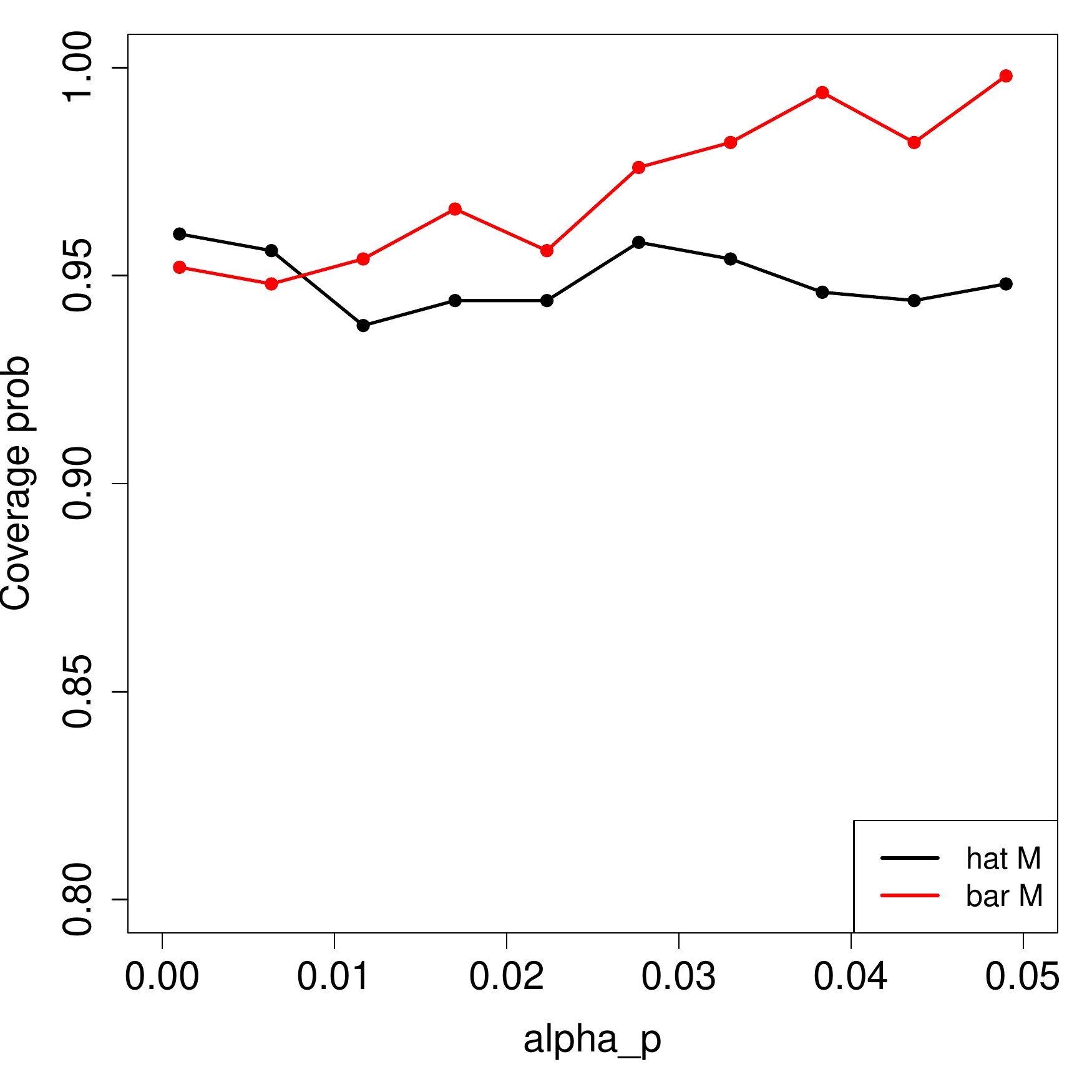}}
   \subfigure[$SNR=9$]{\includegraphics[scale=0.45]{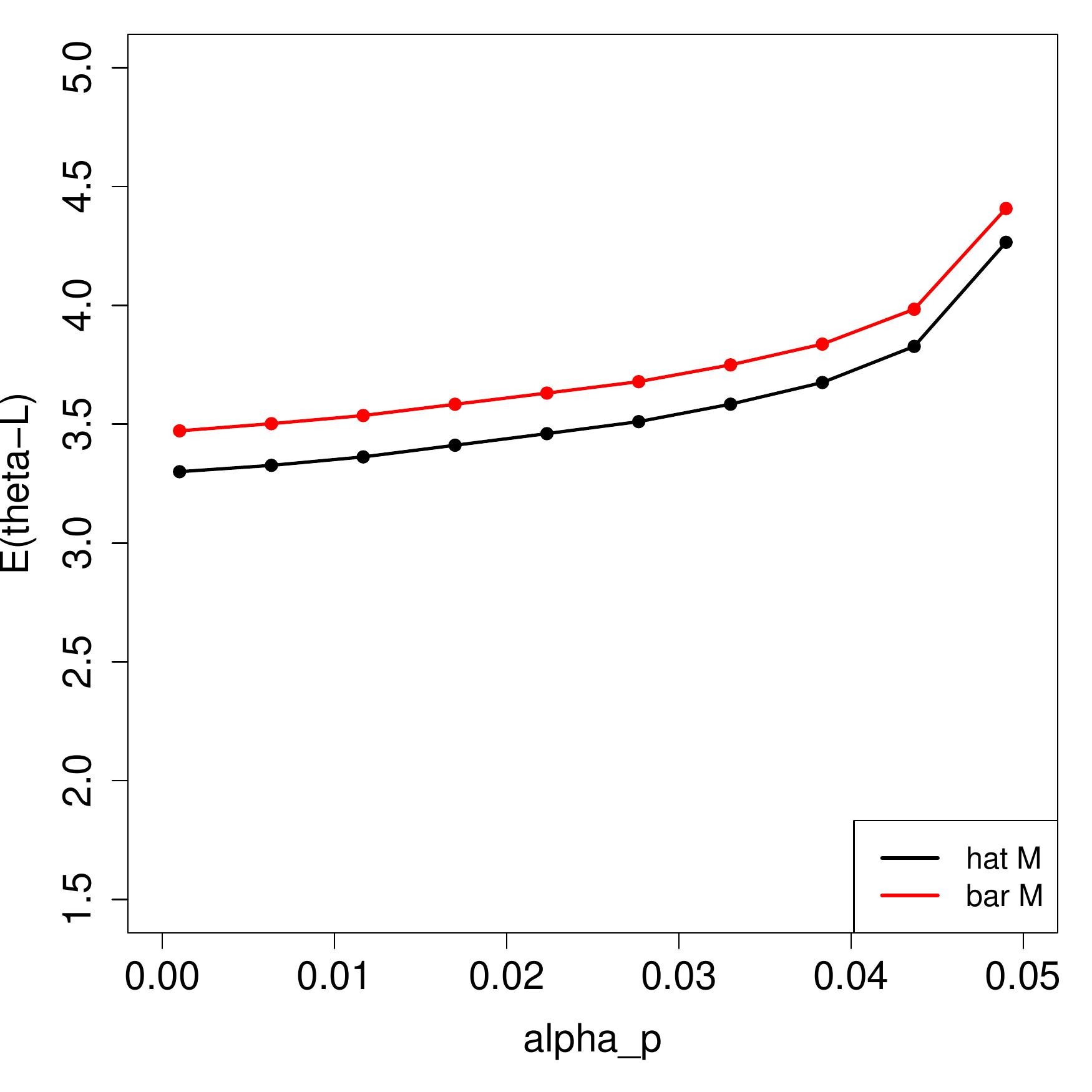}}\\  
   \caption{\label{fig3} Sensitivity analysis of coverage probability and the average distance $(\theta_j-L_j)$ with respect to $\alpha'$ (alpha$\_$p). }\vspace{-3mm}
   \end{figure}

\section{Proof}\label{sec_proof}

Note that, for notational simplicity, the constant $C$ may differ from line to line in the proof.

\subsection{Proof of Theorem \ref{themlower}}
Denote $\partial \Theta^+(s,a)=\{\btheta\in\RR^d: \|\btheta\|_0= s, \theta_j=a ~\textrm{for}~ \forall j, \theta_j\neq 0\}$. We know that
\begin{equation}
\sup_{\btheta\in\Theta^+(s,a)}\PP_{\btheta}( \supp(\btheta) \not\subseteq \hat S )\geq \sup_{\btheta\in\partial\Theta^+(s,a)}\PP_{\btheta}( \supp(\btheta) \not\subseteq \hat S )=1-\inf_{\btheta\in\partial\Theta^+(s,a)}\PP_{\btheta}( \supp(\btheta) \subseteq \hat S ).
\label{pfthemlower0}
\end{equation}
Since $\supp(\btheta)\subseteq \hat S $ is equivalent to the fact $j\in\hat S $ for any $j\in \supp(\btheta)$, we have $\PP_{\btheta}( \supp(\btheta) \subseteq \hat S )=\PP_{\btheta}(\cap_{j\in \supp(\btheta) }\{j\in\hat S \})=\prod_{j\in \supp(\btheta) }\PP_{\theta_j}(j\in\hat S )$, where the last step follows from the definition of the set $\cF(\delta)$. For notational simplicity, we denote $S=\supp(\btheta)$. We have
\begin{align}
\sup_{\btheta\in\partial\Theta^+(s,a)}\PP_{\btheta}( S \not\subseteq \hat S )&\geq \frac{1}{|\partial\Theta^+(s,a)|}\sum_{\btheta\in \partial\Theta^+(s,a)} \PP_{\btheta}( S \not\subseteq \hat S )\nonumber\\
&=\frac{1}{|\partial\Theta^+(s,a)|}\sum_{\btheta\in \partial\Theta^+(s,a)} \PP_{\btheta}(\cup_{j\in S }\{j\notin\hat S \})\nonumber\\
&=\frac{1}{|\partial\Theta^+(s,a)|}\sum_{\btheta\in \partial\Theta^+(s,a)} \Big(\sum_{j_1\in S }\PP_{\theta_{j_1}}(j_1\notin\hat S )\prod_{j\neq j_1\in S }\PP_{\theta_j}(j\in\hat S )\nonumber\\
&~~+\sum_{j_1\neq j_2\in S }\PP_{\theta_{j_1}}(j_1\notin\hat S )\PP_{\theta_{j_2}}(j_2\notin\hat S )\prod_{j\neq j_1,j_2\in S }\PP_{\theta_j}(j\in\hat S )+...+\prod_{j\in S }\PP_{\theta_j}(j\notin\hat S )\Big)\nonumber\\
&\geq \frac{t}{|\partial\Theta^+(s,a)|}\sum_{\btheta\in \partial\Theta^+(s,a)} \Big(\sum_{j_1\in S }\PP_{\theta_{j_1}}(j_1\notin\hat S )+\sum_{j_1\neq j_2\in S }\PP_{\theta_{j_1}}(j_1\notin\hat S )\PP_{\theta_{j_2}}(j_2\notin\hat S )\nonumber\\
&~~+...+\prod_{j\in S }\PP_{\theta_j}(j\notin\hat S )\Big),\label{pfthemlower1}
\end{align}
where $t=\inf_{\btheta\in\partial\Theta^+(s,a)}\prod_{j\in S }\PP_{\theta_j}(j\in\hat S )$. Define $u=\PP_a(j\notin\hat S )$, where $\PP_a$ denotes the probability of $X_j\sim N(a,\sigma^2)$. Note that $|\partial\Theta^+(s,a)|={d\choose s}$. Consider the $k$th term in (\ref{pfthemlower1}) $(1\leq k\leq s)$,
\begin{align*}
&\frac{t}{|\partial\Theta^+(s,a)|}\sum_{\btheta\in \partial\Theta^+(s,a)} \sum_{j_1\neq j_2..\neq j_k\in S }\prod_{m=1}^k\PP_{\theta_{j_m}}(j_m\notin\hat S )\\
&=t{d\choose s}^{-1}{d\choose k}{d-k\choose s-k} u^k=t{s\choose k} u^k.
\end{align*}
Thus, by taking the infimum, (\ref{pfthemlower1}) reduces to
\begin{equation}
\inf_{\hat S \in\cF(\delta)}\sup_{\btheta\in\partial\Theta^+(s,a)}\PP_{\btheta}( S \not\subseteq \hat S )\geq \inf_{\hat S \in\cF(\delta)} t\sum_{k=1}^s {s\choose k} u^k=\inf_{\hat S \in\cF(\delta)} t[(1+u)^s-1].\label{pfthemlower2}
\end{equation}
Next, we consider the infimum of $(1+u)^s$ over all possible $\hat S \in\cF(\delta)$. Then
$$
\inf_{\hat S \in\cF(\delta)}(1+\PP_a(j\notin\hat S ))^s=(1+\inf_{\hat S \in\cF(\delta)}\PP_a(j\notin\hat S ))^s.
$$
Since $j\in\hat S $ only depends on $X_j$, we can denote $j\in\hat S $ by $T(X_j)=1$ for some function $T(\cdot)$. Then Neyman-Pearson lemma implies that the infimum of $\PP_a(T(X_j)=0)$ over all possible $T(\cdot)$ such that $\PP_0(T(X_j)=1)\leq 1-\delta$ is attained by the likelihood ratio test of $X_j\sim N(0,\sigma^2)$ versus $X_j\sim N(a,\sigma^2)$. After some simple calculation, we find that the optimal $T(X_j)$ is
$$
T_{opt}(X_j)=I(\frac{\phi(X_j-a)}{\phi(X_j)}\geq c), ~~\textrm{where}~~c=\exp\Big\{\frac{a}{\sigma}\Phi^{-1}(\delta)-\frac{a^2}{2\sigma^2}\Big\}
$$
and $\phi(\cdot)$ is the pdf of the standard normal distribution. With this $T_{opt}(X_j)$, $\inf_{\hat S \in\cF(\delta)}\PP_a(j\notin\hat S )=\Delta$, where  $\Delta=\Phi(\Phi^{-1}(\delta)-\frac{a}{\sigma})$. Plugging into (\ref{pfthemlower2}), we obtain
$$
\inf_{\hat S \in\cF(\delta)}\sup_{\btheta\in\partial\Theta^+(s,a)}\PP_{\btheta}( \supp(\btheta)\not\subseteq \hat S )\geq t[(1+\Delta)^s-1]. 
$$
As $\inf_{\hat S \in\cF(\delta)}\sup_{\btheta\in\partial\Theta^+(s,a)}\PP(\supp(\btheta) \not\subseteq \hat S )\geq 1-t$ holds by (\ref{pfthemlower0}), optimizing over $t$ we obtain 
$$
\inf_{\hat S \in\cF(\delta)}\sup_{\btheta\in\Theta^+(s,a)}\PP_{\btheta}( \supp(\btheta) \not\subseteq \hat S )\geq 1-\frac{1}{(\Delta+1)^s}.
$$
This completes the proof of (\ref{low}). By (\ref{alower}), $\Delta\ge c_s/s$. When $c_s/s\rightarrow 0$, $\log (1+c_s/s)> (1-\epsilon) c_s/s$ for some constant $0<\epsilon<1$. Thus, 
$$
(\Delta+1)^s=\exp(s\log (1+\Delta))\ge\exp(s\log (1+c_s/s))> \exp((1-\epsilon)c_s)\rightarrow\infty,
$$
as $c_s\rightarrow\infty$ and $c_s/s\rightarrow 0$. Clearly, (\ref{low2}) follows from the non-asymptotic bound (\ref{low}).

\subsection{Proof of Theorem \ref{themupper}}
To show $\hat S_{\alpha'}\in\cF(\delta)$, notice that
$$
\PP_0(j\in\hat S_{\alpha'} )=\PP_0\Big(X_j/\sigma\geq \max(\Phi^{-1}(\frac{\alpha'}{s})+a/\sigma,\Phi^{-1}(\delta))\Big)\leq 
\PP_0(X_j/\sigma\geq \Phi^{-1}(\delta))=1-\delta.
$$
The event $ \supp(\btheta)\not\subseteq \hat S_{\alpha'} $ is equivalent to that there exists $j\in[d]$ such that $j\in \supp(\btheta) $ and $j\notin\hat S_{\alpha'} $. Then
\begin{align*}
\PP_{\btheta}( \supp(\btheta) \not\subseteq \hat S_{\alpha'} )&=\PP_{\btheta}(\exists j\in [d], j\in \supp(\btheta) , j\notin\hat S_{\alpha'} )\\
&\leq\sum_{j: \theta_j\neq 0} \PP_{\theta_j}(j\notin\hat S_{\alpha'} )\\
&=\sum_{j: \theta_j\neq 0} \PP_{\theta_j}(X_j\leq \sigma\Phi^{-1}(\frac{\alpha'}{s})+a),
\end{align*}
where the last line follows from the condition that $a\geq \sigma(\Phi^{-1}(\delta)-\Phi^{-1}(\alpha'/s))$. Since $X_j\sim N(\theta_j,\sigma^2)$, we have $\PP(X_j\leq t)=\Phi(\frac{t-\theta_j}{\sigma})$. Plugging into the above expression, we obtain
$$
\PP_{\btheta}( \supp(\btheta) \not\subseteq \hat S_{\alpha'} )\leq \sum_{j: \theta_j\neq 0} \Phi(\Phi^{-1}(\frac{\alpha'}{s})+\frac{a-\theta_j}{\sigma})\leq \|\btheta\|_0\frac{\alpha'}{s}=\alpha'.
$$
as $\theta_j\geq a$ for $\theta_j\neq 0$. This completes the proof of Theorem \ref{themupper}.

\subsection{Proof of Theorem \ref{thm_spraseCI}}
We first note that $\sup_{\btheta\in\Theta^+(s,a), \theta_j=0}\PP_{\btheta}(j\in \hat S_{\alpha'})\leq 1-\delta$ holds by Theorem \ref{themupper}. In the following, we bound $\PP_{\btheta}(\btheta\notin \hat M_{\alpha'})$ by intersecting with the event $\supp(\btheta) \subseteq \hat S_{\alpha'} $, 
\begin{align}
\PP_{\btheta}(\btheta\notin \hat M_{\alpha'})&\leq \PP_{\btheta}(\btheta\notin \hat M_{\alpha'},  \supp(\btheta)\subseteq \hat S_{\alpha'} )+ \PP_{\btheta}(\supp(\btheta)\not\subseteq \hat S_{\alpha'} )\nonumber\\
&= \PP_{\btheta}(\exists j\in \hat S_{\alpha'}, \theta_j<  \hat L_j, \supp(\btheta) \subseteq \hat S_{\alpha'} )+ \PP_{\btheta}(\supp(\btheta)\not\subseteq \hat S_{\alpha'} )\nonumber\\
&\leq \PP_{\btheta}(\exists j\in \hat S_{\alpha'}, \theta_j<  \hat L_j)+ \PP_{\btheta}(\supp(\btheta)\not\subseteq \hat S_{\alpha'} ).\label{eqpfthm_spraseCI0}
\end{align}
By Theorem \ref{themupper} and $a/\sigma\geq \kappa^*$, $\PP_{\btheta}(\supp(\btheta) \not\subseteq \hat S_{\alpha'} )\leq \alpha'$. The first term can be further bounded as
\begin{align}
&\PP_{\btheta}(\exists j\in \hat S_{\alpha'}, \theta_j< \hat L_j)\\
&\leq \PP_{\btheta}(\exists j\in  \supp(\btheta) , \theta_j<\hat L_j)+\PP_{\btheta}(\exists j\in \hat S_{\alpha'} \backslash  \supp(\btheta) , \theta_j< \hat L_j):=I_1+I_2.\label{eqpfthm_spraseCI1}
\end{align}
Write $u$ for $\hat u_{\alpha'}$. For $I_1$, by noting that $\theta_j< \max(X_j- u\sigma,0)$ is equivalent to $Z_j<  u$ where $Z_j=\frac{X_j-\theta_j}{\sigma}\sim N(0,1)$, we have 
\begin{equation}
I_1\leq \sum_{j\in \supp(\btheta)} \PP_{\btheta}(Z_j> u)=s (1-\Phi(u)). \label{eqpfthm_spraseCI2}
\end{equation}
To bound $I_2$, noting that $j\notin \supp(\btheta)$ implying $\theta_j=0$, we have
$$
I_2=\PP(\exists j\notin \supp(\btheta), Z_j\geq \Phi^{-1}(\frac{\alpha'}{s})+\frac{a}{\sigma}, Z_j> u)\leq \sum_{j\notin \supp(\btheta)}\PP( Z_j\geq \Phi^{-1}(\frac{\alpha'}{s})+\frac{a}{\sigma}, Z_j> u).
$$
To bound the last probability, we now consider the following two cases. 

(1). When $a/\sigma\in R_L$, by setting $u=\Phi^{-1}(1-\frac{\alpha-\alpha'}{d})$, we can easily verify that $\Phi^{-1}(\frac{\alpha'}{s})+\frac{a}{\sigma}\leq u$. Thus, $I_2\leq (d-s) (1-\Phi(u))$.
Together with (\ref{eqpfthm_spraseCI0}), (\ref{eqpfthm_spraseCI1}), (\ref{eqpfthm_spraseCI2}), we have
$$
\PP_{\btheta}(\btheta\notin \hat M_{\alpha'})\leq d(1-\Phi(u))+\alpha'= \alpha.
$$

(2). When $a/\sigma\in R_H$, by setting $u=\Phi^{-1}(1-\frac{\alpha-\alpha'-(d-s)(1-\eta^+)}{s})$, we can easily verify that $\Phi^{-1}(\frac{\alpha'}{s})+\frac{a}{\sigma}> u$. Thus, it implies $I_2\leq (d-s)(1-\eta^+)$, and finally we have
$$
\PP_{\btheta}(\btheta\notin \hat M_{\alpha'})\leq (d-s)(1-\eta^+)+s(1-\Phi(u))+\alpha'= \alpha.
$$

\subsection{Proof of Corollary \ref{cor_asym}}
When $a/\sigma<\bar \kappa$, it holds that
$$
\sup_{\btheta\in\Theta^+(s,a), \theta_j=0}\PP_{\btheta}(j\in \bar S_{\alpha'})=\PP_{\theta_j=0} (X_j/\sigma\geq \Phi^{-1}(\delta))= 1-\delta,
$$
and when $a/\sigma\geq\bar \kappa$ we have
$$
\sup_{\btheta\in\Theta^+(s,a), \theta_j=0}\PP_{\btheta}(j\in \bar S_{\alpha'})=\PP_{\theta_j=0}\Big(X_j/\sigma\geq \sqrt{2\log (\frac{2(d-s)}{(\alpha-\alpha')C_{d-s,\alpha-\alpha'}})}\Big)\leq 1-\delta.
$$
So, it also holds that $\sup_{\btheta\in\Theta^+(s,a), \theta_j=0}\PP_{\btheta}(j\in \bar S_{\alpha'})\leq 1-\delta$. 

In the following, we first focus on the case $a/\sigma<\bar \kappa$. Note that 
\begin{align*}
\PP_{\btheta}( \supp(\btheta) \not\subseteq \bar S_{\alpha'} )&\leq\sum_{j: \theta_j\neq 0} \PP_{\theta_j}(X_j\leq \sigma \Phi^{-1}(\delta))\\
&\leq\sum_{j: \theta_j\neq 0} \PP_{\theta_j}(\frac{X_j-\theta_j}{\sigma}\leq \frac{\sigma \Phi^{-1}(\delta)-a}{\sigma})\\
&\leq \sum_{j: \theta_j\neq 0} \PP_{\theta_j}\Big(\frac{X_j-\theta_j}{\sigma}\leq -\sqrt {2\log (\frac{s}{C_{s,\alpha'}\alpha'})}\Big),
\end{align*}
where the last step follows from $a/\sigma\geq \Phi^{-1}(\delta)+\sqrt {2\log (\frac{s}{C_{s,\alpha'}\alpha'})}$. By the tail probability in Lemma \ref{lemtail}, it yields for any $0<\alpha'<\alpha$
\begin{align}
&\lim_{d,s\rightarrow\infty}\sup_{\btheta\in\Theta^+(s,a)}\PP_{\btheta}(\supp(\btheta) \not\subseteq \bar S_{\alpha'} )-\alpha'\nonumber\\
&\leq \lim_{d,s\rightarrow\infty} s\sqrt{\frac{2}{\pi}}\frac{1}{2\sqrt {2\log (\frac{s}{C_{s,\alpha'}\alpha'})}}\exp\Big(-\log (\frac{s}{C_{s,\alpha'}\alpha'})\Big)-\alpha'\nonumber\\
&=\lim_{d,s\rightarrow\infty}\sqrt{\frac{2}{\pi}}\frac{\alpha'C_{s,\alpha'}}{2\sqrt {2\log (\frac{s}{C_{s,\alpha'}\alpha'})}}-\alpha'= 0.\label{eq_cor_asym1}
\end{align}
By the proof of Theorem \ref{thm_spraseCI}, we can similarly show that
\begin{align}\label{eq_cor_asym2}
\PP_{\btheta}(\btheta\notin \bar M_{\alpha'})\leq d\Big(1-\Phi\Big(\sqrt{2\log (\frac{d}{(\alpha-\alpha')C_{d,\alpha-\alpha'}})}\Big)\Big)+\PP_{\btheta}( \supp(\btheta) \not\subseteq \bar S_{\alpha'} ).
\end{align}
By taking the limit $d,s\rightarrow\infty$, similar to (\ref{eq_cor_asym1}), the tail bound in Lemma \ref{lemtail} implies 
$$\lim_{d,s\rightarrow\infty}\sup_{\btheta\in\Theta^+(s,a)}\PP_{\btheta}(\btheta\notin \bar M_{\alpha'})\leq (\alpha-\alpha')+\alpha'=\alpha.
$$ 
When $a/\sigma\geq\bar \kappa$, it is easily seen that
\begin{align}
\PP_{\btheta}( \supp(\btheta) \not\subseteq \bar S_{\alpha'} )&\leq\sum_{j: \theta_j\neq 0} \PP_{\theta_j}\Big(X_j/\sigma\leq  \sqrt{2\log (\frac{2(d-s)}{(\alpha-\alpha')C_{d-s,\alpha-\alpha'}})}\Big)\nonumber\\
&\leq \sum_{j: \theta_j\neq 0} \PP_{\theta_j}\Big(\frac{X_j-\theta_j}{\sigma}\leq -\sqrt {2\log (\frac{s}{C_s\alpha'})}\Big).\label{eq_cor_asym11}
\end{align}
As a result, (\ref{eq_cor_asym1}) still holds. In the following, we consider two cases separately. 

Case (1) $d\geq 2s$. Recall the way of controlling the term $I_2$ in the proof of Theorem \ref{thm_spraseCI}. In this case, $j$ is selected if $Z_j\geq t$, where 
$$
t=\sqrt{2\log (\frac{2(d-s)}{(\alpha-\alpha')C_{d-s,\alpha-\alpha'}})}.
$$ 
With the monotonicity of the function $\log x-\frac{1}{2}\log\log x$, this term is no smaller than $\bar u_{\alpha'}$ as $d-s\geq s$. Thus, we can show that by the proof of Theorem \ref{thm_spraseCI},
\begin{align*}
\PP_{\btheta}(\btheta\notin \bar M_{\alpha'})&\leq s(1-\Phi(\bar u_{\alpha'}))+\sum_{j\notin \supp(\btheta)}\PP_0\big({X_j}/{\sigma}\geq t\big)+\PP_{\btheta}(\supp(\btheta) \not\subseteq \bar S_{\alpha'} )\\
&=s\Big(1-\Phi\Big(\sqrt{2\log (\frac{2s}{(\alpha-\alpha')C_{s,\alpha-\alpha'}})}\Big)\Big)+(d-s)(1-\Phi(t))+\PP_{\btheta}( \supp(\btheta)\not\subseteq \bar S_{\alpha'} ).
\end{align*}
Similar to (\ref{eq_cor_asym1}), 
\begin{align*}
&\lim_{d,s\rightarrow\infty} s(1-\Phi(\bar u_{\alpha'}))-\frac{\alpha-\alpha'}{2}\\
&=\lim_{d,s\rightarrow\infty}  s\sqrt{\frac{2}{\pi}}\frac{1}{2\sqrt {2\log (\frac{2s}{C_{s,\alpha-\alpha'}(\alpha-\alpha')})}}\exp\Big(-\log (\frac{2s}{C_{s,\alpha-\alpha'}(\alpha-\alpha')})\Big)-\frac{\alpha-\alpha'}{2}=0,
\end{align*}
and
\begin{align*}
&\lim_{d,s\rightarrow\infty} (d-s)(1-\Phi(t))-\frac{\alpha-\alpha'}{2}\\
&=\lim_{d,s\rightarrow\infty}(d-s)\sqrt{\frac{2}{\pi}}\frac{1}{2\sqrt {2\log (\frac{2(d-s)}{C_{d-s,\alpha-\alpha'}(\alpha-\alpha')})}}\exp\Big(-\log (\frac{2(d-s)}{C_{d-s,\alpha-\alpha'}(\alpha-\alpha')})\Big)-\frac{\alpha-\alpha'}{2}=0.
\end{align*}
We obtain
$$
\lim_{d,s\rightarrow\infty}\sup_{\btheta\in\Theta^+(s,a)}\PP_{\btheta}(\btheta\notin \bar M_{\alpha'})\leq \frac{\alpha-\alpha'}{2}+\frac{\alpha-\alpha'}{2}+\alpha'=\alpha.
$$

Case (2) $d<2s$. Unlike the previous case, we now have $t<\bar u_{\alpha'}$. Thus,
\begin{align*}
\PP_{\btheta}(\btheta\notin \bar M_{\alpha'})&\leq d(1-\Phi(\bar u_{\alpha'}))+\PP_{\btheta}(\supp(\btheta) \not\subseteq \bar S_{\alpha'} )\\
&=d\Big(1-\Phi\Big(\sqrt{2\log (\frac{2s}{(\alpha-\alpha')C_{s,\alpha-\alpha'}})}\Big)\Big)+\PP_{\btheta}( \supp(\btheta)\not\subseteq \bar S_{\alpha'} ).
\end{align*}
Note that
\begin{align*}
&\lim_{d,s\rightarrow\infty} d(1-\Phi(\bar u_{\alpha'}))-(\alpha-\alpha')\\
&=\lim_{d,s\rightarrow\infty}  \frac{d}{2s}\sqrt{\frac{2}{\pi}}\frac{C_{s,\alpha-\alpha'}(\alpha-\alpha')}{2\sqrt {2\log (\frac{2s}{C_{s,\alpha-\alpha'}(\alpha-\alpha')})}}-(\alpha-\alpha')\leq 0,
\end{align*}
where we use $d<2s$ in the last step. This implies  $\lim_{d,s\rightarrow\infty}\sup_{\btheta\in\Theta^+(s,a)}\PP_{\btheta}(\btheta\notin \bar M_{\alpha'})\leq \alpha$. 

\subsection{Proof of Theorem \ref{thm_minimax_onesided}}

Define $\Theta(A)=\{\btheta\in\RR^d: \|\btheta\|_0=A, \theta_j=\rho, ~\textrm{for any $\theta_j\neq 0$}\}$, where $0<A\leq s$ and $\rho$ is an arbitrary positive quantity that is $\rho\geq a$. Then, $\Theta(A)$ is contained in the parameter space $ \Theta^+(s,a)$. For any $M$ in $\cM_+(m,\delta)$, we use $CI_j$ to denote the confidence interval for $\theta_j$. Following the similar arguments in the proof of Theorem \ref{themlower}, we have
\begin{align}
\sup_{\btheta\in \Theta^+(s,a)} \PP_{\btheta}(\btheta\notin M)&\geq {d\choose A}^{-1}\sum_{\btheta\in \Theta(A)} \PP_{\btheta}( \btheta\notin M)\nonumber\\
&\geq{d\choose A}^{-1}\sum_{\btheta\in \Theta(A)} \Big(\sum_{j=1}^d\PP_{\theta_j}(\theta_j\notin CI_j )\prod_{j'\neq j }\PP_{\theta_{j'}}(\theta_{j'}\in CI_{j'} )\Big)\nonumber\\
&\geq t{d\choose A}^{-1}\sum_{j=1}^d \sum_{\btheta\in \Theta(A)} \PP_{\theta_j}(\theta_j\notin CI_j )\nonumber\\
&= t{d\choose A}^{-1}\sum_{j=1}^d \Big(\sum_{\btheta\in \Theta(A), \theta_j=0} \PP_{0}(0\notin CI_j )+\sum_{\btheta\in \Theta(A), \theta_j=\rho} \PP_{\rho}(\rho\notin CI_j )\Big)\nonumber\\
&=t\sum_{j=1}^d  \Big(\frac{d-A}{d} \PP_{0}(0\notin CI_j )+\frac{A}{d}\PP_{\rho}(\rho\notin CI_j )\Big),\label{eq_thm_minimax_onesided1}
\end{align}
where $t=\inf_{\btheta\in\Theta(A)}\prod_{j=1}^d\PP_{\theta_j}(\theta_j\in CI_j)$. Furthermore, we can control the last term in (\ref{eq_thm_minimax_onesided1}) as follows,
\begin{equation}\label{eq_thm_minimax_onesided2}
\PP_{\rho}(\rho\notin CI_j )\geq \PP_{\rho}(L_j>\rho)=\PP_{\rho}(L_j=0)+[\PP_{\rho}(L_j>\rho)-\PP_{\rho}(L_j=0)].
\end{equation}
Since $M\in\cM_+(m,\delta)$ implies $ \sup_{1\leq j\leq d}\sup_{\btheta\in \Theta(s,a)}\EE_{\btheta}(\theta_j-L_j)\leq m$, by taking $\theta_j=\rho$ we have
\begin{align*}
\rho-m\leq \EE_\rho L_j&\leq \rho \PP_\rho(0<L_j\leq \rho)+2\rho \PP_\rho (\rho<L_j\leq 2\rho)+\EE_\rho L_j I(L_j>2\rho)\\
&= \rho \PP_\rho(0<L_j\leq \rho)+2\rho \PP_\rho (\rho<L_j)+\EE_\rho (L_j-2\rho) I(L_j>2\rho).
\end{align*}
Then, we can plug $\PP_\rho(0<L_j\leq \rho)=1-\PP_\rho(L_j=0)-\PP_\rho(L_j>\rho)$ into the above display, which can reduce to
\begin{equation}\label{eq_thm_minimax_onesided3}
\PP_{\rho}(L_j>\rho)-\PP_{\rho}(L_j=0)\geq -\frac{m+\EE_\rho (L_j-2\rho) I(L_j>2\rho)}{\rho}.
\end{equation}
Our next step is to upper bound $\EE_\rho (L_j-2\rho) I(L_j>2\rho)$. Recall that we assume $L_j\leq X_j$ whenever $X_j\geq 0$. Thus
\begin{align*}
\EE_\rho (L_j-2\rho) I(L_j>2\rho)&\leq \EE_\rho (X_j\vee 0-2\rho) I(X_j\vee 0>2\rho)\\
&= \EE_\rho (X_j-2\rho) I(X_j>2\rho)=\sigma \EE N I(N>\frac{\rho}{\sigma})-\rho \PP(N>\frac{\rho}{\sigma}),
\end{align*}
where $N\sim N(0,1)$. By the tail bound in Lemma \ref{lemtail} and some simple algebra, 
\begin{align*}
\EE_\rho (L_j-2\rho) I(L_j>2\rho) &\leq \sqrt{\frac{1}{2\pi}}\sigma \exp\Big(-\frac{1}{2}(\frac{\rho}{\sigma})^2\Big)-\sqrt{\frac{2}{\pi}} \exp\Big(-\frac{1}{2}(\frac{\rho}{\sigma})^2\Big)\frac{\rho}{\frac{\rho}{\sigma}+\sqrt{4+\frac{\rho^2}{\sigma^2}}}\\
&\leq \sqrt{\frac{1}{2\pi}}\sigma \exp\Big(-\frac{1}{2}(\frac{\rho}{\sigma})^2\Big)\frac{\sqrt{1+4(\sigma/\rho)^2}-1}{\sqrt{1+4(\sigma/\rho)^2}+1}:=R.
\end{align*}
Combining with (\ref{eq_thm_minimax_onesided1}), (\ref{eq_thm_minimax_onesided2}) and (\ref{eq_thm_minimax_onesided3}), we have shown that
\begin{align*}
\sup_{\btheta\in \Theta^+(s,a)} \PP_{\btheta}(\btheta\notin M) &\geq t\Big\{\sum_{j=1}^d  \Big(\frac{d-A}{d} \PP_{0}(0\notin CI_j )+\frac{A}{d}\PP_{\rho}(0 \in CI_j )\Big)-A\frac{m+R}{\rho}\Big\}\\
& \geq tA\Big\{\inf_{T\in\{0,1\}}  \Big(\frac{d-A}{A} \EE_{0}(1-T)+\EE_{\rho}T\Big)-\frac{m+R}{\rho}\Big\}.
\end{align*}
Here, $T$ denotes an arbitrary test function from $\RR$ to $\{0,1\}$. By the Neyman-Pearson lemma, the optimal test function is given by 
$$
T_{opt}(x)=I\Big(x\leq \frac{\rho}{2}+\frac{\sigma^2}{\rho}\log (\frac{d-A}{A})\Big).
$$
With the optimal test function, lower bound reduces to 
\begin{align*}
\sup_{\btheta\in \Theta^+(s,a)} \PP_{\btheta}(\btheta\notin M) & \geq tA\Big\{g(d,A,\rho)-\frac{m+R}{\rho}\Big\}_+,
\end{align*}
where 
$$
g(d,A,\rho)=\frac{d-A}{A}\Phi\Big(-\frac{\rho}{2\sigma}-\frac{\sigma}{\rho}\log (\frac{d}{A}-1)\Big)+\Phi\Big(-\frac{\rho}{2\sigma}+\frac{\sigma}{\rho}\log (\frac{d}{A}-1)\Big).
$$
We also notice that
$$
\sup_{\btheta\in \Theta^+(s,a)} \PP_{\btheta}(\btheta\notin M)\geq 1-\inf_{\btheta\in\Theta(A)}\PP_{\btheta}(\btheta\in M)=1-\inf_{\btheta\in\Theta(A)}\prod_{j=1}^d\PP_{\theta_j}(\theta_j\in CI_j)=1-t.
$$
We then optimize the lower bound with respect to $t$, which leads to
$$
\sup_{\btheta\in \Theta^+(s,a)} \PP_{\btheta}(\btheta\notin M)\geq G(d, A, \rho, m). 
$$
As the above lower bound holds for any $0< A\leq s$ and $\rho\geq a$, we obtain
$$
\sup_{\btheta\in \Theta^+(s,a)} \PP_{\btheta}(\btheta\notin M)\geq \sup_{\rho\geq a,A\leq s} G(d, A, \rho, m). 
$$
The rest of the proof focuses on showing
\begin{equation}\label{eq_thm_minimax_onesided4}
\sup_{\btheta\in \Theta^+(s,a)} \PP_{\btheta}(\btheta\notin M)\geq \sup_{\rho\geq 0, B\leq s} G(s, B, \rho, m). 
\end{equation}
We first define a $s$-dimensional vector $\ba=(a,a,...,a)\in\RR^s$, and $\Theta_s(B)=\{\btheta\in\RR^s: \|\btheta-\ba\|_0=B, \theta_j=a+\rho~\textrm{for}~\theta_j\neq a\}$, where $0<B\leq s$ and $\rho$ is an arbitrary positive quantity. Then, we define the parameter set $\Theta'(B)=\{(\btheta, 0,...,0)\in\RR^d: \btheta\in  \Theta_s(B)\}$, which is contained in the parameter space $ \Theta^+(s,a)$. In this case, we only perturb the parameters that are nonzero. Let $\btheta_{[s]}$ and $M_{[s]}$ denote the first $s$ entries of $\btheta$ and the confidence intervals for $\btheta_{[s]}$. Similar to the previous argument, we can show that
\begin{align*}
\sup_{\btheta\in \Theta^+(s,a)} \PP_{\btheta}(\btheta\notin M)&\geq \sup_{\btheta\in \Theta^+(s,a)} \PP_{\btheta}(\btheta_{[s]}\notin M_{[s]})\\
&\geq{s\choose B}^{-1}\sum_{\btheta\in \Theta'(B)} \Big(\sum_{j=1}^s\PP_{\theta_j}(\theta_j\notin CI_j )\prod_{j'\neq j \in[s]}\PP_{\theta_{j'}}(\theta_{j'}\in CI_{j'} )\Big)\nonumber\\
&\geq t'{s\choose B}^{-1}\sum_{j=1}^s \Big(\sum_{\btheta\in \Theta'(B), \theta_j=a} \PP_{a}(a\notin CI_j )+\sum_{\btheta\in \Theta'(B), \theta_j=a+\rho} \PP_{a+\rho}(a+\rho\notin CI_j )\Big)\nonumber\\
&=t'\sum_{j=1}^s  \Big(\frac{s-B}{s} \PP_{a}(a\notin CI_j )+\frac{B}{s}\PP_{a+\rho}(a+\rho\notin CI_j )\Big),
\end{align*}
where $t'=\inf_{\btheta\in\Theta'(B)}\prod_{j=1}^s\PP_{\theta_j}(\theta_j\in CI_j)$. In addition, 
$$
\PP_{a+\rho}(a+\rho\notin CI_j )-\PP_{a+\rho}(a\in CI_j )\geq \PP_{a+\rho}(L_j>a+\rho)-\PP_{a+\rho}(0\leq L_j<a).
$$
To further lower bound the right hand side of the above display, we notice that
\begin{align*}
a+\rho-m&\leq \EE_{a+\rho} L_j\\
&\leq a \PP_{a+\rho}(0\leq L_j\leq a)+(a+\rho) \PP_{a+\rho} (a<L_j\leq a+\rho)\\
&~~~+(a+2\rho)\PP_{a+\rho}(a+\rho<L_j\leq a+2\rho) +\EE_{a+\rho}L_j I(L_j>a+2\rho)\\
&\leq a \PP_{a+\rho}(0\leq L_j\leq a)+(a+\rho) \{1-\PP_{a+\rho} (0\leq L_j\leq a)-\PP_{a+\rho} ( L_j> a+\rho)\}\\
&~~~+(a+2\rho)\PP_{a+\rho}(a+\rho<L_j) +\EE_{a+\rho}(L_j-(a+2\rho)) I(L_j>a+2\rho)\\
&=(a+\rho)-\rho \PP_{a+\rho}(0\leq L_j\leq a)+\rho\PP_{a+\rho}(a+\rho<L_j) +\EE_{a+\rho}(L_j-(a+2\rho)) I(L_j>a+2\rho).
\end{align*}
Thus, we have
$$
\PP_{a+\rho}(L_j>a+\rho)-\PP_{a+\rho}(0\leq L_j<a)\geq -\frac{m+\EE_{a+\rho} (L_j-a-2\rho) I(L_j>a+2\rho)}{\rho}.
$$
Finally, we also note that 
$$
\sup_{\btheta\in \Theta^+(s,a)} \PP_{\btheta}(\btheta_{[s]}\notin M_{[s]})\geq 1-\inf_{\btheta\in\Theta'(B)}\PP_{\btheta}(\btheta_{[s]}\in M_{[s]})=1-t'.
$$
The rest of the proofs are similar and therefore we omit the details. Together with Theorem \ref{themlower}, we  complete the proof.

\subsection{Proof of Corollary \ref{cor_minimax_onesided}}

Denote $a_*=\sigma\kappa_*$. When $a\leq a_*$, Corollary \ref{corlower1} holds. When $a_*<a\leq a_1:=\sigma\sqrt{2\log (d/A_d-1)}$, Theorem \ref{thm_minimax_onesided} implies
$$
\inf_{M \in\cM_+(m,\delta)} \sup_{\btheta\in \Theta^+(s,a)} \PP_{\btheta}(\btheta\notin M)\geq G(d,A_d,a_1,m)\geq G(d,A_d,a_1,m^*), 
$$
where $m^*=\sigma(\frac{1}{2}-\frac{W_d}{A_d})\sqrt{2\log (d/A_d-1)}+\frac{\sqrt{2}\sigma}{4\sqrt{\pi}}(1-\frac{A_d}{d-A_d})$ with $A_d, W_d$ defined as in the Corollary. By Lemma \ref{lemtail}
\begin{equation}\label{eq_pf_cor_minimax_onesided1}
g(d,A_d,a_1)=\frac{1}{2}+\frac{d-A_d}{A_d}\Phi(-\sqrt{2\log (d/A_d-1)})\geq \frac{1}{2}+\frac{1}{2\sqrt{\pi}}\frac{1}{\sqrt{\log (d/A_d-1)+2}}.
\end{equation}
In addition,
$$
\frac{m^*}{a_1}=\frac{1}{2}-\frac{W_d}{A_d}+\frac{1}{4\sqrt{\pi}}\frac{1}{\sqrt{\log (d/A_d-1)}}\Big(1-\frac{A_d}{d-A_d}\Big),
$$
and for $d, s$ large enough
\begin{align*}
\frac{R}{a_1}&=\frac{1}{\sqrt{2\pi}}\frac{A_d}{d-A_d}\frac{1}{\sqrt{2\log (d/A_d-1)}}\frac{\sqrt{1+2/\log (d/A_d-1)}-1}{\sqrt{1+2/\log (d/A_d-1)}+1}\\
&\leq \frac{1}{4\sqrt{\pi}}\frac{1}{\sqrt{\log (d/A_d-1)}}\frac{A_d}{d-A_d}.
\end{align*}
Thus, for $d$ large enough and as $\frac{d}{A_d}\rightarrow\infty$, 
\begin{align*}
&A_d[g(d,A_d,a_1)-(m^*+R)/a_1]\\
&\geq A_d[\frac{1}{2\sqrt{\pi}}\frac{1}{\sqrt{\log (d/A_d-1)+2}}+\frac{W_d}{A_d}-\frac{1}{4\sqrt{\pi}}\frac{1}{\sqrt{\log (d/A_d-1)}}\Big(1-\frac{A_d}{d-A_d}\Big)-\frac{R}{a_1}]\\
&=A_d[\frac{1}{2\sqrt{\pi}}\frac{1}{\sqrt{\log (d/A_d-1)+2}}+\frac{W_d}{A_d}-\frac{1}{4\sqrt{\pi}}\frac{1}{\sqrt{\log (d/A_d-1)}}]\\
&> W_d.
\end{align*}
Since $W_d\rightarrow\infty$, we have 
$$
\liminf_{d,s\rightarrow\infty} A_d[g(d,A_d,a_1)-(m^*+R)/a_1]=+\infty, 
$$
which further implies $\liminf_{d,s\rightarrow\infty}G(d,A_d,a_1,m^*)=1$.

Similarly, when $a\geq \sigma\sqrt{2\log (d/A-1)}$, Theorem \ref{thm_minimax_onesided} implies
$$
\inf_{M \in\cM_+(m,\delta)} \sup_{\btheta\in \Theta^+(s,a)} \PP_{\btheta}(\btheta\notin M)\geq G(s,B_s,\rho^*,m^{**}), 
$$
where $m^{**}=\sigma(\frac{1}{2}-\frac{V_s}{B_s})\sqrt{2\log (s/B_s-1)}+\frac{\sqrt{2}\sigma}{4\sqrt{\pi}}(1-\frac{B_s}{s-B_s})$ and $\rho^*=\sigma\sqrt{2\log (s/B_s-1)}$. Following the similar argument, it is shown that 
\begin{align*}
B_s[g(s,B_s,\rho^*)-(m^{**}+R)/\rho^*]&\geq B_s[\frac{1}{2\sqrt{\pi}}\frac{1}{\sqrt{\log (s/B_s-1)+2}}+\frac{V_s}{B_s}-\frac{R}{\rho^*}]\geq V_s.
\end{align*}
This implies $\liminf_{d,s\rightarrow\infty}G(s,B_s,\rho^*,m^{**})=1$. This completes the proof.

\subsection{Proof of Corollary \ref{cor_minimax_onesided2}}
{\bf Proof of (\ref{cor_minimax_onesided2_eq1}) and (\ref{cor_minimax_onesided2_eq3})}.  For simplicity, we follow the the same notations in the proof of Corollary \ref{cor_minimax_onesided} and Theorem  \ref{thm_minimax_onesided}. By the proof of Theorem \ref{thm_minimax_onesided} and $M\in \cM_+$, we have $ \PP_{\btheta}(\btheta\notin M) \geq 1-t$ and
$$
 \PP_{\btheta}(\btheta\notin M)\geq tA_d\Big\{g(d,A_d,\rho)-\frac{m+R}{\rho}\Big\}_+,
$$
where $m=\sup_{1\leq j\leq d}\sup_{\btheta\in \Theta^+(s,a)} \EE_{\btheta}(\theta_j-L_j)$. If $g(d,A_d,\rho)-\frac{m+R}{\rho}\leq 0$, then
\begin{equation}\label{eq_cor_minimax_onesided21}
m\geq \rho g(d,A_d,\rho)-R.
\end{equation}
Otherwise, we have
$$
 \PP_{\btheta}(\btheta\notin M)\geq tA_d\Big\{g(d,A_d,\rho)-\frac{m+R}{\rho}\Big\}\geq \PP_{\btheta}(\btheta\in M)A_d\Big\{g(d,A_d,\rho)-\frac{m+R}{\rho}\Big\},
$$
which implies
\begin{equation}\label{eq_cor_minimax_onesided22}
m\geq \rho \Big\{g(d,A_d,\rho)-\frac{1}{A_d}\frac{ \PP_{\btheta}(\btheta\notin M)}{ \PP_{\btheta}(\btheta\in M)}\Big\}-R
\end{equation}
Clearly, (\ref{eq_cor_minimax_onesided21}) implies (\ref{eq_cor_minimax_onesided22}), and the lower bound reduces to (\ref{eq_cor_minimax_onesided22}) by combining these two cases. 

When  $a\leq a_1:=\sigma\sqrt{2\log (d/A_d-1)}$, we can take $\rho=a_1$. By the proof of Corollary \ref{cor_minimax_onesided}, e.g., (\ref{eq_pf_cor_minimax_onesided1}),  
\begin{align*}
&\liminf_{d,s\rightarrow\infty}\inf_{M\in\cM_+} \frac{m}{a_1}\\
&\geq \liminf_{d,s\rightarrow\infty}\Big[\frac{1}{2}+\frac{1}{2\sqrt{\pi}}\frac{1}{\sqrt{\log (d/A_d-1)+2}}-\frac{1}{A_d}\frac{ \PP_{\btheta}(\btheta\notin M)}{ \PP_{\btheta}(\btheta\in M)}-\frac{R}{a_1}\Big]\\
&\geq \liminf_{d,s\rightarrow\infty}\Big[\frac{1}{2}+\frac{1}{2\sqrt{\pi}}\frac{1}{\sqrt{\log (d/A_d-1)+2}}-\frac{1}{A_d}\frac{\alpha}{1-\alpha}- \frac{1}{4\sqrt{\pi}}\frac{1}{\sqrt{\log (d/A_d-1)}}\frac{A_d}{d-A_d}\Big]\\
&=\frac{1}{2}.
\end{align*}
We then  obtain (\ref{cor_minimax_onesided2_eq1}). 

In case (2) it is easy to verify that for $d, s$ large enough
\begin{align*}
a/\sigma&\geq \sqrt{2\log (d-s)-\log\log (d-s)+C'}+\sqrt{2\log s-\log\log s+C'}\vee \xi_d\\
&\geq \sqrt{2\log (d-s)-\log\log (d-s)+C'}+\sqrt{2\log s-\log\log s+C'}\geq \sqrt{2\log d+C},
\end{align*} 
for some constant $C>0$. This further implies (\ref{cor_minimax_onesided2_eq3}) by the proof of Corollary \ref{cor_minimax_onesided} and the similar argument in case (1). 

{\bf Proof of (\ref{cor_minimax_onesided2_eq2}) and (\ref{cor_minimax_onesided2_eq4})}. We first note that Corollary \ref{cor_asym} implies $\bar M_{\alpha'}\in \cM_+$. For (\ref{cor_minimax_onesided2_eq2}), it suffices to show that the following inequality holds regardless of the value of $a$, 
$$
\limsup_{d,s\rightarrow\infty}\frac{\sup_{1\leq j\leq d}\sup_{\btheta\in \Theta^+(s,a)} \EE_{\btheta}(\theta_j-L_j)}{\sigma\sqrt{2\log d}}\leq 1,
$$
where $L_j$ is defined in (\ref{eqhatmasy}) and for notational simplicity we write $L_j$ for $\bar L_j$. First consider the case that $a^{**}\leq a<\bar a$, where $a^{**}=\kappa^{**}\sigma$ and 
 $\bar {a}/{\sigma}=\sqrt{2\log (\frac{2(d-s)}{(\alpha-\alpha')C_{d-s,\alpha-\alpha'}})}+\sqrt {2\log (\frac{s}{C_{s,\alpha'}\alpha'})}$.  Thus, for $d$ large enough
\begin{align}
\EE_{\btheta}(\theta_j-L_j)&=\EE_{\btheta}(\theta_j-L_j)I(j\in\bar S_{\alpha'})+\EE_{\btheta}(\theta_j-L_j)I(j\notin\bar S_{\alpha'})\nonumber\\
&=\EE_{\btheta}(\theta_j-L_j)I(X_j/\sigma>\Phi^{-1}(\delta))+\theta_j\PP_{\btheta}(X_j/\sigma\leq\Phi^{-1}(\delta))\nonumber\\
&=\EE_{\btheta}(\theta_j-X_j+\bar u_{\alpha'}\sigma)I(X_j/\sigma>\bar u_{\alpha'})+\theta_j\PP_{\btheta}(X_j/\sigma\leq \bar u_{\alpha'})\nonumber\\
&\leq \sigma+\bar u_{\alpha'}\sigma\PP_{\btheta}(X_j/\sigma> \bar u_{\alpha'})+\theta_j\PP_{\btheta}(X_j/\sigma\leq \bar u_{\alpha'}),\label{eq_cor_minimax_onesided231}
\end{align}
where $\bar u_{\alpha'}$ is defined in (\ref{eqhatu1}). If $\bar u_{\alpha'}\sigma\geq \theta_j$, the above display implies 
$$
\EE_{\btheta}(\theta_j-L_j)\leq \sigma(1+\bar u_{\alpha'}). 
$$
If $\bar u_{\alpha'}\sigma< \theta_j$, the above display and the Gaussian tail bound in Lemma \ref{lemtail} leads to 
\begin{align*}
\EE_{\btheta}(\theta_j-L_j)&\leq \sigma(1+\bar u_{\alpha'})+(\theta_j-\bar u_{\alpha'}\sigma)\PP(Z> \theta_j/\sigma-\bar u_{\alpha'})\\
&\leq \sigma(1+\bar u_{\alpha'})+\frac{\sigma}{\sqrt{2\pi}}\exp\Big(-\frac{1}{2}(\theta_j/\sigma-\bar u_{\alpha'})^2\Big)\\
&\leq \sigma(1+\bar u_{\alpha'})+\frac{\sigma}{\sqrt{2\pi}},
\end{align*}
where $Z\sim N(0,1)$.  Combining these two cases, we have
\begin{equation}\label{eq_cor_minimax_onesided24}
\limsup_{d,s\rightarrow\infty}\frac{\sup_{1\leq j\leq d}\sup_{\btheta\in \Theta^+(s,a)} \EE_{\btheta}(\theta_j-L_j)}{\sigma \sqrt{2\log d}}\leq \limsup_{d,s\rightarrow\infty}\frac{1+\bar u_{\alpha'}+1/\sqrt{2\pi}}{\sqrt{2\log d}}=1.
\end{equation}
Now we consider the case $a\geq \bar a$.  If $\bar u_{\alpha'}\leq w$, we have
\begin{align}
\EE_{\btheta}(\theta_j-L_j)&=\EE_{\btheta}(\theta_j-L_j)I(j\in\bar S_{\alpha'})+\EE_{\btheta}(\theta_j-L_j)I(j\notin\bar S_{\alpha'})\nonumber\\
&=\EE_{\btheta}(\theta_j-X_j+\bar u_{\alpha'}\sigma)I(X_j/\sigma>w)+\theta_j\PP_{\btheta}(X_j/\sigma\leq w)\nonumber\\
&\leq \sigma+\bar u_{\alpha'}\sigma\PP_{\btheta}(X_j/\sigma> w)+\theta_j\PP_{\btheta}(X_j/\sigma\leq w),\label{eq_cor_minimax_onesided23}
\end{align}
where $w=\sqrt{2\log (\frac{2(d-s)}{(\alpha-\alpha')C_{d-s,\alpha-\alpha'}})}$ and $\bar u_{\alpha'}$ is defined in (\ref{eqhatu2}). Note that it suffices to only consider the nonzero $\theta_j$, since if $\theta_j=0$, (\ref{eq_cor_minimax_onesided23}) can be trivially bounded by $\sigma(1+\bar u_{\alpha'})$. Then, for $\theta_j\neq 0$ we have
$$
\frac{\theta_j}{\sigma}\geq \frac{a}{\sigma}\geq \bar\kappa >w\geq \bar u_{\alpha'}.
$$
By the same derivation, it can be shown that 
\begin{equation}\label{eq_cor_minimax_onesided25}
\EE_{\btheta}(\theta_j-L_j)\leq \sigma(1+\bar u_{\alpha'})+\frac{\sigma}{\sqrt{2\pi}}\frac{\theta_j/\sigma-\bar u_{\alpha'}}{\theta_j/\sigma-w}\exp\Big(-\frac{1}{2}(\theta_j/\sigma-w)^2\Big).
\end{equation}
Note that 
$$
\lim_{d,s\rightarrow\infty} \frac{\sigma(1+\bar u_{\alpha'})}{\sigma \sqrt{2\log d}}\leq \lim_{d,s\rightarrow\infty} \frac{1+\sqrt{2\log s-\log\log s+C}}{\sqrt{2\log d}}\leq 1,
$$
for some constant $C$ (depending on $\alpha',\alpha$), and 
\begin{align*}
\frac{\theta_j/\sigma-\bar u_{\alpha'}}{\theta_j/\sigma-w}&=1+\frac{w-\bar u_{\alpha'}}{\theta_j/\sigma-w}\leq 1+\frac{w-\bar u_{\alpha'}}{\sqrt {2\log (\frac{s}{C_{s,\alpha'}\alpha'})}}\leq 1+\frac{w}{\sqrt {2\log (\frac{s}{C_{s,\alpha'}\alpha'})}}.
\end{align*}
Plugging into (\ref{eq_cor_minimax_onesided25}) and notice that $w/\sqrt{2\log d}\leq 1$ for $d$ sufficiently large, we obtain 
$$
\frac{\sup_{1\leq j\leq d}\sup_{\btheta\in \Theta^+(s,a)} \EE_{\btheta}(\theta_j-L_j)}{\sigma \sqrt{2\log d}}\leq 1+O(\frac{1}{\sqrt{\log s}})\rightarrow 1.
$$
However, if $\bar u_{\alpha'}> w$, (\ref{eq_cor_minimax_onesided23}) has a slightly different form as follows
\begin{align}
\EE_{\btheta}(\theta_j-L_j)&=\EE_{\btheta}(\theta_j-L_j)I(j\in\bar S_{\alpha'})+\EE_{\btheta}(\theta_j-L_j)I(j\notin\bar S_{\alpha'})\nonumber\\
&=\EE_{\btheta}(\theta_j-X_j+\bar u_{\alpha'}\sigma)I(X_j/\sigma>\bar u_{\alpha'})+\theta_j\PP_{\btheta}(X_j/\sigma\leq \bar u_{\alpha'})\nonumber\\
&\leq \sigma+\bar u_{\alpha'}\sigma\PP_{\btheta}(X_j/\sigma> \bar u_{\alpha'})+\theta_j\PP_{\btheta}(X_j/\sigma\leq \bar u_{\alpha'}),\label{eq_cor_minimax_onesided2311}
\end{align}
which is identical to (\ref{eq_cor_minimax_onesided231}) expect that $\bar u_{\alpha'}$ is defined in (\ref{eqhatu2}) rather than (\ref{eqhatu1}). However, this does not change the proof of (\ref{eq_cor_minimax_onesided24}) (i.e., (\ref{eq_cor_minimax_onesided24}) still holds). This completes the proof of (\ref{cor_minimax_onesided2_eq2}). 

Finally, we focus on the last result  (\ref{cor_minimax_onesided2_eq4}). As $d,s\rightarrow\infty$, 
\begin{align*}
a/\sigma&\geq \sqrt{2\log (d-s)-\log\log (d-s)+C'}+\sqrt{2\log s-\log\log s+C'}\vee \xi_d\\
&\geq \sqrt{2\log (\frac{2(d-s)}{(\alpha-\alpha')C_{d-s,\alpha-\alpha'}})}+\sqrt {2\log (\frac{s}{C_{s,\alpha'}\alpha'})}=\bar\kappa.
\end{align*} 
Thus, by the definition of $\bar M_{\alpha'}$, when $a/\sigma\geq \bar\kappa$, $\bar u_{\alpha'}$ is defined in (\ref{eqhatu2}). By (\ref{eq_cor_minimax_onesided25}), we first note that
$$
\lim_{d,s\rightarrow\infty} \frac{\sigma(1+\bar u_{\alpha'})}{\sigma \sqrt{2\log s}}\leq \lim_{d,s\rightarrow\infty} \frac{1+\sqrt{2\log s-\log\log s+C}}{\sqrt{2\log s}}=1.
$$
For $d,s$ large enough, 
$$
\theta_j/\sigma-w\geq a/\sigma-w\geq \sqrt{2\log s-\log\log s+C'}\vee \xi_d.
$$ 
Recall that  $\xi_d=\sqrt{(\log\log (d-s)-\log\log s)_+}$. Thus, uniformly over $\btheta$ we have
\begin{align*}
&\frac{\theta_j/\sigma-\bar u_{\alpha'}}{\theta_j/\sigma-w}\exp\Big(-\frac{1}{2}(\theta_j/\sigma-w)^2\Big)\\
&=\Big(1+\frac{w-\bar u_{\alpha'}}{\theta_j/\sigma-w}\Big)\exp\Big(-\frac{1}{2}(\theta_j/\sigma-w)^2\Big)\\
&\leq \Big(1+\frac{w}{\sqrt{2\log s-\log\log s+C'}}\Big)\exp\Big(-\frac{1}{2} (\xi_d^2\vee (2\log s-\log\log s+C'))\Big).
\end{align*}
If $d-s>s$ holds, we have that as $d-s\rightarrow\infty$ and $s\rightarrow\infty$,
\begin{align*}
&\frac{\theta_j/\sigma-\bar u_{\alpha'}}{\theta_j/\sigma-w}\exp\Big(-\frac{1}{2}(\theta_j/\sigma-w)^2\Big)\\
&\leq\Big(1+\frac{\sqrt{2\log (d-s)-\log\log (d-s)+C}}{\sqrt{2\log s-\log\log s+C'}}\Big)\sqrt{\frac{\log s}{\log (d-s)}} \leq 2.
\end{align*}
Otherwise, under $d-s\leq s$, 
\begin{align*}
&\frac{\theta_j/\sigma-\bar u_{\alpha'}}{\theta_j/\sigma-w}\exp\Big(-\frac{1}{2}(\theta_j/\sigma-w)^2\Big)\\
&\leq\Big(1+\frac{\sqrt{2\log (d-s)-\log\log (d-s)+C'}}{\sqrt{2\log s-\log\log s+C}}\Big)\frac{\exp(-C'/2)\sqrt{\log s}}{s}\rightarrow 0.
\end{align*}
Thus, by (\ref{eq_cor_minimax_onesided25}) we have
$$
\frac{\sup_{1\leq j\leq d}\sup_{\btheta\in \Theta^+(s,a)} \EE_{\btheta}(\theta_j-L_j)}{\sigma \sqrt{2\log s}}\leq 1+O(\frac{1}{\sqrt{\log s}})\rightarrow 1.
$$
Similarly, if $\bar u_{\alpha'}> w$ holds, then we have (\ref{eq_cor_minimax_onesided2311}). By the proof of (\ref{cor_minimax_onesided2_eq2}) with $\bar u_{\alpha'}$ defined in (\ref{eqhatu2}), we still arrive at 
$$
\frac{\sup_{1\leq j\leq d}\sup_{\btheta\in \Theta^+(s,a)} \EE_{\btheta}(\theta_j-L_j)}{\sigma \sqrt{2\log s}}\leq 1.
$$
This completes the proof of (\ref{cor_minimax_onesided2_eq4}).

\subsection{Proof of Theorem \ref{themadap}}

\begin{proof}[Proof of Theorem \ref{themadap}]
Our proof relies on the following two lemmas. 
\begin{lemma}\label{lemhats}
Under the same condition in Theorem \ref{themadap}, we have
$$
\sup_{\btheta\in \Theta^+(s,a)}\PP_{\btheta}(|\bar S^{ad}_{\alpha'}|\geq 2s)\leq \Big(\frac{C_2(d-s)}{sd}\Big)^{s},~~\sup_{\btheta\in \Theta^+(s,a)}\PP_{\btheta}(|\bar S^{ad}_{\alpha'}|\leq s/2)\leq \Big(\frac{C_4}{s}\Big)^{s/2},
$$
where $C_2$ and $C_4$ are two universal constants. 
\end{lemma}

\begin{lemma}\label{lemlength}
Under the same condition in Theorem \ref{themadap}, there exists a positive constant $C$ such that for  $(t, s,d)$ sufficiently large, 
$$
\sup_{1\leq j\leq d}\sup_{\btheta\in \Theta^+(s,a)}\EE_{\btheta}(\theta_j-\hat L_{j,t})^2\leq C\sigma^2\log t,
$$
where $s,a$ satisfy the scenario (B). 
\end{lemma}

First, the proof of Corollary \ref{cor_asym} implies $\sup_{\theta_j=0}\PP_{\btheta}(j\in \bar S^{ad}_{\alpha'})\leq 1-\delta$. Note that $2s\leq d$ implies $d-s\geq d/2$. By the monotonicity of $\log (x/\sqrt{\log x})$, we can show that for any $s,a$ satisfying the scenario (B) and $\theta_j\neq 0$, we have
\begin{align*}
\frac{\theta_j}{\sigma}\geq \frac{a}{\sigma}&\geq \sqrt{2\log (d-s)-\log\log (d-s)+C'}+\sqrt{2\log s-\log\log s+C'}\\
&\geq \sqrt{2\log (d/2)-\log\log (d/2)+C'}+\sqrt{2\log s-\log\log s+C'}\\
&\geq \sqrt{2\log (2d)-\log\log d+C'/2}+\sqrt{2\log s-\log\log s+C'},
\end{align*}
where $C'$ is a sufficiently large constant. Thus, we can show that
\begin{align}
\PP_{\btheta}( \supp(\btheta) \not\subseteq \bar S^{ad}_{\alpha'} )&\leq\sum_{j: \theta_j\neq 0} \PP_{\theta_j}\Big(X_j/\sigma\leq  \sqrt{2\log (\frac{2d}{(\alpha-\alpha')C_{d,\alpha-\alpha'}})}\Big)\nonumber\\
&=\sum_{j: \theta_j\neq 0} \PP_{\theta_j}\Big(\frac{X_j-\theta_j}{\sigma}\leq  \sqrt{2\log (\frac{2d}{(\alpha-\alpha')C_{d,\alpha-\alpha'}})}-\frac{\theta_j}{\sigma}\Big)\nonumber\\
&\leq \sum_{j: \theta_j\neq 0} \PP_{\theta_j}\Big(\frac{X_j-\theta_j}{\sigma}\leq -\sqrt{2\log s-\log\log s+C'}\Big)\leq\alpha',
\label{eqthemadap_1}
\end{align}
for $s$ sufficiently large, where the last step holds as $C'$ is a sufficiently large constant. Note that under the event $s/2\leq |\bar S^{ad}_{\alpha'}|\leq 2s$, by the definition of $\hat s$, we have $s/2\leq |\bar S^{ad}_{\alpha'}|< \hat s\leq 2|\bar S^{ad}_{\alpha'}|\leq 4s$. That is $s/2\leq \hat s\leq 4s$. Formally, the above argument and Lemma \ref{lemhats} imply 
\begin{align}
\sup_{\btheta\in \Theta^+(s,a)} \PP_{\btheta}(\btheta\notin\hat M^{ad}_{\alpha'})&\leq \sup_{\btheta\in \Theta^+(s,a)} \PP_{\btheta}(\btheta\notin \hat M^{ad}_{\alpha'}, s/2\leq \hat s\leq 4s)+\Big(\frac{C_2(d-s)}{sd}\Big)^{s}+\Big(\frac{C_4}{s}\Big)^{s/2}.\label{eqthemadap_2}
\end{align}
By the proof of Theorem \ref{thm_spraseCI}, we have
\begin{align}
&\PP_{\btheta}(\btheta\notin \hat M^{ad}_{\alpha'}, s/2\leq \hat s\leq 4s)\nonumber\\
&\leq s(1-\Phi(u_{\alpha',s/2}))+\sum_{j\notin \supp(\btheta)}\PP_0\big({X_j}/{\sigma}\geq \eta\big)+\PP_{\btheta}(\supp(\btheta) \not\subseteq \bar S^{ad}_{\alpha'} )\nonumber\\
&=s\Big(1-\Phi\Big(\sqrt{2\log (\frac{2s}{(\alpha-\alpha')C_{s,\alpha-\alpha'}})}\Big)\Big)+(d-s)(1-\Phi(\eta))+\PP_{\btheta}( \supp(\btheta)\not\subseteq \bar S^{ad}_{\alpha'} ).\label{eqthemadap_3}
\end{align}
where
$$
\eta=\sqrt{2\log (\frac{2d}{(\alpha-\alpha')C_{d,\alpha-\alpha'}})}.
$$ 
The first two terms in (\ref{eqthemadap_3}) are both upper bounded by $(\alpha-\alpha')/2$ for $s,d$ sufficiently large, see the proof of Corollary \ref{cor_asym}. The upper bound for the last term in (\ref{eqthemadap_3}) is shown in (\ref{eqthemadap_1}). Thus, from  (\ref{eqthemadap_2}) we obtain $\liminf_{d,s\rightarrow\infty}\inf_{\btheta\in \Theta^+(s,a)} \PP_{\btheta}(\btheta\in \hat M^{ad}_{\alpha'})\geq 1-\alpha$. This shows that the adaptive sparse confidence interval $\hat M^{ad}_{\alpha'}$ belongs to $\cM_+$.

Next, we are ready to show
$$
\limsup_{d,s\rightarrow\infty}\frac{R(\hat M^{ad}_{\alpha'}, \Theta^+(s,a))}{\sigma\sqrt{2\log s}}\leq 1.
$$
Note that 
\begin{equation}\label{eqthemadap1}
\EE_{\btheta}(\theta_j-\hat L_{j,\hat s})= \EE_{\btheta}(\theta_j-\hat L_{j,\hat s})I(\hat s\leq 4s)+ \EE_{\btheta}(\theta_j-\hat L_{j,\hat s})I(\hat s> 4s):=I_1+I_2.
\end{equation}
We further decompose $I_1$ as follows
$$
I_1=\EE_{\btheta}(\theta_j-\hat L_{j,\hat s})I(\hat s\leq 4s)I(X_j/\sigma\geq \eta)+\EE_{\btheta}(\theta_j-\hat L_{j,\hat s})I(\hat s\leq 4s)I(X_j/\sigma< \eta):=I_{11}+I_{12}.
$$
Under the event $\hat s\leq 4s$, it holds that
$$
u_{\alpha',\hat s}=\sqrt{2\log \Big(\frac{4\hat s}{(\alpha-\alpha')C_{2\hat s,\alpha-\alpha'}}\Big)}\leq \sqrt{2\log \Big(\frac{16s}{(\alpha-\alpha')C_{8s,\alpha-\alpha'}}\Big)}:=u_{\alpha',4s}.
$$
Thus, for $I_{11}$, we have
\begin{align*}
I_{11}&=\EE_{\btheta} [\theta_j-(X_j- u_{\alpha',\hat s}\sigma)_+]I(\hat s\leq 4s)I(X_j/\sigma\geq \eta)\\
&\leq \EE_{\btheta} [\theta_j-(X_j- u_{\alpha',4s}\sigma)_+]I(\hat s\leq 4s)I(X_j/\sigma\geq \eta)\\
&=\EE_{\btheta} [\theta_j-(X_j- u_{\alpha',4s}\sigma)_+]I(\hat s\leq 4s)I(\theta_j>(X_j- u_{\alpha',4s}\sigma)_+)I(X_j/\sigma\geq \eta)
\\
&~~~+\EE_{\btheta} [\theta_j-(X_j- u_{\alpha',4s}\sigma)_+]I(\hat s\leq 4s)I(\theta_j\leq(X_j- u_{\alpha',4s}\sigma)_+)I(X_j/\sigma\geq \eta)\\
&\leq \EE_{\btheta} [\theta_j-(X_j- u_{\alpha',4s}\sigma)_+]I(\hat s\leq 4s)I(\theta_j>(X_j- u_{\alpha',4s}\sigma)_+)I(X_j/\sigma\geq \eta)\\
&\leq \EE_{\btheta} [\theta_j-X_j+ u_{\alpha',4s}\sigma]I(\theta_j>(X_j- u_{\alpha',4s}\sigma)_+)I(X_j> u_{\alpha',4s}\sigma)I(X_j/\sigma\geq \eta)\\
&~~~+\theta_j\EE_{\btheta}I(\theta_j>(X_j- u_{\alpha',4s}\sigma)_+)I(X_j\leq u_{\alpha',4s}\sigma)I(X_j/\sigma\geq \eta).
\end{align*}
Following the proof of Corollary \ref{cor_minimax_onesided2}, we can further show that
\begin{align*}
&\EE_{\btheta} [\theta_j-X_j+ u_{\alpha',4s}\sigma]I(\theta_j>(X_j- u_{\alpha',4s}\sigma)_+)I(X_j> u_{\alpha',4s}\sigma)I(X_j/\sigma\geq \eta)\\
&\leq \EE_{\btheta} [\theta_j-X_j]I(\theta_j>(X_j- u_{\alpha',4s}\sigma)_+)I(X_j> u_{\alpha',4s}\sigma)I(X_j/\sigma\geq \eta)+ u_{\alpha',4s}\sigma\PP_{\btheta}(X_j> u_{\alpha',4s}\sigma)\\
&\leq \sigma+u_{\alpha',4s}\sigma. 
\end{align*}
After some simple calculation similar to the proof of Corollary \ref{cor_minimax_onesided2}, we can show that by the tail bound in Lemma \ref{lemtail},
$$
\theta_j\EE_{\btheta}I(\theta_j>(X_j- u_{\alpha',4s}\sigma)_+)I(X_j\leq u_{\alpha',4s}\sigma)I(X_j/\sigma\geq \eta)\leq \theta_j\PP_{\btheta}(X_j\leq u_{\alpha',4s}\sigma)\leq C\sigma,
$$
for any $\theta_j\neq 0$, where $C$ is a positive constant. Combining the above inequalities, we obtain
\begin{equation}\label{eqthemadap2}
I_{11}\leq (1+C)\sigma+u_{\alpha',4s}\sigma.
\end{equation}
For $I_{12}$, recall that it suffices to only consider nonzero $\theta_j$. The tail bound in Lemma \ref{lemtail} leads to
\begin{align*}
I_{12}&\leq \theta_j \PP_{\btheta}(X_j/\sigma< \eta)= \theta_j \PP(Z/\sigma< -(\theta_j/\sigma-\eta))\\
&\leq C\sigma (1+\frac{\eta}{\theta_j/\sigma-\eta})\exp(-\frac{1}{2}\bar\xi_d^2)\\
&\leq C\sigma \Big(1+\frac{\sqrt{2\log d-\log\log d+C''}}{\sqrt{2\log s-\log\log s+C'}}\Big){\frac{\sqrt{\log s}}{\log (d-s)}}\leq 2C\sigma.
\end{align*}

In the following, we consider $I_2$. Let $t_0=\{t\in [T]: 2^{t-1}\leq 4s<2^{t}\}$. Then
\begin{align*}
I_2&=\sum_{t=t_0}^T  \EE_{\btheta}(\theta_j-\hat L_{j, 2^t})I(\hat s=2^t)\leq \sum_{t=t_0}^T  \{\EE_{\btheta}(\theta_j-\hat L_{j, 2^t})^2\}^{1/2}\{\PP_{\btheta}(\hat s=2^t)\}^{1/2}.
\end{align*}
For any $t$, $\hat s=2^t$ implies $2^{t-1}\leq  |\bar S^{ad}_{\alpha'}|<2^{t}$. Thus,
\begin{align*}
\PP_{\btheta}(\hat s=2^t)&\leq \PP_{\btheta}(2^{t-1}\leq  |\bar S^{ad}_{\alpha'}|<2^{t})\leq \PP_{\btheta}(|\bar S^{ad}_{\alpha'}|\geq 2^{t-1}).
\end{align*}
Recall that $2^{t_0-1}>2s$. Following the proof of Lemma \ref{lemhats}, we have for any $t\geq t_0$
$$
\PP_{\btheta}(\hat s=2^t)\leq {d-s\choose 2^{t-1}-s} \Big(\frac{C}{d}\Big)^{2^{t-1}-s}\leq \Big(\frac{C'}{2^{t-1}-s}\Big)^{2^{t-1}-s},
$$
where $C, C'$ are positive constants. In addition, Lemma \ref{lemlength} implies
$$
\EE_{\btheta}(\theta_j-\hat L_{j, 2^t})^2\leq C(\log 2)\sigma^2t.
$$
Thus, for $I_2$, there exists a constant $C>0$ such that
\begin{align*}
I_2&\leq \sigma\sum_{t=t_0}^T t^{1/2}\Big(\frac{C}{2^{t-1}-s}\Big)^{\frac{2^{t-1}-s}{2}}\leq \sigma\sum_{t=t_0}^T \frac{Ct^{1/2}}{2^{t-1}-2^{t_0-2}}=\sigma\sum_{q=0}^{T-t_0} \frac{C(t_0+q)^{1/2}}{2^{q+t_0-1}-2^{t_0-2}}\\
&\leq \frac{C\sigma t_0^{1/2}}{2^{t_0-2}}\sum_{q=0}^{T-t_0}\frac{1}{2^{q+1}-1}+\frac{C\sigma}{2^{t_0-2}}\sum_{q=0}^{T-t_0}\frac{q^{1/2}}{2^{q+1}-1}.
\end{align*}
It is easily seen that the infinite sum $\sum_{q=0}^{\infty}\frac{1}{2^{q+1}-1}$ and $\sum_{q=0}^{\infty}\frac{q^{1/2}}{2^{q+1}-1}$ converges. Thus, there exists a constant $C'>0$ such that
$$
I_2\leq \frac{C'\sigma t_0^{1/2}}{2^{t_0-2}}+\frac{C'\sigma}{2^{t_0-2}}.
$$
Combining with (\ref{eqthemadap1}) and (\ref{eqthemadap2}), we obtain
$$
\EE_{\btheta}(\theta_j-\hat L_{j,\hat s})\leq (1+3C)\sigma+u_{\alpha',4s}\sigma+\frac{C'\sigma t_0^{1/2}}{2^{t_0-2}}+\frac{C'\sigma}{2^{t_0-2}}.
$$
Noting that
$$
\lim_{d,s\rightarrow\infty} \frac{u_{\alpha',4s}}{ \sqrt{2\log s}}= \lim_{d,s\rightarrow\infty} \frac{\sqrt{2\log s-\log\log s+C}}{\sqrt{2\log s}}= 1,
$$
we complete the proof. 
\end{proof}

\begin{proof}[Proof of Lemma \ref{lemhats}]
Consider the event $|\bar S^{ad}_{\alpha'}|\geq 2s$. It implies that there exist at least $s$ number of $\theta_j$ such that $\theta_j=0$ and $X_j/\sigma\geq \sqrt{2\log (\frac{2d}{(\alpha-\alpha')C_{d,\alpha-\alpha'}})}$. Thus, we have uniformly over $\btheta\in \Theta^+(s,a)$,
\begin{align*}
\PP_{\btheta}(|\bar S^{ad}_{\alpha'}|\geq 2s)&\leq {d-s\choose s} \Big[\PP_{\theta_j=0}\Big(X_j/\sigma\geq \sqrt{2\log (\frac{2d}{(\alpha-\alpha')C_{d,\alpha-\alpha'}})}\Big)\Big]^{s}\\
&\leq \Big(\frac{(d-s)e}{s}\Big)^{s}\Big[\sqrt{\frac{2}{\pi}}\frac{1}{2\sqrt {2\log (\frac{2d}{C_{d,\alpha-\alpha'}(\alpha-\alpha')})}}\exp\Big(-\log (\frac{2d}{C_{d,\alpha-\alpha'}(\alpha-\alpha')})\Big)\Big]^{s}\\
&\leq \Big(\frac{(d-s)e}{s}\Big)^{s}\Big[\frac{C_1d^{-1}\sqrt{\log d}}{\sqrt {\log (\frac{2d}{C_{d,\alpha-\alpha'}})}}\Big]^{s}\\
&\leq \Big(\frac{C_2(d-s)}{ sd}\Big)^{s},
\end{align*}
for $d$ large enough, where $C_1, C_2>0$ are two universal constants. Consider the event $|\bar S^{ad}_{\alpha'}|\leq s/2$. It implies that there exist at least $s/2$ number of $\theta_j$ such that $\theta_j>0$ and $j\notin \bar S^{ad}_{\alpha'}$. Following the similar argument and the inequality (\ref{eq_cor_asym11}),  we can show that
\begin{align*}
\PP_{\btheta}(|\bar S^{ad}_{\alpha'}|\leq s/2)&\leq {s\choose s/2}\Big[\PP_{\theta_j}\Big(X_j/\sigma\leq \sqrt{2\log (\frac{2d}{(\alpha-\alpha')C_{d,\alpha-\alpha'}})}\Big)\Big]^{s/2}\\
&\leq {s\choose s/2} \Big[\PP\Big(N \leq -\sqrt{2\log s-\log\log s+C'}\Big)\Big]^{s/2}\\
&\leq (2e)^{s/2}\Big[\frac{C_3s^{-1}\sqrt{\log s}}{\sqrt{2\log s-\log\log s+C'}}\Big]^{s/2}\\
&\leq \Big(\frac{C_4}{s}\Big)^{s/2},
\end{align*}
for $d, s$ large enough, where $N\sim N(0,1)$ and $C_3, C_4>0$ are two universal constants. 
\end{proof}

\begin{proof}[Proof of Lemma \ref{lemlength}]
Following the proof of Corollary \ref{cor_minimax_onesided2}, we consider two cases $d\geq 2t$ and $d<2t$. When the former condition holds,  we can show that 
\begin{align}
\EE_{\btheta}(\theta_j-\hat L_{j,t})^2&=\EE_{\btheta}(\theta_j-X_j+u_{\alpha',t}\sigma)^2I(X_j/\sigma>w)+\theta_j^2\PP_{\btheta}(X_j/\sigma\leq w)\nonumber\\
&=\EE_{\btheta}(\theta_j-X_j)^2I(X_j/\sigma>w)+u^2_{\alpha',t}\sigma^2 \PP_{\btheta}(X_j/\sigma>w)\nonumber\\
&~~~~+2u_{\alpha',t}\sigma \EE_{\btheta}(\theta_j-X_j)I(X_j/\sigma>w)+\theta_j^2\PP_{\btheta}(X_j/\sigma\leq w)\nonumber\\
&\leq \sigma^2+u^2_{\alpha',t}\sigma^2+2u_{\alpha',t}\sigma^2+\theta_j^2\PP_{\btheta}(X_j/\sigma\leq w),\label{eqlemlength1}
\end{align}
where $w=\sqrt{2\log (\frac{2d}{(\alpha-\alpha')C_{d,\alpha-\alpha'}})}$. Furthermore, the tail bound in Lemma \ref{lemtail} implies
\begin{align*}
\theta_j^2\PP_{\btheta}(X_j/\sigma\leq w)&\leq \frac{\sigma^2}{\sqrt{2\pi}}\frac{\theta_j^2/\sigma^2}{\theta_j/\sigma-w}\exp\Big(-\frac{1}{2}(\theta_j/\sigma-w)^2\Big)\\
&\leq \frac{\sigma^2}{\sqrt{2\pi}}\frac{2(\theta_j/\sigma-w)^2+2w^2}{\theta_j/\sigma-w}\exp\Big(-\frac{1}{2}(\theta_j/\sigma-w)^2\Big).
\end{align*} 
Recall that if $d,s$ large enough, we have
$$
\theta_j/\sigma-w\geq a/\sigma-w\geq \sqrt{2\log s-\log\log s+C'}\vee \bar\xi_d,
$$ 
and $w^2\leq 2\log d-\log\log d+C'$ for some constant $C'$. Thus, for $d,s$ large enough, 
\begin{align*}
\theta_j^2\PP_{\btheta}(X_j/\sigma\leq w)
&\leq \frac{2\sigma^2}{\sqrt{2\pi}}\sqrt{2\log s-\log\log s+C'}\exp\Big(-\frac{1}{2}(2\log s-\log\log s+C')^2\Big)\\
&~~+\frac{2\sigma^2}{\sqrt{2\pi}}\frac{2\log d-\log\log d+C'}{\sqrt{2\log s-\log\log s+C'}}\exp\Big(-\frac{1}{2} \bar\xi_d^2\Big)\\
&\leq \sigma^2+\frac{2\sigma^2}{\sqrt{2\pi}}\frac{2\log d-\log\log d+C'}{\sqrt{2\log s-\log\log s+C'}}\frac{\sqrt{\log s}}{\log (d-s)}\\
&\leq \sigma^2+\frac{2\sigma^2}{\sqrt{2\pi}}\frac{2\log d-\log\log d+C'}{\sqrt{2\log s-\log\log s+C'}}\frac{\sqrt{\log s}}{\log (d/2)} \leq C''\sigma^2,
\end{align*} 
where $C''$ is a positive constant. Finally, we plug into the inequality (\ref{eqlemlength1}), 
$$
\EE_{\btheta}(\theta_j-\hat L_{j,t})^2\leq \sigma^2+u^2_{\alpha',t}\sigma^2+2u_{\alpha',t}\sigma^2+C''\sigma^2.
$$
For $(t, s,d)$ sufficiently large, it can be easily verified that the following inequality holds
\begin{align}\label{eqlemlength2}
\sup_{1\leq j\leq d}\sup_{\btheta\in \Theta^+(s,a)}\EE_{\btheta}(\theta_j-\hat L_{j,t})^2\leq C\sigma^2\log t.
\end{align} 
In the following, we consider the second case $d<2t$. Similar to (\ref{eqlemlength1}), we obtain that
\begin{align*}
\EE_{\btheta}(\theta_j-\hat L_{j,t})^2&=\EE_{\btheta}(\theta_j-X_j+u_{\alpha',t}\sigma)^2I(X_j/\sigma>u_{\alpha',t})+\theta_j^2\PP_{\btheta}(X_j/\sigma\leq u_{\alpha',t})\\
&=\EE_{\btheta}(\theta_j-X_j)^2I(X_j/\sigma>u_{\alpha',t})+u^2_{\alpha',t}\sigma^2 \PP_{\btheta}(X_j/\sigma>u_{\alpha',t})\\
&~~~~+2u_{\alpha',t}\sigma \EE_{\btheta}(\theta_j-X_j)I(X_j/\sigma>u_{\alpha',t})+\theta_j^2\PP_{\btheta}(X_j/\sigma\leq u_{\alpha',t})\\
&\leq \sigma^2+u^2_{\alpha',t}\sigma^2\PP_{\btheta}(X_j/\sigma>u_{\alpha',t})+2u_{\alpha',t}\sigma^2+\theta_j^2\PP_{\btheta}(X_j/\sigma\leq u_{\alpha',t}).
\end{align*}
If $2u_{\alpha',t}\sigma\geq \theta_j$, then $\EE_{\btheta}(\theta_j-\hat L_{j,t})^2\leq \sigma^2+5u^2_{\alpha',t}\sigma^2+2u_{\alpha',t}\sigma^2$. If $2u_{\alpha',t}\sigma< \theta_j$, then $\theta_j/\sigma-u_{\alpha',t}>\theta_j/(2\sigma)$, and we further have
\begin{align*}
\EE_{\btheta}(\theta_j-\hat L_{j,t})^2&\leq (1+u_{\alpha',t})^2\sigma^2+(\theta_j^2-u^2_{\alpha',t}\sigma^2)\PP(Z> \theta_j/\sigma-u_{\alpha',t})\\
&\leq (1+u_{\alpha',t})^2\sigma^2+\frac{\sigma^2(\theta_j/\sigma)}{\sqrt{2\pi}}\exp\Big(-\frac{1}{2}(\theta_j/(2\sigma))^2\Big)\\
&\leq (1+u_{\alpha',t})^2\sigma^2+\frac{\sigma^2}{\sqrt{2\pi}},
\end{align*}
as $s,d$ tend to infinity, where $Z\sim N(0,1)$. For $(t, s,d)$ sufficiently large, (\ref{eqlemlength2}) holds as well for the second case $d<2t$. This completes the proof. 
\end{proof}

\subsection{Proof of Theorem \ref{them2s_uplow}}
Denote $A^+=\{\btheta\in\RR^d: \|\btheta\|_0= s, \theta_j=a ~\textrm{for}~ \forall j, \theta_j\neq 0\}$ and $A^-=\{\btheta\in\RR^d: \|\btheta\|_0= s, \theta_j=-a ~\textrm{for}~ \forall j, \theta_j\neq 0\}$. Then
$$
\sup_{\btheta\in\Theta(s,a)}\PP_{\btheta}( \supp(\btheta) \not\subseteq \hat S )\geq \sup_{\btheta\in A^+\cup A^-}\PP_{\btheta}( \supp(\btheta) \not\subseteq \hat S )=1-\inf_{\btheta\in A^+\cup A^-}\PP_{\btheta}( \supp(\btheta) \subseteq \hat S ).
$$
Following the proof of Theorem \ref{themlower}, we obtain that
\begin{align*}
\sup_{\btheta\in\Theta(s,a)}\PP_{\btheta}( S \not\subseteq \hat S )&\geq \frac{1}{2}\Big[\sup_{\btheta\in A^+}\PP_{\btheta}( S \not\subseteq \hat S )+\sup_{\btheta\in A^-}\PP_{\btheta}( S \not\subseteq \hat S )\Big]\\
&\geq \frac{1}{2|A^+|}\sum_{\btheta\in A^+} \PP_{\btheta}( S \not\subseteq \hat S )+\frac{1}{2|A^-|}\sum_{\btheta\in A^-} \PP_{\btheta}( S \not\subseteq \hat S )\\
&\geq t\sum_{k=1}^s {s\choose k} \frac{u_+^k+u_-^k}{2},
\end{align*}
where $t=\inf_{\btheta\in A^+\cup A^-}\prod_{j\in S }\PP_{\theta_j}(j\in\hat S )$, and $u_+=\PP_a(j\notin\hat S )$ and $u_-=\PP_{-a}(j\notin\hat S )$, where $\PP_a$ denotes the probability of $X_j\sim N(a,\sigma^2)$. Applying Jensen's inequality, the above display can be further bounded from below which yields 
$$
\sup_{\btheta\in\Theta(s,a)}\PP_{\btheta}( S \not\subseteq \hat S )\geq t\sum_{k=1}^s {s\choose k} \Big(\frac{u_++u_-}{2}\Big)^k=t\Big[(1+\frac{u_++u_-}{2})^s-1\Big].
$$
We denote $j\in\hat S $ by $T(X_j)=1$ for some function $T(\cdot)$. The Neyman-Pearson lemma implies that the infimum of $(u_++u_-)/2=\PP_a(T(X_j)=0)/2+\PP_{-a}(T(X_j)=0)/2$ over all possible $T(\cdot)$ such that $\PP_0(T(X_j)=1)\leq 1-\delta$ is attained by the likelihood ratio test of $X_j\sim N(0,\sigma^2)$ versus the mixture normal $X_j\sim \frac{1}{2}N(a,\sigma^2)+\frac{1}{2}N(-a,\sigma^2)$, which is
\begin{equation}\label{eq_pf_them2s_uplow}
T(X)=I\Big(\cosh(aX/\sigma^2) \geq c^*\exp(\frac{a^2}{2\sigma^2})\Big),
\end{equation}
where $\cosh(x)=(\exp(x)+\exp(-x))/2$ and $c^*$ is chosen such that $\PP_0(T(X_j)=0)= \delta$. Since $\cosh(x)$ is symmetric and monotonically increasing for $x>0$, we have
$$
\delta=\PP_0\Big(|X/\sigma|\leq \frac{\sigma}{a}\cosh^{-1}(c^*\exp(\frac{a^2}{2\sigma^2}))\Big)=1-2\Phi\Big(-\frac{\sigma}{a}\cosh^{-1}(c^*\exp(\frac{a^2}{2\sigma^2}))\Big).
$$
Solving above equation, we obtain
$$
c^*=\exp(-\frac{a^2}{2\sigma^2})\cosh\Big(\frac{a}{\sigma}\Phi^{-1}(\frac{1+\delta}{2})\Big).
$$
Denote $\Delta_{TS}=\frac{u_++u_-}{2}=\PP_a(T(X_j)=0)/2+\PP_{-a}(T(X_j)=0)/2$ with $T(X)$ defined in (\ref{eq_pf_them2s_uplow}). Then,
\begin{align*}
\Delta_{TS}&=\frac{1}{2}\PP_a\Big(\cosh(aX/\sigma^2) \leq c^*\exp(\frac{a^2}{2\sigma^2})\Big)+\frac{1}{2}\PP_{-a}\Big(\cosh(aX/\sigma^2) \leq c^*\exp(\frac{a^2}{2\sigma^2})\Big)\\
&=\frac{1}{2}\PP_a\Big(|X/\sigma|\leq \Phi^{-1}(\frac{1+\delta}{2})\Big)+\frac{1}{2}\PP_{-a}\Big(|X/\sigma|\leq \Phi^{-1}(\frac{1+\delta}{2})\Big)\\
&=\Phi(\Phi^{-1}(\frac{1+\delta}{2})+\frac{a}{\sigma})-\Phi(-\Phi^{-1}(\frac{1+\delta}{2})+\frac{a}{\sigma}).
\end{align*}
Following the same steps in the proof of Theorem \ref{themlower}, we can obtain (\ref{low_TS}). 

To show (\ref{low2_TS}), we consider the following two cases separately.  Case (1): $a/\sigma\leq \Phi^{-1}(\frac{1+\delta}{2})$. Then
$$
\Delta_{TS}\geq \Phi\Big(\Phi^{-1}(\frac{1+\delta}{2})\Big)-\Phi(0)\geq \frac{c}{2},
$$
since $\delta\geq c$ for some constant $c>0$. As a result, $(1+\Delta_{TS})^s\rightarrow\infty$ as $s\rightarrow\infty$. This yields (\ref{low2_TS}). 

Case (2): $\Phi^{-1}(\frac{1+\delta}{2})<a/\sigma\leq \Phi^{-1}(\frac{1+\delta}{2})-\Phi^{-1}(c_s/s)$. Denote $g(x)=\Phi(\Phi^{-1}(\frac{1+\delta}{2})+x)-\Phi(-\Phi^{-1}(\frac{1+\delta}{2})+x)$. We have $\Delta_{TS}=g(a/\sigma)$. Note that the function $g(x)$ is monotonically decreasing for $x\geq \Phi^{-1}(\frac{1+\delta}{2})$. This implies that 
\begin{align}
\Delta_{TS}&\geq g\Big(\Phi^{-1}(\frac{1+\delta}{2})-\Phi^{-1}(c_s/s)\Big)\nonumber\\
&= \Phi\Big(2\Phi^{-1}(\frac{1+\delta}{2})+\Phi^{-1}(1-c_s/s)\Big)-\Phi(\Phi^{-1}(1-c_s/s))\nonumber\\
&\geq \Phi(c'+T(s/c_s))-(1-c_s/s),\label{eq_pf_them2s_uplow2}
\end{align}
where $c'=2\Phi^{-1}(\frac{1+c}{2})$ and $T(x)=\sqrt{2\log x-\log\log x-C}$ and the last step holds by Lemma \ref{lemtail}. Applying Lemma \ref{lemtail} again yields
\begin{align}
\Phi(c'+T(s/c_s))&\geq 1-\exp\Big(-\frac{1}{2}(c'+T(s/c_s))^2\Big)\nonumber\\
&\geq 1-C\exp\Big(-\log (s/c_s)+\frac{1}{2}\log\log (s/c_s)-c'\sqrt{\log (s/c_s)}\Big)\nonumber\\
&=1-C\frac{c_s}{s} \exp\Big(\frac{1}{2}\log\log (s/c_s)-c'\sqrt{\log (s/c_s)}\Big),\label{eq_pf_them2s_uplow3}
\end{align}
where $C$ is a generic constant which may differ from line to line. 
Since $s/c_s$ tends to infinity, $\log\log (s/c_s)\ll \sqrt{\log (s/c_s)}$. Combining (\ref{eq_pf_them2s_uplow2}) and (\ref{eq_pf_them2s_uplow3}), we obtain
$$
\Delta_{TS}\geq \frac{c_s}{s} \Big[1-C \exp\Big(\frac{1}{2}\log\log (s/c_s)-c'\sqrt{\log (s/c_s)}\Big)\Big]\geq \frac{c_s}{2s}.
$$
Following the same steps in the proof of Theorem \ref{themlower}, we can obtain $(\Delta_{TS}+1)^s\rightarrow\infty$. This completes the proof of (\ref{low2_TS}).

To show $\hat S^{TS}_{\alpha'}\in\cF(\delta)$, notice that
$$
\PP_0(j\in\hat S^{TS}_{\alpha'} )\leq \PP_0(|X_j|/\sigma\geq \Phi^{-1}(1/2+\delta/2))=1-\delta.
$$
The event $ \supp(\btheta)\not\subseteq \hat S^{TS}_{\alpha'} $ is equivalent to that there exists $j\in[d]$ such that $j\in \supp(\btheta) $ and $j\notin\hat S^{TS}_{\alpha'} $. Then
\begin{align*}
\PP_{\btheta}( \supp(\btheta) \not\subseteq \hat S^{TS}_{\alpha'} )&=\PP_{\btheta}(\exists j\in [d], j\in \supp(\btheta) , j\notin\hat S^{TS}_{\alpha'} )\\
&\leq\sum_{j: \theta_j\neq 0} \PP_{\theta_j}(j\notin\hat S^{TS}_{\alpha'} )\\
&=\sum_{j: \theta_j\neq 0} \PP(|X_j|\leq (\sigma\Phi^{-1}(\frac{\alpha'}{2s})+a)_+),
\end{align*}
where the last step holds since $a/\sigma\geq \Phi^{-1}(\frac{\delta+1}{2})-\Phi^{-1}(\frac{\alpha'}{2s})$. 
If $\sigma\Phi^{-1}(\frac{\alpha'}{2s})+a<0$, the above probability is 0. Otherwise, 
$$
\PP(  \supp(\btheta) \not\subseteq \hat S^{TS}_{\alpha'}  ) \leq \sum_{j: \theta_j\neq 0} \PP(|\theta_j/\sigma|-|Z_j|\leq \Phi^{-1}(\frac{\alpha'}{2s})+a/\sigma)\leq \sum_{j: \theta_j\neq 0} \PP(|Z_j|\geq -\Phi^{-1}(\frac{\alpha'}{2s}))=\alpha',
$$
where $Z_j=\frac{X_j-\theta_j}{\sigma}\sim N(0,1)$, and we use $\min_{j: \theta_j\neq 0} |\theta_j|\geq a$.  This completes the proof.

\subsection{Proof of Theorem \ref{thm_spraseCI2}}
Similar to the proof of Theorem \ref{thm_spraseCI}, we can bound $\PP(\btheta\notin \hat M^{TS}_{\alpha'})$ by 
\begin{align}
\PP(\btheta\notin \hat M^{TS}_{\alpha'})&\leq \PP(\exists j\in \hat S^{TS}_{\alpha'} , \theta_j\notin [\hat L_j, \hat U_j])+ \PP( \supp(\btheta) \not\subseteq \hat S^{TS}_{\alpha'}  ).\label{eqpfthm_spraseCI20}
\end{align}
By part (3) of Theorem  \ref{them2s_uplow}, $\PP( \supp(\btheta) \not\subseteq \hat S^{TS}_{\alpha'}  )\leq \alpha'$. The first term can be further bounded as follows
\begin{align}
&\PP(\exists j\in \hat S^{TS}_{\alpha'} , \theta_j\notin [\hat L_j, \hat U_j])\\
&\leq \PP(\exists j\in \supp(\btheta)  , \theta_j\notin [\hat L_j,\hat U_j])+\PP(\exists j\in \hat S^{TS}_{\alpha'} \backslash  \supp(\btheta) , \theta_j \notin [\hat L_j, \hat U_j]):=I_1+I_2.\label{eqpfthm_spraseCI21}
\end{align}
For $I_1$, by noting that $\theta_j\notin [\hat L_j,\hat U_j]$ is equivalent to $|Z_j|\geq u$ where $Z_j=\frac{X_j-\theta_j}{\sigma}\sim N(0,1)$ and $u=\hat u^{TS}_{\alpha'}$, we have 
\begin{equation}
I_1\leq \sum_{j\in  \supp(\btheta) } \PP(|Z_j|\geq u)=2s (1-\Phi(u)). \label{eqpfthm_spraseCI22}
\end{equation}
To bound $I_2$, noting that $j\notin   \supp(\btheta) $ implying $\theta_j=0$, we have
$$
I_2=\PP(\exists j\notin  \supp(\btheta) , |Z_j|\geq \Phi^{-1}(\frac{\alpha'}{2s})+\frac{a}{\sigma}, |Z_j|\geq u)\leq \sum_{j\notin   \supp(\btheta) }\PP( |Z_j|\geq \Phi^{-1}(\frac{\alpha'}{2s})+\frac{a}{\sigma}, |Z_j|\geq u).
$$
To bound the last probability, we now consider the following two cases. 

(1). When $\frac{a}{\sigma}\leq -\Phi^{-1}(\frac{\alpha-\alpha'}{2d})-\Phi^{-1}(\frac{\alpha'}{2s})$, by setting $u=\Phi^{-1}(1-\frac{\alpha-\alpha'}{2d})$, we have
$I_2\leq 2(d-s)(1-\Phi(u))$. Combining with (\ref{eqpfthm_spraseCI20}), (\ref{eqpfthm_spraseCI21}), (\ref{eqpfthm_spraseCI22}), we have
$$
\PP(\btheta\notin \hat M^{TS}_{\alpha'})\leq 2d(1-\Phi(u))+\alpha'= \alpha.
$$

(2). When $\frac{a}{\sigma}>-\Phi^{-1}(\frac{\alpha-\alpha'}{2d})-\Phi^{-1}(\frac{\alpha'}{2s})$, we have $\eta>1-\frac{\alpha-\alpha'}{2d}$, where $\eta=\Phi(\frac{a}{\sigma}+\Phi^{-1}(\frac{\alpha'}{2s}))$. Recall that $u=\Phi^{-1}(1-\frac{\alpha-\alpha'-2(d-s)(1-\eta)}{2s})$. Then
$$
\Phi(u)=1-\frac{\alpha-\alpha'-2(d-s)(1-\eta)}{2s}<\eta,
$$
which is equivalent to $\Phi^{-1}(\frac{\alpha'}{2s})+\frac{a}{\sigma}\geq u$. Therefore, $I_2\leq 2(d-s)(1-\eta)$, and finally we have
$$
\PP(\btheta\notin \hat M^{TS}_{\alpha'})\leq 2(d-s)(1-\eta)+2s(1-\Phi(u))+\alpha'= \alpha.
$$

\subsection{Proof of Theorem \ref{thm_minimax_twosided}}

The idea of the proof is very similar to the proof of Theorem \ref{thm_minimax_onesided}. For simplicity of presentation, we skip some intermediate steps. Denote $A^+=\{\btheta\in\RR^d: \|\btheta\|_0=A, \theta_j=\rho, ~\textrm{for any $\theta_j\neq 0$}\}$ and $A^-=\{\btheta\in\RR^d: \|\btheta\|_0=A, \theta_j=-\rho, ~\textrm{for any $\theta_j\neq 0$}\}$, where $0<A\leq s$ and $\rho$ is an arbitrary positive quantity that is $\rho\geq a$. Then
\begin{align}
\sup_{\btheta\in \Theta(s,a)} \PP_{\btheta}(\btheta\notin M)&\geq \frac{1}{2}{d\choose A}^{-1}\Big[\sum_{\btheta\in A^+} \PP_{\btheta}( \btheta\notin M)+\sum_{\btheta\in A^-} \PP_{\btheta}( \btheta\notin M)\Big]\nonumber\\
&\geq \frac{t}{2}{d\choose A}^{-1}\Big[\sum_{j=1}^d \sum_{\btheta\in A^+} \PP_{\theta_j}(\theta_j\notin CI_j )+\sum_{j=1}^d \sum_{\btheta\in A^-} \PP_{\theta_j}(\theta_j\notin CI_j )\Big]\nonumber\\
&=t\sum_{j=1}^d  \Big(\frac{d-A}{d} \PP_{0}(0\notin CI_j )+\frac{A}{2d}\PP_{\rho}(\rho\notin CI_j )+\frac{A}{2d}\PP_{-\rho}(-\rho\notin CI_j )\Big),\label{eq_thm_minimax_twosided1}
\end{align}
where $t=\inf_{\btheta\in A^+\cup A^-}\prod_{j=1}^d\PP_{\theta_j}(\theta_j\in CI_j)$. Since $\PP_{\rho}(0,\rho\in CI_j )\leq \EE_{\rho}|U_j-L_j|/\rho$, we have
$$
\PP_{\rho}(\rho\notin CI_j )\geq \PP_{\rho}(0 \in CI_j )-\PP_{\rho}(0,\rho\in CI_j )\geq \PP_{\rho}(0 \in CI_j )-m/\rho,
$$
and similarly $\PP_{-\rho}(-\rho\notin CI_j )\geq \PP_{-\rho}(0 \in CI_j )-m/\rho$. Together with (\ref{eq_thm_minimax_twosided1}), 
\begin{align}
\sup_{\btheta\in \Theta(s,a)} \PP_{\btheta}(\btheta\notin M) &\geq t\Big\{\sum_{j=1}^d  \Big(\frac{d-A}{d} \PP_{0}(0\notin CI_j )+\frac{A}{2d}\PP_{\rho}(0 \in CI_j )+\frac{A}{2d}\PP_{-\rho}(0 \in CI_j )\Big)-\frac{Am}{\rho}\Big\}\nonumber\\
& \geq tA\Big\{\inf_{T\in\{0,1\}}  \Big(\frac{d-A}{A} \EE_{0}(1-T)+\frac{1}{2}\EE_{\rho}T+\frac{1}{2}\EE_{-\rho}T\Big)-\frac{m}{\rho}\Big\},\label{eq_thm_minimax_twosided2}
\end{align}
where $T$ denotes a test function from $\RR$ to $\{0,1\}$. Note that  
$$
\frac{d-A}{A} \EE_{0}(1-T)+\frac{1}{2}\EE_{\rho}T+\frac{1}{2}\EE_{-\rho}T=\EE_0\Big(\frac{d-A}{A} +T(\frac{f_\rho(x)+f_{-\rho}(x)}{2f_0(x)}-\frac{d-A}{A})\Big),
$$
where $f_\rho(x)$ denotes the pdf of $N(\rho,\sigma^2)$. Thus, the above function is minimized by
$$
T^*(x)=I(\frac{f_\rho(x)+f_{-\rho}(x)}{2f_0(x)}-\frac{d-A}{A}\leq 0)=I\Big(|x/\sigma|\leq \frac{\sigma}{\rho}\cosh^{-1}(\frac{d-A}{A}\exp(\frac{\rho^2}{2\sigma^2}))\Big),
$$
where $\cosh^{-1}(\cdot)$ is the inverse function of $\cosh$. Plugging the definition of $T^*(x)$ into (\ref{eq_thm_minimax_twosided2}), after some calculation we obtain
$$
\sup_{\btheta\in \Theta(s,a)} \PP_{\btheta}(\btheta\notin M) \geq tA\Big\{\frac{2(d-A)}{A}\Phi(-D)+\Phi(\frac{\rho}{\sigma}+D)-\Phi(\frac{\rho}{\sigma}-D)-\frac{m}{\rho}\Big\},
$$
where $D=\frac{\sigma}{\rho}\cosh^{-1}(\frac{d-A}{A}\exp(\frac{\rho^2}{2\sigma^2}))$. The rest of the proof is the same as Theorem \ref{thm_minimax_onesided}. We omit the details. 

\subsection{Proof of Corollary \ref{cor_minimax_twosided2}}

The proof of this corollary follows from the same line as in the proof of Corollary \ref{cor_minimax_onesided2}. We only highlight the main difference. By Theorem \ref{thm_minimax_twosided}, we can obtain that
$$
m\geq \rho \Big\{g_{TS}(d,A_d,\rho)-\frac{1}{A_d}\frac{ \PP_{\btheta}(\btheta\notin M)}{ \PP_{\btheta}(\btheta\in M)}\Big\}.
$$
To show (\ref{cor_minimax_twosided2_eq1}), denote $ a_1:=\sigma\sqrt{2\log (d/A_d-1)}$, and we can take $\rho=a_1$. The key step is to lower bound $g_{TS}(d,A_d,\rho)$. Recall that
$$
g_{TS}(d,A_d,\rho)=\frac{2(d-A_d)}{A_d}\Phi(-D)+\Big\{\Phi\Big(\frac{\rho}{\sigma}+D\Big)-\Phi\Big(\frac{\rho}{\sigma}-D\Big)\Big\}:=I_1+I_2,
$$
where 
$$
D=\frac{\sigma}{\rho}\cosh^{-1}(\frac{d-A_d}{A_d}\exp(\frac{\rho^2}{2\sigma^2})).
$$
We now consider the two terms $I_1, I_2$. By the definition of $\cosh$, we can easily verify the following inequality $\log y<\cosh^{-1}(y)<\log (2y)$ holds. Applying it to $I_1$ yields
\begin{align*}
I_1&\geq \frac{2(d-A_d)}{A_d}\Phi\Big(-\frac{\sigma}{\rho}\log \frac{2(d-A_d)}{A_d}-\frac{\rho}{2\sigma}\Big)\\
&=\frac{2(d-A_d)}{A_d}\Phi\Big(-\sqrt{2\log (d/A_d-1)}-\frac{\log 2}{\sqrt{2\log (d/A_d-1)}}\Big)\\
&\geq \frac{C}{\sqrt{2\log (d/A_d-1)}},
\end{align*}
for some universal positive constant $C$. Similarly, we can show that
\begin{align*}
I_2&\geq \Phi\Big(\frac{\sigma}{\rho}\log \frac{d-A_d}{A_d}+\frac{3\rho}{2\sigma}\Big)-\Phi\Big(-\frac{\sigma}{\rho}\log \frac{d-A_d}{A_d}+\frac{\rho}{2\sigma}\Big)\\
&=\Phi\Big(2\sqrt{2\log (d/A_d-1)}\Big)-\Phi(0)\\
&\geq 1/2-C\Big(\frac{A_d}{d-A_d}\Big)^4,
\end{align*}
The same argument in the proof of Corollary \ref{cor_minimax_onesided2} implies (\ref{cor_minimax_twosided2_eq1}). The proof of (\ref{cor_minimax_twosided2_eq3}) is similar. Finally, to show (\ref{cor_minimax_twosided2_eq2}), notice that 
$$
\EE_{\btheta}(\bar U_j^{TS}-\bar L_j^{TS})= \EE_{\btheta}(\bar U_j^{TS}-\bar L_j^{TS})I(j\in \bar S^{TS}_{\alpha'})\leq 2\sigma\bar u^{TS}_{\alpha'}.
$$
Regardless of the SNR, $\lim\sup \bar u^{TS}_{\alpha'}/\sqrt{2\log d}\leq 1$ always holds. Finally, for (\ref{cor_minimax_twosided2_eq4}), we have
\begin{align*}
a/\sigma&\geq \sqrt{2\log (d-s)-\log\log (d-s)+C}+\sqrt{2\log s-\log\log s+C}\geq \bar\phi,
\end{align*} 
and therefore 
$$
\bar u^{TS}_{\alpha'}=\sqrt{2\log \Big(\frac{4s}{(\alpha-\alpha')C_{2s,\alpha-\alpha'}}\Big)}, 
$$
which further implies $\lim\sup \bar u^{TS}_{\alpha'}/\sqrt{2\log s}\leq 1$. This completes the proof.

\begin{lemma}[Tail bound for Gaussian distribution]\label{lemtail}
Let $N\sim N(0,1)$. Then for any $y>0$
$$
\sqrt{\frac{2}{\pi}}\frac{\exp(-y^2/2)}{y+\sqrt{y^2+4}}<\PP(N>y)\leq \sqrt{\frac{2}{\pi}}\frac{\exp(-y^2/2)}{y+\sqrt{y^2+8/\pi}}.
$$
Conversely, for any $t>2$,
$$
\sqrt{(2\log t-\log\log t-C)_+}\leq \Phi^{-1}(1-\frac{1}{t})\leq \sqrt{2\log t-\log\log t},
$$
where $C=2\log 4+\log \pi$. 
\end{lemma}

\bibliographystyle{ims}
\bibliography{refer}

\end{document}